\documentclass[11pt, a4paper]{article}

\usepackage[backend=bibtex]{biblatex}
\usepackage{comment}

\usepackage{fullpage}

\usepackage{tikz}
\usetikzlibrary{automata,positioning}
\usetikzlibrary{arrows.meta}
\DeclareRobustCommand{\tlinf}{%
\begin{tikzpicture}
    \draw[|-|] (0,0) -- (0.4,0);
\end{tikzpicture}%
}

\usepackage{subcaption}

\usepackage{graphicx}

\usepackage{enumitem}
\usepackage{xcolor}

\usepackage{hyperref}

\usepackage{amsthm}
\usepackage{amsmath}
\usepackage{amsfonts}
\usepackage{amssymb}
\usepackage{euscript}
\usepackage{bm}

\newcommand{\sC}{\EuScript{C}}
\newcommand{\sD}{\EuScript{D}}

\newcommand{\sF}{\EuScript{F}}
\newcommand{\sG}{\EuScript{G}}
\newcommand{\sH}{\EuScript{H}}

\newcommand{\sJ}{\EuScript{J}}
\newcommand{\sT}{\EuScript{T}}

\newcommand{\sX}{\EuScript{X}}

\newcommand{\balpha}{\bm{\alpha}}
\newcommand{\bx}{\bm{x}}

\newcommand{\bbf}{\mathbf{f}}
\newcommand{\bbg}{\mathbf{g}}
\newcommand{\bbt}{\mathbf{t}}
\newcommand{\bA}{\mathbf{A}}
\newcommand{\bP}{\mathbf{P}}
\newcommand{\bK}{\mathbf{K}}
\newcommand{\bH}{\mathbf{H}}

\newcommand{\cS}{\mathcal{S}}
\newcommand{\cT}{\mathcal{T}}

\newcommand{\cL}{\mathcal{L}}
\newcommand{\cR}{\mathcal{R}}

\newcommand{\bzero}{\mathbf{0}}
\newcommand{\bI}{\mathbf{I}}

\newcommand{\bbE}{\mathbb{E}}
\newcommand{\bbP}{\mathbb{P}}
\newcommand{\bbR}{\mathbb{R}}
\newcommand{\bbZ}{\mathbb{Z}}

\DeclareMathOperator*{\argmin}{arg\,min}

\DeclareMathOperator{\tr}{tr}

\DeclareMathOperator{\Span}{Span}

\DeclareMathOperator{\Ext}{Ext}
\DeclareMathOperator{\glim}{\Gamma-lim}

\DeclareMathOperator{\opi}{\hat{\otimes}_{\pi}}

\theoremstyle{plain}
\newtheorem{definition}{Definition}[section]
\newtheorem{lemma}[definition]{Lemma}
\newtheorem{theorem}[definition]{Theorem}
\newtheorem{corollary}[definition]{Corollary}
\newtheorem{remark}[definition]{Remark}
\newtheorem{example}[definition]{Example}

\usepackage{multirow}
\usepackage{multicol}

\usepackage{caption} 
\definecolor{kartik}{rgb}{0.8,0, 0.2}
\newcommand{\kartik}[1]{\textcolor{kartik}{#1}}

\usepackage[nottoc,notlot,notlof]{tocbibind}

\title{\bfseries The Positive-Definite Completion Problem}
\author{Kartik G. Waghmare and Victor M. Panaretos}

\begin{document}
\maketitle
\begin{abstract}
    We study the positive-definite completion problem for kernels on a variety of domains and prove results concerning the existence, uniqueness, and characterization of solutions. In particular, we study a special solution called the canonical completion which is the reproducing kernel analogue of the determinant-maximizing completion known to exist for matrices. We establish several results concerning its existence and uniqueness, which include algebraic and variational characterizations. 	Notably, we prove the existence of a canonical completion for domains which are equivalent to the band containing the diagonal. This corresponds to the existence of a canonical extension in the context of the classical extension problem of positive-definite functions, which can be understood as the solution to an abstract Cauchy problem in a certain reproducing kernel Hilbert space.
\end{abstract}


\section{Introduction}
Let $X$ be a set and $\Omega \subset X \times X$. Given a function $K_{\Omega}: \Omega \to \bbR$, we consider the problem of extending $K_{\Omega}$ to $X \times X$ such that the resulting extension $K: X \times X \to \bbR$ is a reproducing kernel, which is to say $K(x, y) = K(y, x)$ for $x,y \in X$ and
\begin{equation*}
	\sum_{i,j=1}^{n} \alpha_{i} \alpha_{j} K(x_{i}, x_{j}) \geq 0
\end{equation*}
for every $n \geq 1$, $\{\alpha_{j}\}_{j=1}^{n} \subset \bbR$ and $\{x_{j}\}_{j=1}^{n} \subset X$. We shall refer to this as a \emph{completion problem} of $K_{\Omega}$ and the extensions $K$, which can be regarded as its solutions, shall be called \emph{completions}.

The problem has been studied before in the literature for certain special cases. For finite $X$, the problem can be understood as that specifying the unspecified entries of a partially specified matrix so as to make it positive semidefinite. In this form, the problem has been studied for the \emph{band} case, where $\Omega = \{(i,j): |i-j| \leq p\}$ for $X = \{j: 1 \leq j \leq n\}$ for some $n > 1$, by \textcite{dym1981}. They derived necessary and sufficient conditions on $K_{\Omega}$ for the existence of a completion $K$ and established the existence of a unique special completion which maximizes the determinant of the matrix $[K(i,j)]_{i,j \in X}$ among all completions $K$. In particular, they showed that this is the unique completion with the property that the $ij$th entry of the inverse of the matrix $[K(i,j)]_{i,j \in X}$ vanishes if $(i, j) \notin \Omega$. \textcite{grone1984} studied the problem for general $\Omega$ and proved the existence and uniqueness of this special completion, provided a completion exists. Necessary and sufficient conditions for the existence of a completion for general $\Omega$ were derived by \textcite{paulsen1989}. A complete characterization of completions for the band case was obtained by \textcite{gohberg1989} and the results were also extended to matrices of operators, which can be thought of as operator-valued kernels in our setting (see \textcite{bakonyi2011} and \textcite{paulsen2016}).


For infinite $X$, the completion problem has been studied mostly in the form of the extension problem for positive-definite \emph{functions}. In this setting, usually $X = \bbZ$ or $\bbR$ and one is concerned with extending a positive-definite function $F$ defined on $\{x \in X: |x| < a\}$, for some $a > 0$, to all of $X$. In our language, this means that $\Omega$ is the band $\{(x, y): |x - y| < a\} \subset X \times X$,  $K_{\Omega}(x, y) = F(x - y)$ is \emph{stationary} (translation invariant) and we only consider stationary completions.

For $X = \bbZ$, it was shown by \textcite{caratheodory1907}, that every positive-definite function $F$ on $\{x \in \bbZ: |x| < a\}$ for some integer $a > 0$, admits a positive-definite extension to $\bbZ$. The analogous result for $X = \bbR$ was proved by \textcite{krein1940} for continuous $F$, and later by \textcite{artjomenko1941} without the continuity assumption. Necessary and sufficient conditions for uniqueness of the extension were derived by \textcite{Keich1999}. Moreover, a description of all extensions was given by \textcite{kaltenbach1998}. A short historical survey of further developments can be found in \textcite{Sasvari2006}. An analogue of the special solution from the matrix case for $X = \bbZ$ arose in the work of \textcite{burg1975} concerning spectral estimation for stationary time series. However, no such analogue for $X = \bbR$ has been studied in the existing literature to the best of our knowledge.


In this article, we study the positive-definite completion problem in considerable generality, and in particular, without requiring stationarity of $K_{\Omega}$ or finiteness of $X$. Needless to say, this is a non-trivial problem because the methods used for proving the classical results discussed previously, such as matrix determinants and factorization or unitary representation, do not generalize in an obvious way to arbitrary reproducing kernels. The problem has not been studied in such a setting before in the existing literature (with the exception of \textcite{waghmare2021}, discussed below). Our approach relies on the the use of  tools from the theory of reproducing kernels, such as contraction maps and inner products of reproducing kernel Hilbert spaces (\textcite{paulsen2016}), and some results from the theory of tensor products of Hilbert spaces (\textcite{Ryan2002,Treves2016}), $\Gamma$-convergence (\textcite{Braides2002, DalMaso2012}) and strongly continuous one-parameter semigroups (\textcite{Davies1980,Engel2000}). 

Special cases of certain results presented here 
appeared in \textcite{waghmare2021} for the case of $X=[0,1]$ and $\Omega$ a \emph{serrated domain}, i.e. $\Omega=\cup_{j=1}^{k} I_j\times I_j$, for $\{I_j\}_{j=1}^{k}$ a finite interval cover of $[0,1]$. These include the existence/uniqueness of the canonical solution, and its characterization/construction. The present article represents a  mathematically more complete and general treatment of the subject, not restricted to the simpler case of serrated domains.

\subsection{Contributions}

We study the general characteristics of the set of completions and derive a surprisingly simple characterization of its extreme points in terms of their reproducing kernel Hilbert space.

For domains which are, in a certain sense, \emph{large} (see Figure \ref{fig:domains}), we show that positive-definite completion is equivalent to solving a linear equation in the projective tensor product space of certain reproducing kernel Hilbert spaces. As a consequence, we characterize the set of completions in terms of bounded extensions of a linear functional on the tensor product space. 

For the class of \emph{serrated} domains, we prove the existence of a unique canonical completion and given an iterative formula involving certain contraction maps for computing it. We derive a particularly simple closed form expression for the inner product of its reproducing kernel Hilbert space. Furthermore, we present several interesting variational characterizations of the canonical completion. Finally, we prove partial analogues of the determinant maximization and inverse zero characterizations. All of these results can be generalized to a more expansive class of domains we call junction-tree domains.  

For $X = \bbR$, we establish the existence of the analogue of the special completion from the matrix case, which we call the \emph{canonical completion}, for continuous $K_{\Omega}$ on domains $\Omega$ which are, in a sense, \emph{band-like}. Importantly, we prove the existence of a \emph{canonical extension} $F_{\star}$ to $\bbR$ of a positive-definite function $F$ on $(-a, a) \subset \bbR$ for some $a > 0$, thus demonstrating the existence of an analogue of the determinant-maximizing special completion from the matrix case for positive-definite functions on $\bbR$. The extension is shown to correspond to a certain strongly continuous semigroup on a reproducing kernel Hilbert space and consequently, can be thought of as the solution of an abstract Cauchy problem in that space. Under certain technical conditions, we also show the uniqueness of the canonical extension and recover the generator of its semigroup as the closure of a certain operator, which basically amounts to recovering the canonical extension.

\begin{figure}
\centering
\begin{subfigure}{0.32\textwidth}
\centering
\begin{tikzpicture}
	\draw (0, 0) rectangle (4,4);
	\draw[fill=red!15] (0, 0) rectangle (2.5,2.5);
	\draw[fill=red!15] (2.5,2.5) rectangle (4,4);
	\draw[fill=red!15, shift={(0.2,0)}] plot [smooth cycle, tension=2] coordinates {(2.8, 0.5) (3.6, 1) (2.8,1.5)};
	\begin{scope}[cm={0,1,1,0,(0,0)}]
		\draw[fill=red!15, shift={(0.2,0)}] plot [smooth cycle, tension=2] coordinates {(2.8, 0.5) (3.6, 1) (2.8,1.5)};
	\end{scope}
\end{tikzpicture}
\caption{A large domain.}
\end{subfigure}
\begin{subfigure}{0.32\textwidth}
\centering
\begin{tikzpicture}	
	\draw[fill=red!15] (0, 0) rectangle (1.5,1.5);
	\draw[fill=red!15] (0.8, 0.8) rectangle (2.5,2.5);
	\draw[fill=red!15] (2, 2) rectangle (3.5,3.5);
	\draw[fill=red!15] (3, 3) rectangle (4, 4);
	\pgfmathsetmacro{\h}{0.01}
	\fill[red!15] (0+\h, 0+\h) rectangle (1.5-\h,1.5-\h);
	\fill[red!15] (0.8+\h, 0.8+\h) rectangle (2.5-\h,2.5-\h);
	\fill[red!15] (2+\h, 2+\h) rectangle (3.5-\h,3.5-\h);
	\fill[red!15] (3+\h, 3+\h) rectangle (4-\h, 4-\h);
	\draw (0, 0) rectangle	(4,4);
\end{tikzpicture}
\caption{A serrated domain.}
\end{subfigure}
\begin{subfigure}{0.32\textwidth}
\centering
\begin{tikzpicture}
	\draw (0, 0) rectangle	(4,4);
	\draw[fill=red!15] (1.5, 0) .. controls +(0,1) and +(0,-1) .. (4,2.5) -- (4,4) -- (2.5,4) .. controls +(-1, 0) and +(1, 0) .. (0, 1.5) -- (0, 0);
\end{tikzpicture}
\caption{A regular domain.}
\end{subfigure}
\caption{Domains. The red region represents $\Omega$.}
\label{fig:domains}
\end{figure}
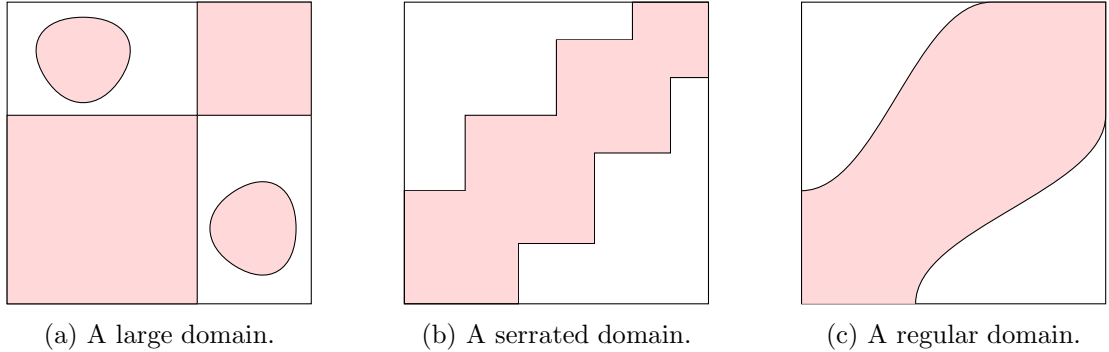

\subsection{Interpretations and Connections to other Problems}

Reproducing kernels are omnipresent in analysis and probability. In some contexts, they arise purely by virtue of being the essence of positive-definiteness, for example, as positive semidefinite matrices, inner products and Mercer kernels; while elsewhere they manifest for less obvious reasons, for example as characteristic functions of distributions.
In this section, we discuss how the completion problem relates to these other contexts.

\subsubsection{Fourier Transforms and Characteristic Functions}
Positive-definite functions occur naturally as Fourier transforms of finite positive Borel measures in probability and analysis and as characteristic functions of random variables in probability theory. Because the correspondence is precise, we can think of positive-definite extensions of a continuous positive-definite function $F$ on an interval $(-a, a)$ for $a > 0$, as corresponding to Borel measures $\mu$ on $\bbR$ which satisfy 
\begin{equation}\label{eqn:fourier}
	\int_{-\infty}^{\infty} e^{itx} ~d\mu(x) = F(t) 
	\quad \mbox{ for } t \in (-a, a).
\end{equation}
This can be regarded as a generalization of the Hamburger moment problem, since the moments of a measure are determined by the values of the Fourier transform around the origin. Krein's result implies the existence of a measure $\mu$ satisfying (\ref{eqn:fourier}). Our result concerning the existence of a canonical extension $\tilde{F}$ of $F$ points to the existence of a special solution of the above problem.

\subsubsection{Gaussian Processes and Graphical Models}
There is a well-known bijective correspondence between reproducing kernels and the covariances of Gaussian processes. The completions $K$ of $K_{\Omega}$ thus correspond to zero-mean Gaussian processes $Y = \{Y_{x}: x \in X\}$ satisfying
\begin{equation}\label{eqn:gaussian}
	\bbE[Y_{x}Y_{y}] = K_{\Omega}(x,y) \quad \mbox{ for } (x, y) \in \Omega.
\end{equation}
In finite dimensions, the differential entropy of a zero-mean Gaussian distribution is proportional the logarithm of the determinant of its covariance matrix. Therefore, for finite $X$, the canonical completion $K_{\star}$ corresponds to the Gaussian process $Y$ which maximizes differential entropy under the constraint (\ref{eqn:gaussian}). The canonical completion also has an interesting interpretation in terms of the probability density $p$ of $Y$ because the inverse of $[K_{\star}(i, j)]_{i,j \in X}$ being zero at the entries corresponding to $(i, j) \notin \Omega$ implies that products of the form $t_{i}t_{j}$ for $(i, j) \in \Omega$ do not appear in $p(\bbt)$ where $\bbt = (t_{j})_{j \in X}$.

The canonical completion can also be interpreted in this context for possibly infinite $X$. It corresponds to the Gaussian process satisfying (\ref{eqn:gaussian}) which is Markov with respect to $\Omega$ in the extended sense of the \emph{global Markov property}:
\begin{equation}\label{eqn:global-markov}
	\bbP[Y_{u} \in A, Y_{v} \in B| Y_{S}] = \bbP[Y_{u} \in A| Y_{S}]\bbP[Y_{v} \in B| Y_{S}]
\end{equation}
where $A, B \subset \bbR$ and $Y_{S} = \{Y_{s}: s \in S\}$. In other words, the random variables $Y_{u}$ and $Y_{v}$ for $u, v \in X$ separated by $S \subset X$ are conditionally independent given $Y_{S}$. This is analogous to how the future and the past are conditionally independent given the present for an ordinary Markov process. The global Markov property is of natural way of extending Markovianity to processes indexed by vertices of a graph instead of time. Alternatively, we can say that $K_{\star}$ is the covariance of the Gaussian graphical model $Y$ corresponding to the ``graph'' $\Omega$ satisfying (\ref{eqn:gaussian}).

\subsubsection{Constrained Embeddings into Hilbert Spaces}
Notice that for every completion $K$ of $K_{\Omega}$, we can write for the generators $k_{x} \in \sH(K)$ of $K$ given by $k_{x}(y) = K(x, y)$ for $x,y \in X$, that $\langle k_{x}, k_{y} \rangle = K_{\Omega}(x, y)$ for $(x, y) \in \Omega$. Every completion thus corresponds to an embedding $x \mapsto \varphi_{x}$ of $X$ into a Hilbert space $\sH$ satisfying the constraint that 
\begin{equation}\label{eqn:embedding}
	\langle \varphi_{x}, \varphi_{y} \rangle  = K_{\Omega}(x, y) \quad \mbox{ for } (x, y) \in \Omega.
\end{equation}
In fact, every such embedding into a Hilbert space $\sH$ will be equal, up to isometry, to an embedding of the form $x \mapsto k_{x}$ into the reproducing kernel Hilbert space $\sH(K)$ of some completion $K$ of $K_{\Omega}$. The set $\sC = \sC(K_{\Omega})$ can thus be regarded as the set of solutions to a constrained embedding problem (\ref{eqn:embedding}). 

A canonical solution to the completion problem naturally corresponds to a special solution to the embedding problem. In fact, the canonical solution $K_{\star}$ can be understood as corresponding to an embedding $x \mapsto k_{x}$ such that the vectors $k_{x}$ are, in a sense, \emph{maximally dispersed} in $\sH(K)$. When $X$ is finite, this is can be easily formalized by choosing the determinant of the matrix $[K(i, j)]_{i,j \in X}$ as the measure of dispersion, which is only natural given that the determinant is proportional to the ``volume'' of the simplex formed by the vectors $\{k_{j}: j \in X\}$ in $\sH(K)$. Moreover, it vanishes if $\{k_{j}: j \in X\}$ are linearly dependent. Furthermore, if $\Omega$ is the diagonal $\{(j, j): j \in X\}$, then by Hadamard's inequality, it follows that the determinant is maximized precisely when the vectors $\{k_{j}: j \in X\}$ are orthogonal to each other, which perfectly conforms with our intuitive understanding of dispersion.

For infinite $X$ and under certain conditions, we prove a local analogue of the determinant maximization principle, which essentially says that every nice perturbation of a canonical solution $K_{\star}$ tends to increase the determinant in an appropriate sense, thus justifying the interpretation of the canonical completion in terms of dispersion for infinite $X$.

\subsubsection{Metric Embeddings into Hilbert Spaces}
According to Schoenberg's embedding theorem, a metric space $(X, d)$, where $d: X \times X \to \bbR$ is a distance function on $X$, can be embedded into a Hilbert space if and only if $K_{t}(x, y) = e^{-td^{2}(x,y)}$ is a reproducing kernel for every $t > 0$. Naturally, one can think of a partially specified counterpart of the distance function $d: X \times X \to \bbR_{+}$ and this gives rise to the notion of a partially specified metric space $(X, d_{\Omega})$ where 
$d_{\Omega}: \Omega \to \bbR_{+}$ for some $\Omega \subset X \times X$ is a partially specified distance function. Many natural phenomena such as molecules can be regarded as partially specified metric spaces because the distances between two points are not always fixed. Every extension of $d_{\Omega}$ to $ X \times X$ which is a valid distance function can be thought of as a \emph{conformation} of the partially specified metric space $(X, d_{\Omega})$. The problem of determining whether $(X, d_{\Omega})$ admits a conformation that can be embedded into a Hilbert space is equivalent to that of determining whether there exists an extension $d$ of $d_{\Omega}$ to $X \times X$ such that $K_{t}(x, y) = e^{-td^{2}(x, y)}$ where $x,y \in X$ is a completion of $K_{t\Omega}(x, y) = e^{-td_{\Omega}^{2}(x, y)}$ where $(x, y) \in \Omega$ for every $t > 0$.

\subsection{Organization of the Article}

After discussing some preliminaries in Section \ref{sec:preliminaries}, we begin by treating the general properties of completions in Section \ref{sec:compl}. We then proceed by studying the completion problem while gradually increasing the extent of positive-definiteness imposed on $K_{\Omega}$. In Section \ref{sec:compl}, we impose no assumption on $K_{\Omega}$. In Section \ref{sec:completion-large} we assume that certain restrictions of $K_{\Omega}$ are reproducing kernels. Finally, from Section \ref{sec:canonical-completion} onwards, we deal exclusively with $K_{\Omega}$ for which every restriction to $A \times A \subset \Omega$ for $A \subset X$ is a reproducing kernel. Section \ref{sec:pdfunctions} is dedicated to the study of canonical extensions of positive-definite functions.

\section{Preliminaries and Notation}
\label{sec:preliminaries}

\subsection{Reproducing Kernels}

Let $X$ be a set. A \emph{reproducing kernel} $K$ on $X$ is defined as a function $K: X \times X \to \bbR$ satisfying $K(x, y) = K(y, x)$ for $x,y \in X$ and
\begin{equation*}
	\sum_{i,j=1}^{n} \alpha_{i} \alpha_{j} K(x_{i}, x_{j}) \geq 0
\end{equation*}
for every $n \geq 1$, $\{\alpha_{j}\}_{j=1}^{n} \subset \bbR$ and $\{x_{j}\}_{j=1}^{n} \subset X$. The functions $k_{x}: X \to \bbR$ given by $k_{x}(y) = K(x, y)$ for $x, y \in X$ are called the \emph{generators} of $K$. The closure of the linear span of the generators under the norm induced by the inner product $\langle k_{x}, k_{y}\rangle = K(x, y)$ for $x, y \in X$ is called the reproducing kernel Hilbert space or associated Hilbert space of $K$ and denoted by $\sH(K)$ and associated with the inner product $\langle \cdot, \cdot \rangle_{\sH(K)}$ and the induced norm $\|\cdot\|_{\sH(K)}$. To avoid cluttering our notation, we shall always omit the subscript and denote the inner product as $\langle \cdot, \cdot \rangle$ and the norm as $\|\cdot\|$, except in cases where there is a possiblity of confusion. Note that for a function $f$ and kernel $K$ on $S$, $f \in \sH(K)$ with $\|f\| \leq C$ if and only if for some $C > 0$,
\begin{equation}\label{eqn:rkhs-membership}
	\left|\sum_{i=1}^{m} \alpha_{i}f(x_{i})\right| \leq C\sqrt{\sum_{i,j = 1}^{m} \alpha_{i}\alpha_{j}K(x_{i}, x_{j})}
\end{equation}
for every $m \geq 1$, $\{x_{i}\}_{i=1}^{m} \subset S$ and $\{\alpha_{i}\}_{i=1}^{m} \subset \bbR$.

For $A \subset X$ and $x \in X$, we define $k_{x, A}: A \to \bbR$ as $k_{x, A}(y) = K(x, y)$ for $y \in A$. Furthermore, for $A \subset X$, we can define the \emph{subkernel} $K_{A}: A \times A \to \bbR$ as the restriction $K_{A} = K|_{A \times A}$. Naturally, $K_{A}$ is also a reproducing kernel. Its associated Hilbert space is given by $\sH(K_{A}) = \{f|_{A}: f \in \sH(K)\}$ and the functions $k_{x, A}$ for $x \in A$ are its generators. Using (\ref{eqn:rkhs-membership}), one can show that the restriction $\sJ_{A}: \sH(K) \to \sH(K_{A})$ given by $f \mapsto f|_{A}$ is a bounded linear map satisfying $\|f|_{A}\| \leq \|f\|$ for $f \in \sH(K)$, where $\|f|_{A}\|$ is understood as the norm of $f|_{A}$ in $\sH(K_{A})$. Its adjoint $\sJ_{A}^{\ast}: \sH(K_{A}) \to \sH(K)$ is given by $\sJ_{A}^{\ast}g(x) = \langle \sJ_{A}^{\ast}g, k_{x}\rangle = \langle g, \sJ_{A}k_{x}\rangle = \langle g, k_{x, A} \rangle$ which is equal to $g(x)$ for $x \in A$. Notice that $\sJ_{A}^{\ast}k_{x, A} = k_{x}$ for $x \in A$. In fact, the associated Hilbert space $\sH(K_{A})$ is isometrically isomorphic to the closed linear subspace in $\sH(K)$ spanned by $\{k_{x}: x \in A\}$ under the inner product induced by the ambient space and the isometry is given by $\sJ_{A}$. This result is known as \emph{subspace isometry}. A direct consequence of this result is that
\begin{equation}\label{eqn:subspace-isometry}
	\| \Pi_{A}f\| = \|f|_{A}\|.
\end{equation}  
for $f \in \sH(K)$ where $\Pi_{A}$ is the projection to the closed linear subspace spanned by $\{k_{x}: x \in A\}$. Similarly, the orthogonal complement of a subspace is also isomorphic to a certain reproducing kernel Hilbert space. For $B \subset X$, we can define the \emph{Schur complement} $K/K_{B}: (X \setminus B) \times (X \setminus B) \to \bbR$ as $K/K_{B}(x, y) = K(x, y) - \langle k_{x, B}, k_{y, B} \rangle$. $K/K_{B}$ is a reproducing kernel because, by subspace isometry, we can write $K(x, y) - \langle k_{x, B}, k_{y, B} \rangle = \langle k_{x}, k_{y} \rangle - \langle \Pi_{B}k_{x}, \Pi_{B}k_{y} \rangle = \langle (k_{x} - \Pi_{B}k_{x}), (k_{y} - \Pi_{B}k_{y}) \rangle$. It can be shown that $\sH(K/K_{B}) = \{f \in \sH(K): f|_{B} = 0\}$ and that it is isometrically isomorphic to the orthogonal complement of $\Pi_{B}\sH(K)$. Furthermore,  
\begin{equation}\label{eqn:schur-isometry}
	\|f - \Pi_{B}f\| = \|g\|_{\sH(K/K_{B})}
\end{equation}  
where $g = (f - \Pi_{B}f)|_{X \setminus B}$ or equivalently, $g(y) = f(y) - \langle f|_{B}, k_{y, B}\rangle$ for $y \in X \setminus B$.

\subsection{Graphs}

Experience with the positive-definite completions of partially specified matrices and their connection to Gaussian graphical models suggests that there is great utility to thinking of a domain $\Omega$ as an \emph{undirected graph} $(X, \Omega)$ on the set of vertices $X$, with the vertices $x, y \in X$ being \emph{adjacent} iff $(x, y) \in \Omega$. The pair $(x, y) \in \Omega$ can thus be thought of as the \emph{edge} between $x$ and $y$, which makes $\Omega$ the edge set. Since $X$ will almost always be fixed, we shall often omit writing $(X, \Omega)$ and simply identify the graph $(X, \Omega)$ with its edge set $\Omega$. Notice that for a set $S \subset X$ such that $S \times S \subset \Omega$, every $x, y \in S$ are adjacent. We call such sets \emph{cliques}. For $x,y \in X$, a \emph{path} in $\Omega$ between $x$ and $y$ is a finite sequence $\{z_{k}\}_{k=0}^{n+1} \subset X$ such that $z_{0} = x$, $z_{n+1} = y$ and $(z_{k}, z_{k+1}) \in \Omega$ for $0 \leq k \leq n$. We say that $x,y \in X$ are \emph{connected} in $\Omega$ if there is a path in $\Omega$ between them and \emph{disconnected} otherwise. We say $S \subset X$ is a \emph{separating set} or a \emph{separator} of $\Omega$, if there exist $x, y \in X \setminus S$ such that for every path $\{z_{k}\}_{k=0}^{n+1} \subset X$ between $x$ and $y$, $z_{k} \in S$ for some $1 \leq k \leq n$, or in other words, every path between $x$ and $y$ passes through $S$. Alternatively, $S \subset X$ is a separator if $X \setminus S$ is disconnected. We adopt the convention that, if $x$ and $y$ are disconnected, then they are separated by the empty set $\varnothing$. 

\subsection{Domains and Completions}
A \emph{domain} $\Omega$ on $X$ is a subset of $X \times X$ which is symmetric in that $(x, y) \in \Omega$ if and only if $(y, x) \in \Omega$ and contains the diagonal $\{(x, x): x \in X\} \subset X \times X$. If $K_{\Omega}: \Omega \to \bbR$ is a function, an extension $K: X \times X \to \bbR$ of $K_{\Omega}$ which is a reproducing kernel on $X$ is called a \emph{positive-definite completion} or simply, a \emph{completion} of $K_{\Omega}$.

\begin{definition}[Completion]
	Let $K_{\Omega}$ be a function on a domain $\Omega$ on $X$. A reproducing kernel $K$ on $X$ is called a \emph{completion} of $K_{\Omega}$ if the restriction of $K$ to $\Omega$ is $K_{\Omega}$.
\end{definition}

We shall denote the set of completions of a function $K_{\Omega}$ by $\sC(K_{\Omega})$ or simply $\sC$. The symmetry of the domain $\Omega$ merely accounts for the fact that the completions are themselves symmetric by definition, while containing the diagonal ensures that the set of completions $\sC$ is bounded (Theorem \ref{thm:sC-cnvx-cmpct}). 


Of course, not every such function $K_{\Omega}$ admits a completion. A necessary condition is that suitable restrictions of $K_{\Omega}$ be reproducing kernels. A function $K_{\Omega}: \Omega \to \bbR$ on a domain $\Omega \subset X \times X$ is called a \emph{partially reproducing kernel} if for every $A \subset X$ such that $A \times A \subset \Omega$, the restriction $K_{A} = K_{\Omega}|_{A \times A}$ is a reproducing kernel on $A$. Naturally, every reproducing kernel $K$ on $X$ is a partially reproducing kernel on the domain $\Omega = X \times X$. We extend the definition of $k_{x, A}$ for partially reproducing kernels $K_{\Omega}$ by defining them for $x \in X$ and $A \subset X$ such that $x \times A \subset \Omega$ as the functions $k_{x, A}: A \to \bbR$ given by $k_{x, A}(y) = K(x, y)$ for $y \in A$.

If $X$ is a subset of $\bbR$ or $\bbZ$, we call a partially reproducing kernel $K_{\Omega}$ \emph{stationary} if for some $F: X \to \bbR$ we have $K_{\Omega}(x, y) = F(x - y)$ for $(x,y) \in \Omega$. Note that this includes reproducing kernels $K$ on $X$ as they can be considered as partially reproducing kernels with $\Omega = X \times X$.

\subsection{Projective Tensor Product}
Consider two Hilbert spaces $\sH_{1}$ and $\sH_{2}$ and their tensor product  
\begin{equation*}
	\sH_{1} \otimes \sH_{2} = \Span \{f \otimes g : f \in \sH_{1} \mbox{ and } g \in \sH_{2}\}.
\end{equation*}
We define the \textit{projective tensor product norm} or more simply, the $\pi$-norm $\|\cdot\|_{\pi}$ on $\sH_{1} \otimes \sH_{2}$, as
\begin{equation*}
	\|\tau\|_{\pi} = \inf \Bigg\{\sum_{i=1}^{\infty}\|f_{i}\|\|g_{i}\|: \tau = \sum_{i=1}^{n}f_{i} \otimes g_{i} \mbox{ where } n \geq 1, f_{i} \in \sH_{1}, g_{i} \in \sH_{2} \mbox{ for } i \geq 1\Bigg\}.
\end{equation*}
The completion of $\sH_{1} \otimes \sH_{2}$ under $\|\cdot\|_{\pi}$ is a Banach space called the projective tensor product space of $\sH_{1}$ and $\sH_{2}$, and denoted by $\sH_{1} \opi \sH_{2}$. It turns out that the dual of the projective tensor product space of two Hilbert spaces is isometrically isomorphic to the space of bounded linear operators between them (\textcite[Chapter 2.2]{Ryan2002} \textcite[Proposition 43.8]{Treves2016}). 
In other words,
\begin{equation}\label{eqn:iso_pi}
	\left[\sH_{1} \opi \sH_{2}\right]^{\ast} = \cL(\sH_{1}, \sH_{2})
\end{equation}
and we can think of every $\Phi \in \cL(\sH_{1}, \sH_{2})$ as a bounded linear functional $\Phi$ on $\sH_{1} \opi \sH_{2}$ in the following sense: $\Phi[f \otimes g] = \langle\Phi f, g\rangle$
for $f \in \sH_{1}$ and $g \in \sH_{2}$. The expression $\Phi[\tau]$ is well-defined for every $\tau \in \sH_{1} \opi \sH_{2}$ as a result of continuous extension. This result provides an alternative expression for the $\pi$-norm which we shall call the \emph{duality formula} given by
\begin{equation}\label{eqn:duality_formula}
	\|\tau\|_{\pi} = \sup\{ |\Phi[\tau]| : \Phi \in \cL(\sH_{1}, \sH_{2}), \|\Phi\| \leq 1\}.
\end{equation}
The space $\sH_{1} \opi \sH_{2}$ can also be thought of as the space $\cL_{1}(\sH_{1}, \sH_{2})$ of nuclear operators from $\sH_{1}$ to $\sH_{2}$ and of course, vice-versa \cite[Proposition 47.2]{Treves2016}.

\section{General Properties of Completions}\label{sec:compl}

In this section, we study some of the general properties of the set of completions $\sC$ such as convexity and compactness, and their consequence. 
\begin{theorem}\label{thm:sC-cnvx-cmpct}
	The set of completions $\sC(K_{\Omega})$ is convex and compact in the topology of pointwise convergence.
\end{theorem}
\begin{proof}
	Let $\sC_{0} = \{K: |K(x,y)| \leq \sqrt{K_{\Omega}(x,x)K_{\Omega}(y,y)} \mbox{ for } x, y \in X\}$.	For $K \in \sC_{0}$, the range of $K(x,y)$ is compact by the Heine-Borel theorem for every $x,y \in X$. By Tychonoff's theorem, $\sC_{0}$ is itself compact in the product topology, which is same as the topology of pointwise convergence.
	
	Notice that $K \in \sC$ if and only if $K \in \sC_{0}$, $K(x, y) = K_{\Omega}(x, y)$ for $(x, y) \in \Omega$, $K(x, y) - K(y, x) = 0$ for $x, y \in X$, and
	\begin{equation*}
		\sum_{i,j=1}^{n} \alpha_{i}\alpha_{j}K(x_{i}, x_{j}) \geq 0
	\end{equation*}
	for every $n \geq 1$, $\{x_{i}\}_{i=1}^{n} \subset X$ and $\{\alpha_{i}\}_{i=1}^{n} \subset \bbR$. Because the expressions $K \mapsto K(x, y)$, $K \mapsto K(x, y) - K(y, x)$ and $K \mapsto \sum_{i,j=1}^{n} \alpha_{i}\alpha_{j}K(x_{i}, x_{j})$ are continuous linear functionals under the topology of pointwise convergence, $\sC$ is a closed subset of $\sC_{0}$ implying that it is compact.
\end{proof}

Note that $\sC$ under the topology of pointwise convergence is not, in general, second-countable, and therefore, compactness does not necessarily imply sequential compactness. In Section \ref{sec:beyond-large-domains}, we shall show that $\sC$ is also sequentially compact under an additional assumption on $\Omega$.

\subsection{Convexity}

The set of completions $\sC$ is a compact convex subset of the space of real-valued functions on $X \times X$ which forms a Hausdorff, locally convex topological vector space under the topology of pointwise convergence. 
By the Krein-Milman theorem, $\sC$ is equal to the closed convex hull of $\Ext(\sC)$, where $\Ext(\sC)$ denotes the set of extreme points of $\sC$. 
A completion $K \in \sC$ is an extreme point of $\sC$ if it can not be represented as a proper linear combination of other completions. 
In other words, there do not exist completions $K_{1}, K_{2} \in \sC$ such that $K = \alpha K_{1} + (1 - \alpha)K_{2}$ for $0 < \alpha < 1$. The following result gives remarkably simple characterization of the set of extreme points of $\sC$ in terms of their reproducing kernel Hilbert spaces.

\begin{theorem}\label{thm:sC-extrms}
    $K \in \Ext(\sC)$ if and only if for every self-adjoint $\Psi:\sH(K) \to \sH(K)$, 
	\begin{equation*}
		\langle k_{x}, \Psi k_{y} \rangle = 0 \mbox{ for } (x,y) \in \Omega \quad \implies \quad \Psi = \bzero,
	\end{equation*}
	where $k_{x} \in \sH(K)$ is given by $k_{x}(y) = K(x, y)$ for $x,y \in X$.
\end{theorem}
	The above result is a direct consequence of the following lemma, the proof of which can be found in the appendix. 
\begin{lemma}\label{thm:pm-adjoint}
	Let $K$ be a reproducing kernel on $X$ with the associated Hilbert space $\sH$. There is a bijective correspondence between $H: X \times X \to \bbR$ such that $K+H, K-H \geq O$ and self-adjoint contractions $\Psi \in \mathcal{L}(\sH)$ given by $H(x,y) = \langle \Psi k_{x}, k_{y} \rangle$ for $x,y \in X$, where $k_{x} \in \sH(K)$ is given by $k_{x}(y) = K(x, y)$ for $x,y \in X$.
\end{lemma}


\subsection{Compactness}

An important consequence of compactness in the topology of pointwise convergence is that a completion problem admits a solution if and only if so does every finite subproblem. Let $K_{\Omega \sF}$ denote the restriction of $K_{\Omega}$ to the set $\Omega \cap (\sF \times \sF)$.
\begin{theorem}\label{thm:sC-exist-finite}
	$\sC(K_{\Omega})$ is nonempty if and only if so is $\sC(K_{\Omega \sF})$ for every finite $\sF \subset X$.
\end{theorem}
\begin{proof}
	Let $a$ be a finite subset of $X$ and $K_{a}$ denote a completion of $K_{\Omega a}$. Define 
	\begin{equation*}
		K^{a} = \begin{cases}
			K_{a}(x,y) & x,y \in a \\
			0 & \mbox{otherwise.} 
		\end{cases}
	\end{equation*}
	The mapping $a \mapsto K^{a}$ forms a net on the directed set $A = \{ a \subset X : a \mbox{ is finite} \}$ ordered by inclusion. By compactness of $\sC_{0}$, $K^{a}$ has a convergent subnet, say $K^{b} = K^{b(a)}$ which converges to some $K \in \sC_{0}$. It turns out that $K \in \sC$. Indeed, $K(x, y) = \lim_{b} K^{b}(x, y) = K_{\Omega}(x,y)$ for $(x, y) \in \Omega$, $K(x, y) - K(y, x) = \lim_{b} \left[ K^{b}(x, y) - K^{b}(y, x) \right] = 0$ for $x, y \in X$ and 
	\begin{equation*}
		\sum_{i,j=1}^{n}\alpha_{i}\alpha_{j}K(x_{i}, x_{j}) = \lim_{b} \left[ \sum_{i,j=1}^{n}\alpha_{i}\alpha_{j}K^{b}(x_{i}, x_{j}) \right] \geq 0
	\end{equation*}
	for $n \geq 1$, $\{x_{i}\}_{i=1}^{n} \subset X$ and $\{\alpha_{i}\}_{i=1}^{n} \subset \bbR$. The converse is trivial because if $K$ is a completion of $K_{\Omega}$, then $K|_{\sF \times \sF}$ is a completion of $K_{\Omega \sF}$.
\end{proof}

For finite $X$, we have the following result of Paulsen which gives necessary and sufficient conditions for the existence of a completion, in the language of matrices.
\begin{theorem}[{\parencite[Theorem 2.1]{paulsen1989}}]
	Let $J \subset \{1, \dots, n\}^{2}$ for some $n \geq 1$ such that $(j, j) \in J$ for $1 \leq j \leq n$ and $(i, j) \in J$ if $(j, i) \in J$. 
	A partially specified matrix $T = [t_{ij}]_{(i,j) \in J}$ which is symmetric (i.e. $t_{ij} = t_{ji}$ for $(i, j) \in J$) admits a completion if and only if for every positive semidefinite matrix $M = [m_{ij}]_{i,j = 1}^{n}$ such that $m_{ij} = 0$ for $(i, j) \notin J$ we have
	\begin{equation*}
		\sum_{(i, j) \in J} m_{ij}t_{ij} \geq 0.
	\end{equation*} 
\end{theorem}

The above result gives a concrete but somewhat unwieldy criterion for determining whether $\sC$ is nonempty. We shall say that $K_{\Omega}$ is \emph{symmetric}, if $K_{\Omega}(x, y) = K_{\Omega}(y, x)$ for $(x, y) \in \Omega$.

\begin{corollary}\label{thm:sC-exist-paulsen}
	Assume that $K_{\Omega}$ is symmetric. $\sC(K_{\Omega})$ is nonempty if and only if for every finite $F = \{x_{i}\}_{i=1}^{n} \subset X$ and positive semidefinite matrix $M = [m_{ij}]_{i,j = 1}^{n}$ such that $m_{ij} = 0$ for $(x_{i}, x_{j}) \notin \Omega$ we have
	\begin{equation*}
		\sum_{(x_{i}, x_{j}) \in \Omega} m_{ij} K_{\Omega}(x_{i}, x_{j}) \geq 0.
	\end{equation*}
\end{corollary}

A criterion of this form can be easily used to work out maximum and minimum values that a completion can have at a given point. Define $m, M: X \times X \to \bbR$ as
\begin{equation*}
	M(x, y) = \sup \{K(x, y): K \in \sC(K_{\Omega})\} \mbox{ and }
	m(x, y) = \inf \{K(x, y): K \in \sC(K_{\Omega})\}.
\end{equation*}
We fix $K(x, y) = c$ for some $c \in \bbR$ and formulate a new completion problem on the domain $\Omega \cup \{(x, y), (y, x)\}$ for a new function equal to $K_{\Omega}$ on $\Omega$ and $c$ on $\{(x, y), (y, x)\}$. By Corollary \ref{thm:sC-exist-paulsen}, the function admits a completion if and only if for every finite $F = \{x, y\} \cup \{x_{k}\}_{k=1}^{n} \subset X$ and positive semidefinite matrix $M = [m_{ij}]$ where $i,j \in \{x, y\} \cup \{k\}_{k=1}^{n}$ such that $m_{ij}$, $m_{xj}$ and $m_{iy}$ are zero when $(x_{i}, x_{j})$, $(x, x_{j})$ and $(x_{i}, y)$ is not in $\Omega$,  respectively, we have
\begin{equation*}
	2m_{xy}c + 2\sum_{(x, x_{j}) \in \Omega} m_{xj}K_{\Omega}(x, x_{j}) + 2\sum_{(x_{i}, y) \in \Omega} m_{iy}K_{\Omega}(x_{i}, y) + \sum_{(x_{i}, x_{j}) \in \Omega} m_{ij}K_{\Omega}(x_{i}, x_{j}) \geq 0.
\end{equation*}
Define for the pair $(M, F)$ where $M$ and $F$ are as described above,
\begin{equation*}
	\cR_{xy}(M, F) = \frac{-1}{m_{xy}}\Bigg[\sum_{(x, x_{j}) \in \Omega} m_{xj}K_{\Omega}(x, x_{j}) + \sum_{(x_{i}, y) \in \Omega} m_{iy}K_{\Omega}(x_{i}, y) + \frac{1}{2}\sum_{(x_{i}, x_{j}) \in \Omega} m_{ij}K_{\Omega}(x_{i}, x_{j}) \Bigg].
\end{equation*}
By working out the values of $c$ for which the above statement is true the following result becomes apparent.
\begin{theorem}
	Assume that $K_{\Omega}$ is symmetric and $\sC$ is nonempty. We have 
	\begin{equation*}
		M(x, y) = \inf_{m_{xy} < 0} \cR_{xy}(M, F) \mbox{ and } m(x, y) = \sup_{m_{xy} > 0} \cR_{xy}(M, F).
	\end{equation*}
\end{theorem}

Notice that the value of a completion $K$ at a point $(x, y) \in \Omega^{c}$ is uniquely determined if and only if $m(x, y) = M(x, y)$. Using this observation, it is not difficult to see why the following result holds.

\begin{theorem}\label{thm:sC-single}
	Assume that $K_{\Omega}$ is symmetric. $\sC$ is a singleton if and only if for every $(x, y) \in \Omega^{c}$ and $\epsilon > 0$ there exist pairs $(M, F)$ and $(M', F')$ where $M' = [m_{ij}']$ such that $m_{xy} < 0$, $m_{xy}' > 0$ and
	\begin{equation*}
		\cR_{xy}(M, F) - \cR_{xy}(M', F') < \epsilon.
	\end{equation*}
\end{theorem}

\section{Completion on Large Domains}
\label{sec:completion-large}

We say that a domain $\Omega$ is \emph{large} if there exist $X_{1}, X_{2} \subset X$ such that $X = X_{1} \cup X_{2}$ and $X_{1} \times X_{1}, X_{2} \cup X_{2} \subset \Omega$. Let $\Delta = \Omega \cap (X_{2} \times X_{1})$ and $\Delta^{\ast} = \Omega \cap (X_{1} \times X_{2})$ (see Figure \ref{fig:large-domain}). We shall assume throughout this section that the restrictions $K_{X_{1}} = K_{\Omega}|_{X_{1} \times X_{1}}$ and $K_{X_{2}} = K_{\Omega}|_{X_{2} \times X_{2}}$ are reproducing kernels. 

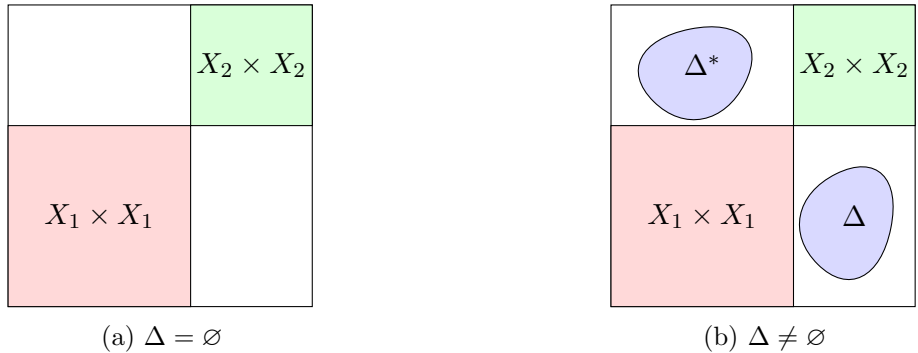
\begin{figure}[htbp]
	\centering
	\begin{subfigure}{0.49\textwidth}
		\centering
		\begin{tikzpicture}[scale=0.8]
			\pgfmathsetmacro{\a}{5}
			\pgfmathsetmacro{\b}{3}
			\pgfmathsetmacro{\c}{2}
	
			\draw (0, 0) rectangle (\a, \a);
			\draw[fill=red!15] (0, 0) rectangle (\b, \b);
			\node at (\b/2, \b/2) {$X_{1} \times X_{1}$};
	
			\draw[fill=green!15] (\a-\c, \a-\c) rectangle (\a, \a);
			\pgfmathsetmacro{\d}{\a-0.5*\c}
			\node at (\d, \d) {$X_{2} \times X_{2}$};
		\end{tikzpicture}
	\caption{$\Delta  = \varnothing$}
	\label{fig:large-domain-nodelta}
	\end{subfigure}
	\begin{subfigure}{0.49\textwidth}
		\centering
		\begin{tikzpicture}[scale=0.8]
			\pgfmathsetmacro{\a}{5}
			\pgfmathsetmacro{\b}{3}
			\pgfmathsetmacro{\c}{2}
	
			\draw (0, 0) rectangle (\a, \a);
			\draw[fill=red!15] (0, 0) rectangle (\b, \b);
			\node at (\b/2, \b/2) {$X_{1} \times X_{1}$};
	
			\draw[fill=green!15] (\a-\c, \a-\c) rectangle (\a, \a);
			\pgfmathsetmacro{\d}{\a-0.5*\c}
			\node at (\d, \d) {$X_{2} \times X_{2}$};
	
			\draw[fill=blue!15] plot [smooth cycle, tension = 2] coordinates {(3.4,0.7) (4.6,1.3) (3.7,2.2)};
			\node at (4, 1.5) {$\Delta$};
	
			\begin{scope}[cm={0,1,1,0,(0,0)}]
				\draw[fill=blue!15] plot [smooth cycle, tension = 2] coordinates {(3.4,0.7) (4.6,1.3) (3.7,2.2)};
				\node at (4, 1.5) {$\Delta^{\ast}$};
			\end{scope}
		\end{tikzpicture}
	\caption{$\Delta \neq \varnothing$}
	\label{fig:large-domain-delta}
	\end{subfigure}
	\caption{Large domain. The colored regions represent $\Omega$.}
	\label{fig:large-domain}
\end{figure}


\subsection{Contractions and Completions}
We begin by considering the special case where $\Delta$ is empty (see Figure \ref{fig:large-domain-nodelta}).  For every $U \subset X$ such that $U\times U \subset \Omega$ and $u\in U$, we denote $k_{u, U}: U \to \bbR$ as $k_{u, U}(x) = K_{\Omega}(x, u)$.

\begin{theorem}[Contraction Characterization]\label{thm:2sq-contraction}
	Let $K_{\Omega}$ be a partially reproducing kernel on a domain $\Omega = (X_{1} \times X_{1}) \cup (X_{2} \times X_{2})$ where $X_{1}, X_{2} \subset X$ (see Figure \ref{fig:2-serrated-domain}). There is a bijective correspondence between the completions $K$ of $K_{\Omega}$ and contractions $\Phi: \sH(K_{X_{1}}) \to \sH(K_{X_{2}})$ satisfying $\Phi k_{x, X_{1}} = k_{x, X_{2}}$ for $x \in X_{1} \cap X_{2}$ given by
	\begin{equation}\label{eqn:contraction-characterization}
		K(x, y) = \langle \Phi  k_{x, X_{1}}, k_{y, X_{2}} \rangle \quad\mbox{ for } x \in X_{1} \mbox{ and } y \in X_{2}.
	\end{equation} 
	If $X_{1} \cap X_{2} = \varnothing$, then there is a bijective correspondence between the completions $K$ of $K_{\Omega}$ and the contractions $\Phi: \sH(K_{X_{1}}) \to \sH(K_{X_{2}})$.
\end{theorem}

\begin{figure}[htbp]
	\centering
	\begin{tikzpicture}[scale = 0.6]
		\pgfmathsetmacro{\a}{9}
		\pgfmathsetmacro{\b}{6}
		\pgfmathsetmacro{\c}{5}				

		\pgfmathsetmacro{\x}{0.9}
		\pgfmathsetmacro{\y}{8}	
		\pgfmathsetmacro{\h}{0.05}

		\draw (0, 0) rectangle (\a, \a);
		\draw[fill=red!15] (0, 0) rectangle (\b, \b);
		\node at (\b/2, \b/2) {$K_{X_{1}}$};
		\draw[|-|] (0+\h, \x) -- (\b-\h, \x) node [midway, fill = red!15, text opacity=1, opacity = 1] {$k_{x, X_{1}}$}; 

		\draw[fill=green!15] (\a-\c, \a-\c) rectangle (\a, \a);
		\pgfmathsetmacro{\d}{\a-0.5*\c}
		\node at (\d, \d) {$K_{X_{2}}$};
		\draw[|-|] (\y, \a-\c+\h) -- (\y, \a-\h) node [midway, fill = green!15, sloped, text opacity=1, opacity = 1] {$k_{y, X_{2}}$};
		\begin{scope}[cm={0,1,1,0,(0,0)}]
			\draw[|-|] (0+\h, \x) -- (\b-\h, \x) node [midway, fill = red!15, rotate=90, text opacity=1, opacity = 1] {$k_{x, X_{1}}$};
			\draw[|-|] (\y, \a-\c+\h) -- (\y, \a-\h) node [midway, fill = green!15, text opacity=1, opacity = 1] {$k_{y, X_{2}}$};
		\end{scope}	
		
		\draw[fill=blue!15] (\a-\c, \a-\c) rectangle (\b, \b);
		\pgfmathsetmacro{\z}{4.7}
		\draw[|-|] (0+\h, \z) -- (\b-\h, \z) node [midway, fill = red!15, text opacity=1, opacity = 1] {$k_{z, X_{1}}$}; 
		\draw[|-|] (\a-\c+\h, \z) -- (\a-\h, \z) node [midway, fill = green!15, text opacity=1, opacity = 1, shift={(0.3,0)}] {$k_{z, X_{2}}$};

		\draw [fill] (\y, \x) circle [radius=0.05] node [above] {$(x, y)$};			
	\end{tikzpicture}
	\caption{The coloured region represents $\Omega$ with the kernels $K_{X_{1}}$ and $K_{X_{2}}$ being represented by the red and blue regions and blue and green regions, respectively. The functions $k_{u, U}$ are being represented by $\tlinf$ at the position corresponding to their values relative to the kernel.}	
	\label{fig:2-serrated-domain}
\end{figure}
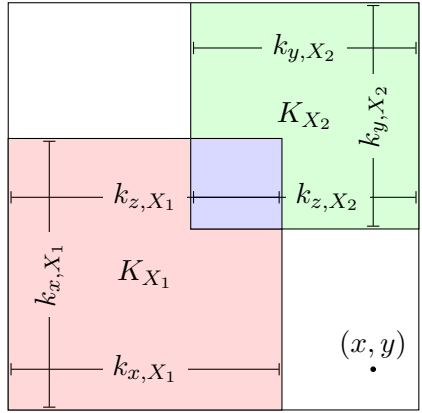

\begin{proof}[Proof of Theorem \ref{thm:2sq-contraction}]
	Let $K$ be a completion of $K_{\Omega}$. Define $\Phi_{0}: \Span \{ k_{x, X_{1}}: x \in X_{1} \} \to \sH(K_{X_{2}})$ as $\Phi_{0} k_{x, X_{1}} = k_{x, X_{2}}$. For $m , n \geq 1$, let $\{x_{i}\}_{i=1}^{m} \subset X_{1}$, $\{y_{k}\}_{k=1}^{n} \subset X_{2}$ and $\{\alpha_{i}\}_{i=1}^{m},\{\beta_{k}\}_{k=1}^{n} \subset \bbR$. By positive definiteness, the expression
	\begin{equation}\label{eqn:quad_form}
		\sum_{i, j = 1}^{m} \alpha_{i}\alpha_{j}K(x_{i}, x_{j}) + 2\sum_{i, k = 1}^{m, n} \alpha_{i}\beta_{k}K(x_{i}, y_{k}) + \sum_{k, l = 1}^{n} \beta_{k}\beta_{l}K(y_{k}, y_{l}) \geq 0
	\end{equation}
	is non-negative. This can be rewritten in terms of $\Phi_{0}$ and $f = \sum_{i=1}^{m}\alpha_{i}k_{x_{i}, X_{1}}$ and $g = \sum_{k=1}^{n}\beta_{k}k_{y_{k}, X_{2}}$ as follows
	\begin{equation*}
		\langle f, f \rangle + 2 \langle \Phi_{0} f, g \rangle + \langle g, g \rangle \geq 0
	\end{equation*}
	Replacing $g$ with $-g$ and using continuity of the inner product, we get for every $g \in \sH(K_{X_{2}})$,
	\begin{equation*}
		| \langle \Phi_{0} f, g \rangle | \leq \frac{1}{2} \left[ \langle f, f \rangle + \langle g, g\rangle \right].
	\end{equation*}
	For $\|f\|, \|g\| \leq 1$, we have $| \langle \Phi_{0} f, g \rangle | \leq 1$ and thus, $\| \Phi_{0} f \| \leq 1$. It follows that for $f \in \Span \{ k_{x, X_{1}}: x \in X_{1} \}$, $\| \Phi_{0} f \| \leq \| f \|$. As a consequence, $\Phi_{0}$ uniquely extends by continuity to a contraction $\Phi$ on $\sH(K_{X_{1}})$ satisfying $\Phi k_{x,X_{1}} = k_{x, X_{2}}$ for $x \in X_{1} \cap X_{2}$ by construction.

	To show the converse, let $\Phi: \sH(K_{X_{1}}) \to \sH(K_{X_{2}})$ be a contraction satisfying $\Phi k_{x,X_{1}} = k_{x, X_{2}}$ for $x \in X_{1} \cap X_{2}$. Define 
	\begin{equation}\label{eqn:compln-contracn}
		K(x, y) = \begin{cases}
			K_{\Omega}(x,y) & \mbox{for } (x, y) \in \Omega,\\
			\langle \Phi k_{x, X_{1}}, k_{y, X_{2}} \rangle 
			& \mbox{if } x \in X_{1} \setminus X_{2} \mbox{ and } y \in X_{2} \setminus X_{1}, \\
			\langle \Phi k_{y, X_{1}}, k_{x, X_{2}} \rangle 
			& \mbox{if } y \in X_{1} \setminus X_{2} \mbox{ and } x \in B \setminus X_{1}.
		\end{cases}		
	\end{equation}
	By construction, $K$ is symmetric. We can write (\ref{eqn:quad_form}) using the Cauchy-Schwarz inequality as
	\begin{eqnarray*}
		&\langle f, f \rangle + 2 \langle \Phi f, g \rangle + \langle g, g \rangle
		&\geq \|f\|^{2} - 2\|f\|\|g\| + \|g\|^{2} \\ 
		&&= (\|f\|-\|g\|)^{2} \geq 0.
	\end{eqnarray*}
	It follows that $K$ is indeed a completion. Hence proved.
\end{proof}

This is a slightly more general analogue of a standard operator-theoretic result \cite[Lemma 2.4.4]{bakonyi2011} concerning the necessary and sufficient conditions for the non-negativity of a $2 \times 2$ operator matrix with only the diagonal entries specified, which was derived by Baker \cite[Theorem 1]{baker1973} in the context of joint Gaussian measures on Hilbert spaces.

\subsection{Existence, Characterization and Uniqueness}
\label{sec:exist_and_charac}

We are now prepared to deal with the more general case where $\Delta$ can be non-empty. We can assume without any loss of generality that $X_{1} \cap X_{2} = \varnothing$, by simply taking $X_{2} = X \setminus X_{1}$. Notice that every completion of $K_{\Omega}$ is also a completion of $K_{\Omega}|_{(X_{1} \times X_{1}) \cup (X_{2} \times X_{2})}$. By Theorem \ref{thm:2sq-contraction}, we can write every completion $K$ of $K_{\Omega}$ as in (\ref{eqn:contraction-characterization}) for some contraction $\Phi: \sH(K_{X_{1}}) \to \sH(K_{X_{2}})$ satisfying $\langle \Phi  k_{x, X_{1}}, k_{y, X_{2}} \rangle = K(x, y)$ for $(x, y) \in \Delta$, which can be thought of as a linear equation in $\Phi$. Indeed, using (\ref{eqn:iso_pi}) allows us to rewrite it as
\begin{equation}\label{eqn:Phi-2}
	\Phi[k_{x,X_{1}} \otimes k_{y, X_{2}}] = K_{\Omega}(x, y) \quad\mbox{ for } (x, y) \in \Delta,
\end{equation}
where $\Phi$ is a bounded linear functional on the projective tensor product space $\sH(K_{X_{1}}) \opi \sH(K_{X_{2}})$. It follows that $\Phi$ is a bounded extension of the linear functional $\Phi_{0}: \Span \{k_{x,X_{1}} \otimes k_{y, X_{2}}: (x, y) \in \Delta\} \to \bbR$ given by
\begin{equation}\label{eqn:Phi-3}
	\Phi_{0}[k_{x,X_{1}} \otimes k_{y, X_{2}}] = K_{\Omega}(x, y) \quad\mbox{ for } (x, y) \in \Delta.
\end{equation}
If $\Phi_{0}$, thus defined, is a linear functional of norm not exceeding $1$, then the Hahn-Banach theorem guarantees the existence of an extension $\Phi$ of $\Phi_{0}$ to $\sH(K_{X_{1}}) \opi \sH(K_{X_{2}})$ such that $\|\Phi\| = \|\Phi_{0}\|$. In fact, every extension $\Phi$ of norm not exceeding $1$ will correspond to a completion of $K_{\Omega}$ according to (\ref{eqn:contraction-characterization}). On the other hand, if $\Phi_{0}$ is not well-defined or $\|\Phi_{0}\| > 1$, then $K_{\Omega}$ does not admit a completion. The following result summarizes our discussion. 

\begin{theorem}[Existence and Characterization of Completion]\label{thm:2sq-extension}
	Consider $K_{\Omega}: \Omega \to \bbR$ on a domain $\Omega$ on $X$. Assume that there exists a partition $\{X_{i}\}_{i=1}^{2}$ of $X$ such that for every $i$, $X_{i} \times X_{i} \subset \Omega$ and $K_{X_{i}} = K_{\Omega}|_{X_{i} \times X_{i}}$ is a reproducing kernel. The following statements hold:
	\begin{enumerate}
		\item The function $K_{\Omega}$ admits a completion to $X$ if and only if \emph{(\ref{eqn:Phi-3})} defines a bounded linear functional $\Phi_{0}: \Span \{k_{x,X_{1}} \otimes k_{y, X_{2}}: (x, y) \in \Delta\} \to \bbR$ such that $\|\Phi_{0}\| \leq 1$ or equivalently,
		\begin{equation}\label{eqn:2sq-extension-1}
			\left| \sum_{j=1}^{n} \alpha_{j}K_{\Omega}(x_{j}, y_{j}) \right| \leq \left\| \sum_{j=1}^{n} \alpha_{j}k_{x_{j},X_{1}} \otimes k_{y_{j}, X_{2}} \right\|_{\pi}
		\end{equation}
		for every $n \geq 1$, $\{(x_{j}, y_{j})\}_{j=1}^{n} \subset \Delta$ and $\{\alpha_{j}\}_{j=1}^{n} \subset \bbR$.
		\item There is a bijective correspondence between the completions $K$ of $K_{\Omega}$ and bounded extensions $\Phi$ of $\Phi_{0}$ to $\sH(K_{X_{1}}) \opi \sH(K_{X_{2}})$ satisfying $\|\Phi\| \leq 1$ given by
		\begin{equation*}
			\Phi[k_{x,X_{1}} \otimes k_{y, X_{2}}] = K(x, y) \quad\mbox{ for } x \in X_{1} \mbox{ and } y \in X_{2}.
		\end{equation*}		
	\end{enumerate}
\end{theorem}
In essence, Theorem \ref{thm:2sq-contraction} together with the isomorphism (\ref{eqn:iso_pi}) allowed us to \emph{linearize} the completion problem for $K_{\Omega}$ by framing it as a linear equation (\ref{eqn:Phi-2}) on a tensor product space. 

Equation (\ref{eqn:2sq-extension-1}) is sometimes called \textit{Helly's theorem} or \textit{extension principle} (see \cite[Theoreom 7.10.1]{Narici2010} and \cite[2.3.1 Theorem]{Edwards2012}). It is reminiscent of the condition (\ref{eqn:rkhs-membership}) for a function to belong to a reproducing kernel Hilbert space. Equation (\ref{eqn:2sq-extension-1}) can also be used to derive tight lower and upper bounds for the values of completions at points outside $\Delta$. To find the maximum value $M(x, y)$ and minimum value $m(x, y)$ of $K(x, y)$ for some $(x, y) \in (X_{1} \times X_{2}) \setminus \Delta$ over the completions $K$ of $K_{\Omega}$, we consider an augmented completion problem: let $\tilde{\Omega} = \Omega \cup \{(x, y), (y, x)\}$ and define $K_{\tilde{\Omega}}: \tilde{\Omega} \to \bbR$ as $K_{\tilde{\Omega}}|_{\Omega} = K_{\Omega}$ and $K_{\tilde{\Omega}}(x, y) = K_{\tilde{\Omega}}(y, x) = \nu$. There exists a completion $K$ of $K_{\Omega}$ which satisfies $K(x, y) = \nu$ if and only if $K_{\tilde{\Omega}}$ admits a completion, which is when
\begin{equation*}
	\textstyle \left| \nu - \sum_{j=1}^{n} \alpha_{j}K_{\Omega}(x_{j}, y_{j}) \right| \leq \left\| k_{x,X_{1}} \otimes k_{y, X_{2}} -  \sum_{j=1}^{n} \alpha_{j}k_{x_{j},X_{1}} \otimes k_{y_{j}, X_{2}} \right\|_{\pi}
\end{equation*}
for every $n \geq 1$, $\{(x_{j}, y_{j})\}_{j=1}^{n} \subset \Delta$ and $\{\alpha_{j}\}_{j=1}^{n} \subset \bbR$. It follows that
\begin{alignat*}{2}
	\textstyle M(x, y) &= \inf &&\textstyle\left\lbrace \sum_{j=1}^{n} \alpha_{j}K_{\Omega}(x_{j}, y_{j}) + \left\| k_{x,X_{1}} \otimes k_{y, X_{2}} -  \sum_{j=1}^{n} \alpha_{j}k_{x_{j},X_{1}} \otimes k_{y_{j}, X_{2}} \right\|_{\pi} \right\rbrace \\
	\textstyle m(x, y) &= \sup &&\textstyle\left\lbrace \sum_{j=1}^{n} \alpha_{j}K_{\Omega}(x_{j}, y_{j}) - \left\| k_{x,X_{1}} \otimes k_{y, X_{2}} - \sum_{j=1}^{n} \alpha_{j}k_{x_{j},X_{1}} \otimes k_{y_{j}, X_{2}} \right\|_{\pi} \right\rbrace
\end{alignat*}
where the supremum and infimum are taken over $n \geq 1$, $\{(x_{j}, y_{j})\}_{j=1}^{n} \subset \Delta$ and $\{\alpha_{j}\}_{j=1}^{n} \subset \bbR$. Note that the value of a completion $K$ s uniquely determined at $(x, y)$ if and only if $m(x, y) = M(x, y)$. If $m(x, y) = M(x, y)$ for every $(x, y)$ outside $\Delta$, then $K_{\Omega}$ admits a unique completion. The following result is now immediate.

\begin{theorem}[Uniqueness of Completion]\label{thm:compln-uniqueness}
	Let $K_{\Omega}$ be as in Theorem \ref{thm:2sq-extension}. Then $K_{\Omega}$ admits a unique completion if and only if for every $(x, y) \in X_{1} \times X_{2} \setminus \Delta$ and $\epsilon > 0$ there exist $n \geq 1$, $\{(x_{j}, y_{j})\}_{j=1}^{n} \subset \Delta$ and $\{\alpha_{j}\}_{j=1}^{n}, \{\beta_{j}\}_{j=1}^{n} \subset \bbR$ such that
	\begin{align*}		
			\textstyle \sum_{j=1}^{n} (\alpha_{j} - \beta_{j})K_{\Omega}(x_{j}, y_{j}) +
			\left\{\begin{aligned}
			&\textstyle \left\| k_{x,X_{1}} \otimes k_{y, X_{2}} - \sum_{j=1}^{n} \alpha_{j}k_{x_{j},X_{1}} \otimes k_{y_{j}, X_{2}} \right\|_{\pi} \\ 
			&\textstyle - \left\| k_{x,X_{1}} \otimes k_{y, X_{2}} + \sum_{j=1}^{n} \beta_{j}k_{x_{j},X_{1}} \otimes k_{y_{j}, X_{2}} \right\|_{\pi}
		\end{aligned} \right\} < \epsilon.
	\end{align*}
	In particular, this holds if $\Span \{k_{x,X_{1}} \otimes k_{y, X_{2}}: (x, y) \in \Delta\}$ is dense in  $\sH(K_{X_{1}}) \opi \sH(K_{X_{2}})$.
\end{theorem}

\begin{remark}\label{rmk:extra-constraint}
	The linearization approach to completion can be used to determine the existence of completions with given constraints so long as the constraints are linear for $\Phi$. For example, if we want to ascertain whether there exists a completion $K$ of $K_{\Omega}$ for which $K(x, y) = K(x', y')$ for some points $(x, y), (x', y')$ outside $\Omega$, we need only to impose an additional constraint on $\Phi_{0}$, that is $\Phi_{0}[k_{x, X_{1} \otimes k_{y, X_{2}}} - k_{x', X_{1} \otimes k_{y', X_{2}}}] = 0$ and check if $\|\Phi_{0}\| \leq 1$ as before. We can do the same for a partial derivative $\partial_{1}K$ of a completion $K$ which can be expressed as $\partial_{1}K(x, y) = \langle \Phi k_{x, X_{1}}', k_{y, X_{2}^{}} \rangle$ for some $k_{x, X_{1}}' \in \sH(K_{X_{1}})$ under appropriate conditions. This allows us to find the maximum and minimum values of the derivative of a completion at any point.
\end{remark}

\subsection{Completion on Large Regular Domains} 

Although Equation (\ref{eqn:2sq-extension-1}) may appear too unwieldy to be of any use, it is quite straightforward to apply it for bootstrapping on results for finite domains such as those concerning completions of matrices.

\begin{theorem}\label{thm:completion-regular}
	Let $\Omega$ be a large regular domain on $X = [0, 1]$. Every partially reproducing kernel $K_{\Omega}$ admits a completion.
\end{theorem}
\begin{proof}
	Pick $n \geq 1$, $\{(x_{i}, y_{i})\}_{i=1}^{n} \subset \Delta$ and $\{\alpha_{i}\}_{i=1}^{n} \subset \bbR$. Let $\sF_{1} = \{x_{i}\}_{i=1}^{n}$, $\sF_{2} = \{y_{i}\}_{i=1}^{n}$ and $\sF = \sF_{1} \cup \sF_{2}$. We consider the completion problem as restricted to $\sF \times \sF$. According to a classical result \cite[Theorem 7]{grone1984} concerning the completions of partially specified Hermitian matrices, $K_{\Omega}$ restricted to $\Omega \cap (\sF \times \sF)$ admits an extension to $\sF \times \sF$. By Theorem \ref{thm:2sq-extension}, this means
	\begin{equation*}
		\left|\sum_{i=1}^{n} \alpha_{i}K_{\Omega}(x_{i}, y_{i})\right| \leq \left\| \sum_{i=1}^{n} \alpha_{i}k_{x_{i},\sF_{1}} \otimes k_{y_{i}, \sF_{2}} \right\|_{\pi}.
	\end{equation*}
	Let $R_{1}: \sH(K_{X_{1}}) \to \sH(K_{\sF_{1}})$ and $R_{2}:\sH(K_{X_{2}}) \to \sH(K_{\sF_{2}})$ denote the restrictions to $\sF_{1}$ and $\sF_{2}$ respectively. 
	Observe that for every contraction $\Phi_{F}:\sH(K_{\sF_{1}}) \to \sH(K_{\sF_{2}})$ there exists a contraction $\Phi:\sH(K_{X_{1}}) \to \sH(K_{X_{2}})$ such that $\Phi = R_{2}^{\ast}\Phi_{F} R_{1}^{}$:
	\begin{equation*}
		\langle \Phi_{F} k_{x,\sF_{1}}, k_{y,\sF_{2}} \rangle = \langle \Phi_{F} R_{1}k_{x,X_{1}}, R_{2}k_{y,X_{2}} \rangle = \langle R_{2}^{\ast} \Phi_{F}^{} R_{1}^{}k_{x,X_{1}}, k_{y,X_{2}} \rangle.
	\end{equation*}
	Using the duality formula (\ref{eqn:duality_formula}), we can write
	\begin{align*}
		\textstyle \| \sum_{i=1}^{n} \alpha_{i}k_{x_{i},\sF_{1}} \otimes k_{y_{i}, \sF_{2}} \|_{\pi} 
		&= \textstyle \sup 
		\{|\sum_{i=1}^{n} \alpha_{i} \langle \Phi_{F} k_{x_{i},\sF_{1}}, k_{y_{i},\sF_{2}} \rangle|: 
		\| \Phi_{F}\| \leq 1\}\\ 
		&\leq \textstyle \sup 
		\{|\sum_{i=1}^{n} \alpha_{i} \langle \Phi k_{x_{i}, X_{1}}, k_{y_{i},X_{2}} \rangle|: \|\Phi\| \leq 1 \} \\
		&= \textstyle \| \sum_{i=1}^{n} \alpha_{i}k_{x_{i},  X_{1}} \otimes k_{y_{i}, X_{2}} \|_{\pi}. 
	\end{align*}
	Therefore,
	\begin{equation*}
		\left|\sum_{i=1}^{n} \alpha_{i}K_{\Omega}(x_{i}, y_{i})\right| \leq \left\| \sum_{i=1}^{n} \alpha_{i}k_{x_{i}, X_{1}} \otimes k_{y_{i}, X_{2}} \right\|_{\pi}
	\end{equation*}
	and the conclusion follows from Theorem \ref{thm:2sq-extension}. 
	The converse is trivially true because any completion of $K_{\Omega}$ restricted to $\sF \times \sF$ is a completion of $K_{\Omega}|_{\sF \times \sF}$.
\end{proof}

Of course, we could have derived the result far more easily using Theorem \ref{thm:sC-exist-finite}. But this was good preparation for proving Artjomenko's generalization of Krein's extension theorem which is what follows.

\subsection{Extension of Positive-definite Functions}

Let $F:(-a, a) \to \bbR$ be a positive-definite function for some $a > 0$. An extension $\tilde{F}: (-2a, 2a) \to \bbR$ of $F$ is a positive-definite function such that $\tilde{F}|_{(-a, a)} = F$. To express the extension problem of $F$ as a completion problem on a large domain, let $X = [0, 2a)$ with $X_{1} = [0, a)$ and $X_{2} = [a, 2a)$. Define $K_{\Omega}: \Omega \to \bbR$ as $K_{\Omega}(x, y) = F(x - y)$ for $\Omega= \{(x, y): |x - y| < a\} \subset X \times X$. The extensions $\tilde{F}$ of $F$ correspond to the stationary completions $\tilde{K}$ of $K_{\Omega}$. As discussed in Remark \ref{rmk:extra-constraint}, we can account for the stationarity of $K$ by imposing an additional constaint on $\Phi_{0}$. Define $\cS, \cT \subset X_{1} \times X_{2}$ as
\begin{align*}
	\cS &= \Span \{ k_{x, X_{1}} \otimes k_{y, X_{2}} : y - x < a\}, \mbox{ and } \\\cT &= \Span \{ k_{x, X_{1}} \otimes k_{y, X_{2}} - k_{w, X_{1}} \otimes k_{z, X_{2}} : y - x = z - w\}.
\end{align*}
To show that a stationary completion $\tilde{K}$ exists, we need to show that there exists a contraction $\Phi$ such that $\langle \Phi k_{x, X_{1}}, k_{y, X_{2}} \rangle = F(y - x)$ and $\Phi[\tau] = 0$ for $\tau \in \cT$.

\begin{theorem}
	Every positive-definite function $F$ of $(-a, a) \subset \bbR$ for some $a > 0$ admits an extension to $(-2a, 2a)$.
\end{theorem}
\begin{proof}
	As before, we construct a grid. Pick $\delta > 0$ and let $n = \max\{j: j \delta < a\}$. Let $\sF_{1} = \{x_{i}\}_{i=1}^{n}$ where $x_{i} = a - \delta i \in X_{1}$ for $1 \leq i \leq n$ and $\sF_{2} = \{y_{j}\}_{j=0}^{n}$ where $y_{j} = a + \delta j \in X_{2}$ for $0 \leq j \leq n$. Let $\sF = \sF_{1} \cup \sF_{2}$ The restriction of $K_{\Omega}$ to $\sF \times \sF$ can now be thought of as a partially specified matrix $\bA = [A_{ij}]_{i,j=1}^{2n+1}$ where 
	\begin{equation*}
		A_{ij} = \begin{cases}
			F(\delta|i-j|) &\mbox{ for } |i - j| \leq n+1\\
			\mbox{unspecified.} &\mbox{ for } |i - j| > n+1
		\end{cases}
	\end{equation*}
	By Carath\'{e}odory's result, this partially specified matrix admits a positive-definite completion which is also Toeplitz. 
	We can argue as in Theorem \ref{thm:completion-regular} that
	\begin{equation}\label{eqn:phi-ineq}
		|\Phi_{0}[\sigma]| \leq \|\sigma + \tau\|_{\pi} 
	\end{equation}
	for $\sigma \in \cS_{0}$ and $\tau \in \cT_{0}$ for dense subsets $\cS_{0} \subset \cS$ and $\cT_{0} \subset \cT$ given by
	\begin{align*}
		\cS_{0} &= \cup_{\delta > 0} \left[ \Span \{ k_{x_{i}, X_{1}} \otimes k_{y_{j}, X_{2}}: \delta|j-i| \leq a\} \right]\\
		\cT_{0} &= \cup_{\delta > 0} \left[ \Span \{ k_{x_{i}, X_{1}} \otimes k_{y_{j}, X_{2}} - k_{x_{k}, X_{1}} \otimes k_{y_{l}, X_{2}}: i-j = k - l\} \right]
	\end{align*}
	The density follows from the observation that every bounded linear functional which vanishes on $\cS_{0}$ ($\cT_{0}$) vanishes on all of $\cS$ ($\cT$). The inequality (\ref{eqn:phi-ineq}) implies that $\Phi_{0}$ is well-defined. Using extension of continuity, we have that (\ref{eqn:phi-ineq}) holds for all $\sigma \in \cS$ and $\tau \in \cT$. The conclusion now follows from the Hahn-Banach theorem as in Theorem \ref{thm:2sq-extension}.
\end{proof}

Since every positive-definite function on $(-a, a)$ can be extended to a positive-definite function on $(-2a, 2a)$ for any $a > 0$, we can iterate the argument and conclude:

\begin{corollary}
	Every positive-definite function $F:(-a, a) \to \bbR$ for some $a > 0$, admits an extension $\tilde{F}$ to the real line.
\end{corollary}

Needless to say, we can derive analogous expressions for the maximum and minimum values of the extension on $(-2a, 2a)$ as well as conditions for uniqueness.

\subsection{Beyond Large Domains}
\label{sec:beyond-large-domains}
There does not appear to be an obvious way of extending the linearization technique to ``smaller'' domains, say if $\Omega$ is a domain on $X$ such that $\Omega \supset \cup_{i} (X_{i} \times X_{i})$ for some partition $\{X_{i}\}_{i=1}^{p}$ of $X$ where $p > 2$. However, we can still draw some general conclusions using Theorem \ref{thm:2sq-contraction}.
\begin{theorem}
	Let $\Omega$ be as above and $K_{\Omega}: \Omega \to \bbR$ be such that the restrictions $K_{X_{i}} = K_{\Omega}|_{x_{i} \times x_{i}}$ are reproducing kernels. Then there exists a positive semidefinite operator matrix $[\Phi_{ij}]_{i,j=1}^{p}$ of contractions $\Phi_{ij}: \sH(K_{X_{j}}) \to \sH(K_{X_{j}})$ such that
	\begin{equation*}
		K(x, y) = \langle \Phi_{ji}k_{x, X_{i}}, k_{y, X_{j}} \rangle \quad\mbox{ for } x \in X_{i} \mbox{ and } y \in X_{j}. 
	\end{equation*}
\end{theorem}
A consequence of the Banach-Alaoglu theorem and the above embedding of completions into the product of unit balls in $\cL(\sH(K_{X_{i}}), \sH(K_{X_{j}}))$ is the following result.

\begin{theorem}\label{thm:sC-seq-cmpct}
	If $X$ can be partitioned into a finite or countably infinite number of cliques $X_{i}$ in $\Omega$, then the set of completions $\sC$ of $K_{\Omega}$ is sequentially compact under the topology of pointwise convergence.
\end{theorem} 

\section{Canonical Completion}
\label{sec:canonical-completion}

In this section, we study special solutions of the completion problem we shall call canonical completions. We begin by introducing a family of domains for which doing this is particularly simple.
\begin{definition}[Serrated Domain]\label{def:serrated-domains}
	Let $X$ be a set. We say that a domain $\Omega$ on $X$ is a $n$-\emph{serrated domain} if there exists $n \geq 1$ and subsets $\{X_{j}\}_{j=1}^{n}$ of $X$ such that \emph{(a)} $X = \cup_{j=1}^{n} X_{j}$ \emph{(b)} $X_{i} \cap X_{k} \subset X_{i} \cap X_{j}$ for $1 \leq i < j < k \leq n$ and \emph{(c)} $\Omega = \cup_{j=1}^{n} (X_{j} \times X_{j})$. 

	Furthermore, every $n$-serrated domain is a serrated domain.
\end{definition}
We shall derive interesting characterizations of canonical completions in terms of their associated norms and contraction maps. We shall also prove partial analogues of the classical results concerning determinant maximization and inverse zero properties known for matrices. We shall also extend some of these results to a larger families of domains.

\subsection{Contractions}

Let $K$ be a reproducing kernel on $X$. For $A, B \subset X$, let $K_{A} = K|_{A \times A}$ and $K_{B} = K|_{B \times B}$. Define $\Phi_{BA}: \sH(K_{A}) \to \sH(K_{B})$ as the unique bounded linear map satisfying $\Phi_{BA}k_{x, A} = k_{x, B}$ for $x \in A$. 
By thinking of $K_{A \cup B}$ as the completion of $K_{\Omega} = K|_{\Omega}$ where $\Omega = (A \times A) \cup (B \times B)$, we can deduce using Theorem \ref{thm:2sq-contraction} that $\Phi_{BA}$ is a contraction. We shall see that these contraction maps can be used to construct completions.

\begin{theorem}[Properties of Contraction Maps]
	Let $A, B \subset X$ and $f \in \sH(K_{A})$.
	\begin{enumerate}
		\item Adjoint. $\Phi_{BA}^{\ast} = \Phi_{AB}^{}$,
		\item Evaluation. $\Phi_{BA}f(y) = \langle f, k_{y, A} \rangle$ for $y \in B$.
		\item Restriction. If $B \subset A$, then $\Phi_{BA}f = f|_{B}$,
		\item Minimum Norm Interpolation. If $A \subset B$, then
		\begin{equation*}
			\Phi_{BA}f = \argmin_{g \in \sH(K_{B})} \{\|g\|: g|_{A} = f\}
		\end{equation*}
	\end{enumerate}
\end{theorem}
\begin{proof}
Property (1) follows from writing
\begin{equation*}
	\langle k_{y, B}, \Phi_{BA}k_{x, A} \rangle = \langle k_{x, B}, k_{y, B} \rangle = K(x, y) = \langle k_{x, A}, k_{y, A} \rangle = \langle \Phi_{AB}k_{y, B}, k_{x, A} \rangle. 
\end{equation*}
for every $x \in A$ and $y \in B$. For properties (2) and (3), notice that $\Phi_{BA}f(y) = \langle \Phi_{BA}f, k_{y, B} \rangle = \langle f, \Phi_{AB}k_{y, B} \rangle = \langle f, k_{y, A} \rangle$ for $y \in B$ which is equal to $f(y)$ if $A \subset B$. Finally, to show property (4), let $g \in \sH(K_{B})$ such that $g|_{A} = f$. Then
\begin{equation*}
	\langle g - \Phi_{BA}f, \Phi_{BA}f \rangle = \langle \Phi_{AB} g, f \rangle - \langle f, \Phi_{AB}\Phi_{BA}f \rangle = \langle g|_{A},f \rangle - \langle f, f \rangle = 0
\end{equation*}
and we can write $\|g\|^{2} = \|g- \Phi_{BA}f\|^{2} + \|\Phi_{BA}f\|^{2}$ which implies that the norm of $g$ is minimum precisely when $g = \Phi_{BA}f$. Hence proved.
\end{proof}
 
\subsection{Canonical Completion for 2-Serrated Domains}

Consider a set $X$ with subsets $X_{1}$ and $X_{2}$ such that $X_{1} \cup X_{2} = X$. Let $\Omega = (X_{1} \times X_{1}) \cup (X_{2} \times X_{2})$. Let $K_{\Omega}$ be a partially reproducing kernel on $\Omega$. In other words, $K_{X_{1}} = K_{\Omega}|_{X_{1} \times X_{1}}$ and $K_{X_{2}} = K_{\Omega}|_{X_{2} \times X_{2}}$ are reproducing kernels on $X_{1}$ and $X_{2}$ respectively. Using Theorem \ref{thm:sC-exist-finite}, we can argue as in Theorem \ref{thm:completion-regular}, that $K_{\Omega}$ admits a completion. Furthermore, by Theorem \ref{thm:2sq-contraction}, the set $\sC$ of completions $K$ is parametrized by contractions $\Phi: \sH(K_{X_{1}}) \to \sH(K_{X_{2}})$ satisfying $\Phi k_{x, X_{1}} = k_{x, X_{2}}$ for $x \in X_{1} \cap X_{2}$ according to the relation
\begin{equation*}
	K(x, y) = \langle \Phi k_{x, X_{1}}, k_{y, X_{2}} \rangle
\end{equation*}
for $x \in X_{1}$ and $y \in X_{2}$. The case where $x \in X_{2}$ and $y \in X_{1}$ is covered by the symmetry of $K$.

We shall construct a special completion $K_{\star}$ of $K$. Notice that for $x \in X_{1} \cap X_{2}$, the contraction map $\Phi_{X_{1} \cap X_{2}, X_{1}}$ maps $k_{x, X_{1}}$ to $k_{x, X_{1} \cap X_{2}}$ and the contraction map $\Phi_{X_{2}, X_{1} \cap X_{2}}$ maps $k_{x, X_{1} \cap X_{2}}$ to $k_{x, X_{2}}$. It follows that the product $\Phi_{\star} = \Phi_{X_{2}, X_{1} \cap X_{2}}\Phi_{X_{1} \cap X_{2}, X_{1}}$ satisfies $\Phi_{\star} k_{x, X_{1}} = k_{x, X_{2}}$ for $x \in X_{1} \cap X_{2}$ and is obviously a contraction by virtue of being the product of two contractions. We have thus successfully constructed a member of the family of contractions $\Phi$ which parametrizes $\sC$. Corresponding to the constructed contraction $\Phi_{\star}$ is a completion $K_{\star}$ of $K$ given by 
\begin{equation*}
	K_{\star}(x, y) = \langle \Phi_{\star} k_{x, X_{1}}, k_{y, X_{2}} \rangle
\end{equation*}
for $x \in X_{1}$ and $y \in X_{2}$. If for $x \in X$ and $U \subset X$, we define $k_{x, U}^{\star}: U \to \bbR$ as $k_{x, U}^{\star}(y) = K_{\star}(x, y)$, then we can also describe $K_{\star}$ in terms of its generators as
\begin{equation}\label{eqn:canonical-2-lineq-1}
	k_{x, X_{2}}^{\star} = \Phi k_{x, X_{1}}^{} \mbox{ for } x \in X_{1}
\end{equation}
or equivalently, $k_{y, X_{1}}^{\star} = \Phi^{\ast} k_{y, X_{2}}^{}$ for $y \in X_{2}$. Alternatively, we can express $K_{\star}$ without using any contraction maps as
\begin{align*}
	K_{\star}(x, y) 
	&= \langle \Phi_{\star} k_{x, X_{1}}, k_{y, X_{2}} \rangle\\
	&= \langle \Phi_{X_{1} \cap X_{2}, X_{1}}k_{x, X_{1}}, \Phi_{X_{1} \cap X_{2}, X_{2}}k_{y, X_{2}} \rangle\\
	&= \langle k_{x, X_{1} \cap X_{2}}, k_{y, X_{1} \cap X_{2}} \rangle
\end{align*}
for $x \in X_{1}$ and $y \in X_{2}$. Needless to say, the completion does not change if we switch $X_{1}$ and $X_{2}$, which implies that is a property of $K_{\Omega}$ only and does not depend on how $\Omega$ is parametrized in terms of $X_{1}$ and $X_{2}$. The following result summarizes the above discussion.

\begin{theorem}\label{thm:canonical-compln}
	Let $X$ be a set with subsets $X_{1}, X_{2} \subset X$ such that $X_{1} \cup X_{2} = X$ and let $\Omega = (X_{1} \times X_{1}) \cup (X_{2} \times X_{2})$. If $K_{\Omega}$ is a partially reproducing kernel on $\Omega$, then it admits a completion given by
	\begin{equation*}
		K_{\star}(x, y) = \langle \Phi_{\star} k_{x, X_{1}}, k_{y, X_{2}} \rangle
	\end{equation*}
	for $x \in X_{1}$ and $y \in X_{2}$, where $\Phi_{\star} = \Phi_{X_{2},X_{1} \cap X_{2}}\Phi_{X_{1} \cap X_{2}, X_{1}}$. Furthermore, $K_{\star}$ can also be expressed as
	\begin{equation}\label{eqn:compln-eqn}
		K_{\star}(x, y) = \langle k_{x, X_{1} \cap X_{2}}, k_{y, X_{1} \cap X_{2}} \rangle
	\end{equation}
	for $x \in X_{1}$ and $y \in X_{2}$.
\end{theorem}

\subsubsection{Minimum Norm Interpolation} 

We can acquire deeper insight into the construction of $K_{\star}$ by understanding the nature of the contraction $\Phi_{\star}$. Notice that we can rewrite Equation (\ref{eqn:canonical-2-lineq-1}) as $k_{x, X_{2}}^{\star} = \Phi_{X_{2}, X_{1} \cap X_{2}} k_{x, X_{1} \cap X_{2}}^{}$ for $x \in X_{1}$. It follows that $k_{x, X_{2}}^{\star}$ for $x \in X_{1}$ is the minimum norm interpolation of $k_{x, X_{1} \cap X_{2}}$ in $\sH(K_{X_{2}})$: 
\begin{equation}\label{eqn:canonical-min-norm-1}
	k_{x, X_{2}}^{\star} = \argmin_{f\in \sH(K_{X_{2}})} \{ \|f\|: f|_{X_{1} \cap X_{2}} = k_{x, X_{1} \cap X_{2}} \} \quad \mbox{ for } x \in X_{1}.
\end{equation}
Similarly, we can write for $y \in X_{2}$, that $k_{y, X_{1}}^{\star} = \argmin\{ \|g\| \}$ over $g \in \sH(K_{X_{1}})$ such that $g|_{X_{1} \cap X_{2}} = k_{y, X_{1} \cap X_{2}}$.
\subsubsection{Characterization of Completions} 

In a certain sense, the canonical completion $K_{\star}$ lies at the center of the set of completions $\sC$, which allows us to come up with a simpler characterizations of completions of $K_{\Omega}$ than provided by Theorem \ref{thm:2sq-extension}. The following result is a reproducing kernel counterpart of a classic result \textcite[][Section II]{johnson1990} in the theory of matrix completions.
\begin{theorem}
	There is a bijective correspondence between the completions $K$ of $K_{\Omega}$ and the contractions $\Psi: \sH(K_{X_{1}}/K_{X_{1} \cap X_{2}}) \to \sH(K_{X_{2}}/K_{X_{1} \cap X_{2}})$ given by
	\begin{equation}\label{eqn:johnson-formula}
		K(x, y) = \langle k_{x, X_{1} \cap X_{2}}, k_{y, X_{1} \cap X_{2}} \rangle + \langle \Psi p_{x}, q_{y}\rangle
	\end{equation}
	for $x \in X_{1} \setminus X_{2}$ and $y \in X_{2} \setminus X_{1}$, where $p_{x}(u) = K_{X_{1}}/K_{X_{1} \cap X_{2}}(x, u)$ and $q_{y}(u) = K_{X_{2}}/K_{X_{1} \cap X_{2}}(y, u)$.
\end{theorem}
\begin{proof}
	Define the partially reproducing kernel $D_{\Omega}$ on $\Omega$ given by $D_{\Omega}(x, y) = K_{\Omega}(x, y) - \langle k_{x, X_{1} \cap X_{2}}, k_{y, X_{1} \cap X_{2}} \rangle$ for $(x, y) \in \Omega$. Clearly, 
	\begin{equation*}
		D_{\Omega}(x, y) = \begin{cases}
			K_{X_{1}}/K_{X_{1} \cap X_{2}}(x, y) & x, y \in X_{1} \setminus X_{2} \\
			0 & x \mbox{ or } y \in X_{1} \cap X_{2}\\
			K_{X_{2}}/K_{X_{1} \cap X_{2}}(x, y) & x, y \in X_{2} \setminus X_{1} \\			
		\end{cases}
	\end{equation*} 
	which means that $D_{\Omega}$ is indeed a partially reprocuding kernel. Notice that because $D_{\Omega}(x, y) = 0$ for $x,y \in X_{1} \cap X_{2}$, completing $D_{\Omega}$ is equivalent to completing $D_{\Omega}|_{\tilde{\Omega}}$ to a kernel on $X \setminus (X_{1} \cap X_{2})$ where $\tilde{\Omega} = [(X_{1} \setminus X_{2}) \times (X_{1} \setminus X_{2})] \cup [(X_{2} \setminus X_{1}) \times (X_{2} \setminus X_{1})]$. The setting of completing $D_{\Omega}|_{\tilde{\Omega}}$ is equivalent to that of Theorem \ref{thm:2sq-contraction} and therefore, the completions $D$ of $D_{\Omega}$ are characterized by $D(x, y) = \langle \Psi p_{x}, q_{y}\rangle$ for $x \in X_{1} \setminus X_{2}$ and $y \in X_{2} \setminus X_{1}$ where $\Psi: \sH(K_{X_{1}}/K_{X_{1} \cap X_{2}}) \to \sH(K_{X_{2}}/K_{X_{1} \cap X_{2}})$ is a contraction. 
	
	Notice that there is a bijective correspondence between the completions $K$ of $K_{\Omega}$ and the completions $D$ of $D_{\Omega}$ given by $D(x, y) = K(x, y) - \langle k_{x, X_{1} \cap X_{2}}, k_{y, X_{1} \cap X_{2}} \rangle$ for $x, y \in X$. Indeed, for every completion $K$ we can write for $x, y \in X$
	\begin{equation*}
		D(x, y) = \langle k_{x} - \Pi_{X_{1} \cap X_{2}}k_{x}, k_{y} - \Pi_{X_{1} \cap X_{2}}k_{y} \rangle 
	\end{equation*}
	which satisfies $D|_{\Omega} = D_{\Omega}$ and is clearly a reproducing kernel. On the other hand, for every completion $D$, we have for $x, y \in X$,
	\begin{equation*}
		K(x, y) = D(x, y) + \langle k_{x, X_{1} \cap X_{2}}, k_{y, X_{1} \cap X_{2}} \rangle.
	\end{equation*}
	Since$(x, y) \mapsto \langle k_{x, X_{1} \cap X_{2}}, k_{y, X_{1} \cap X_{2}} \rangle$ $K|_{\Omega} = K_{\Omega}$ and $D$ are reproducing kernels so is $K$ and clearly, $K|_{\Omega} = K_{\Omega}$. The conclusion follows.
\end{proof}

Using Theorem \ref{thm:pm-adjoint}, we can obtain a slightly more elegant characterization of $\sC$ by simply observing that $\sC$ is centered around $K_{\star}$ and for $H: X \times X \to \bbR$, $K_{\star}+H \in \sC$ if and only if $K_{\star}-H \in \sC$.
\begin{corollary}
	There is a bijective correspondence between the completions $K$ of $K_{\Omega}$ and self-adjoint contractions $\Psi: \sH(K_{\star}) \to \sH(K_{\star})$ satisfying $\langle \Psi k_{x}^{\star}, k_{y}^{\star} \rangle = 0$ for $(x, y) \in \Omega$ given by
	\begin{equation}\label{eqn:johnson-formula-1}
		K(x, y) = \langle (\bI + \Psi) k_{x}^{\star}, k_{y}^{\star}\rangle
	\end{equation}
	for $x, y \in X$, where $k_{x}^{\star} \in \sH(K_{\star})$ is given by $k_{x}^{\star}(y) = K_{\star}(x, y)$.
\end{corollary}

\subsubsection{The Associated Inner Product}

We shall now derive the inner product associated with the canonical completion for $2$-serrated domains. To this end, the following lemma is useful. 
\begin{lemma}\label{thm:schur-canonical}
	If $K = K_{\star}$, then we have $K/K_{X_{1}} = K_{X_{2}}/K_{X_{1} \cap X_{2}}$ and $K/K_{X_{2}} = K_{X_{1}}/K_{X_{1} \cap X_{2}}$.
\end{lemma}
\begin{proof}
	For $x, y \in X_{2} \setminus X_{1}$ 
	\begin{eqnarray*}
		K/K_{X_{1}}(x, y) 
		&=& K(x, y) - \langle k_{x, X_{1}}, k_{y, X_{1}} \rangle \\
		&=& K(x, y) - \langle \Phi_{X_{1}, X_{1} \cap X_{2}}\Phi_{X_{1} \cap X_{2}, X_{2}}k_{x, X_{2}}, \Phi_{X_{1}, X_{1} \cap X_{2}}\Phi_{X_{1} \cap X_{2}, X_{2}}k_{y, X_{2}} \rangle \\
		&=& K(x, y) - \langle \Phi_{X_{1}, X_{1} \cap X_{2}}k_{x, X_{1} \cap X_{2}}, \Phi_{X_{1}, X_{1} \cap X_{2}}k_{y, X_{1} \cap X_{2}} \rangle \\
		&=& K(x, y) - \langle \Phi_{X_{1} \cap X_{2}, X_{1}}\Phi_{X_{1}, X_{1} \cap X_{2}}k_{x, X_{1} \cap X_{2}}, k_{y, X_{1} \cap X_{2}} \rangle \\
		&=& K(x, y) - \langle k_{x, X_{1} \cap X_{2}}, k_{y, X_{1} \cap X_{2}} \rangle = K_{X_{2}}/K_{X_{1} \cap X_{2}}(x, y).
	\end{eqnarray*}
	We can argue similarly that $K/K_{X_{2}} = K_{X_{1}}/K_{X_{1} \cap X_{2}}$. Hence proved.
\end{proof}
Thus, when $K = K_{\star}$, the Schur complements $K/K_{X_{1}}$ and $K/K_{X_{2}}$ also simplify.

\begin{theorem}\label{thm:canonical-inner-product}
	Let $K_{\Omega}$ and $K_{\star}$ be as before. The norm associated with $K_{\star}$ can be expressed as follows:
	\begin{equation}\label{eqn:canonical-inner-product}
		\|f\|_{\star}^{2} = \|f_{X_{1}}\|^{2} - \|f_{X_{1} \cap X_{2}}\|^{2} + \|f_{X_{2}}\|^{2}.
	\end{equation}
	where $\|f_{U}\|$ denotes the norm of $f_{U} = f|_{U}$ in $\sH(K_{U})$ for $U \subset X$ such that $U \times U \subset \Omega$. Consequently, the inner product is given by
	\begin{equation}\label{eqn:canonical-inner-product2}
		\langle f, g \rangle_{\star} = \langle f_{X_{1}}, g_{X_{1}} \rangle - \langle f_{X_{1} \cap X_{2}}, g_{X_{1} \cap X_{2}} \rangle + \langle f_{X_{2}}, g_{X_{2}} \rangle.
	\end{equation}
\end{theorem}
\begin{proof}
	Notice that for $f \in \sH$ we can write
	\begin{equation*}
		\|f\|^{2} = \|\Pi_{X_{1}}f + f - \Pi_{X_{1}}f\|^{2} = \|\Pi_{X_{1}}f\|^{2} + \|f - \Pi_{X_{1}}f\|^{2} + 2 \langle \Pi_{X_{1}}f, f - \Pi_{X_{1}}f\rangle 
	\end{equation*}
	where $\langle \Pi_{X_{1}}f, f - \Pi_{X_{1}}f\rangle = 0$ by the projection theorem. By (\ref{eqn:subspace-isometry}) and (\ref{eqn:schur-isometry}), $\|\Pi_{X_{1}}f\| = \|f_{X_{1}}\|$ and $\|f - \Pi_{X_{1}}f\| = \|g\|_{\sH(K/K_{X_{1}})}$ where $g \in \sH(K/K_{X_{1}})$ is given by $g(x) = (f - \Pi_{X_{1}}f)(x)$ for $x \in X_{2} \setminus X_{1}$. Thus,
	\begin{equation}\label{eqn:proj-norm1}
		\|f\|^{2} =\|f_{X_{1}}\|^{2} + \|g\|_{\sH(K/K_{X_{1}})}^{2}.
	\end{equation}	
	On a closer look, $g(x) = \langle f - \Pi_{X_{1}}f, k_{x}\rangle = f(x) - \langle f_{X_{1}}, k_{x, X_{1}, x} \rangle$ where
	\begin{eqnarray*}
		\langle f_{X_{1}}, k_{x, X_{1}, x} \rangle 
		&=& \langle f_{X_{1}}, \Phi_{X_{1}, X_{1} \cap X_{2}}\Phi_{X_{1} \cap X_{2}, X_{2}}k_{x, X_{2}} \rangle \\
		&=& \langle \Phi_{X_{1} \cap X_{2}, X_{1}}f_{X_{1}}, \Phi_{X_{1} \cap X_{2}, X_{2}}k_{x, X_{2}} \rangle \\		
		&=& \langle f_{X_{1} \cap X_{2}}, k_{x, X_{1} \cap X_{2}} \rangle
	\end{eqnarray*}
	implying that $g(x) = f(x) - \langle f_{X_{1} \cap X_{2}}, k_{x, X_{1} \cap X_{2}} \rangle = (f_{X_{2}} - \Pi_{X_{1} \cap X_{2}}f_{X_{2}})(x)$ for $x \in X_{2} \setminus X_{1}$. By Lemma \ref{thm:schur-canonical}, $K/K_{X_{1}}(x, y) = K_{X_{2}}/K_{X_{1} \cap X_{2}}(x, y)$ for $x, y \in X_{2} \setminus X_{1}$. It follows that the norm of $g$ in $\sH(K \setminus K_{X_{1}})$ can be written as
	\begin{equation}\label{eqn:schur-norm1}
		\|g\|_{\sH(K/K_{X_{1}})}^{2} = \|f_{X_{2}} - \Pi_{X_{1} \cap X_{2}}f_{X_{2}}\|^{2}  = \|f_{X_{2}}\|^{2} - \|f_{X_{1} \cap X_{2}}\|^{2}. 
	\end{equation}
	The conclusion follows by substituting (\ref{eqn:schur-norm1}) in (\ref{eqn:proj-norm1}). The inner product formulas can be derived using the observation that $\langle f, g \rangle = \tfrac{1}{4} \left[ \|f + g \|^{2} - \|f - g \|^{2} \right]$. Hence proved.
\end{proof}

\subsubsection{Projections}
Let $\Pi_{U}$ denote the projection $\sH(K)$ to the closed linear subspace spanned by $\{ k_{u, X}: u \in U \}$. Using the equivalence between a projection $\Pi_{U}$ and restriction to $U$, we can rewrite (\ref{eqn:canonical-inner-product}) as 
\begin{align*}
	\langle f, f \rangle 
	&= \|f_{X_{1}}\|^{2} + \|f_{X_{2}}\|^{2} - \|f_{X_{1} \cap X_{2}}\|^{2}\\
	&= \langle \Pi_{X_{1}} f, f \rangle + \langle \Pi_{X_{2}} f, f \rangle - \langle \Pi_{X_{1} \cap X_{2}} f, f \rangle \\
	&= \langle (\Pi_{X_{1}} + \Pi_{X_{2}} - \Pi_{X_{1} \cap X_{2}}) f, f \rangle.
\end{align*}
In fact, we can express that $K$ is the canonical completion purely in terms of these projection operators on $\sH(K)$.
\begin{theorem}\label{thm:proj-2-serrated}
	Let $K$ be a completion of a partially reproducing kernel $K_{\Omega}$ on a $2$-serrated domain $\Omega$. The following statements are equivalent.
	\begin{enumerate}
		\item $K = K_{\star}$,
		\item $\bI - \Pi_{X_{1}} - \Pi_{X_{2}} + \Pi_{X_{1} \cap X_{2}} = \bzero$, and
		\item $\Pi_{X_{1} \cap X_{2}} = \Pi_{X_{1}}\Pi_{X_{2}} = \Pi_{X_{2}}\Pi_{X_{1}}$.
	\end{enumerate}
\end{theorem}
\begin{proof} We reason as follows: 
	\begin{enumerate}[leftmargin = 2cm]
		\item[$(1\implies2)$] By Lemma \ref{thm:canonical-inner-product}, we can write $\|f\|^{2}$ as
		\begin{align*}
			\langle f, f \rangle &= \|f_{X_{1}}\|^{2} + \|f_{X_{2}}\|^{2} - \|f_{X_{1} \cap X_{2}}\|^{2}\\
			&= \langle \Pi_{X_{1}} f, f \rangle + \langle \Pi_{X_{2}} f, f \rangle - \langle \Pi_{X_{1} \cap X_{2}} f, f \rangle \\
			&= \langle (\Pi_{X_{1}} + \Pi_{X_{2}} - \Pi_{X_{1} \cap X_{2}}) f, f \rangle.
		\end{align*}
		Thus, $\langle (\bI - \Pi_{X_{1}} - \Pi_{X_{2}} + \Pi_{X_{1} \cap X_{2}}) f, f \rangle$ for $f \in \sH$.
		\item[$(2\implies3)$] Multiplying both sides of the above equation with $\Pi_{X_{1}}$ gives
		\begin{equation*}
			\Pi_{X_{1}} - \Pi_{X_{1}}\Pi_{X_{1}} - \Pi_{X_{1}}\Pi_{X_{2}} + \Pi_{X_{1}}\Pi_{X_{1} \cap X_{2}} = - \Pi_{X_{1}}\Pi_{X_{2}} + \Pi_{X_{1} \cap X_{2}} = \bzero.
		\end{equation*}
		Similarly, we can show $\Pi_{X_{2}}\Pi_{X_{1}} = \Pi_{X_{1} \cap X_{2}}$.
		\item[$(3\implies1)$] For $x \in X_{1} \setminus X_{2}$ and $y \in X_{2} \setminus X_{1}$,
		$K(x, y) = \langle k_{x}, k_{y} \rangle = \langle \Pi_{X_{1}}k_{x}, \Pi_{X_{2}}k_{y} \rangle = \langle \Pi_{X_{2}}\Pi_{X_{1}}k_{x}, k_{y} \rangle = \langle \Pi_{X_{1} \cap X_{2}}k_{x}, k_{y} \rangle = \langle k_{x, X_{1} \cap X_{2}}, k_{y, X_{1} \cap X_{2}} \rangle = K_{\star}(x, y)$. Similarly, $K(x, y) = K_{\star}(x, y)$ for $x \in X_{2} \setminus X_{1}$ and $y \in X_{1} \setminus X_{2}$. 
	\end{enumerate}
	Hence proved.
\end{proof}

\subsubsection{Separation and Inheritance}

It turns out that (\ref{eqn:compln-eqn}) in Theorem \ref{thm:canonical-compln} holds more generally for a separator $S \subset X$ of $x, y$ in $\Omega$, so long as we replace $k_{u, X_{1} \cap X_{2}}$ with $k_{u, S}^{\star}$ for $u = x, y$ and we can write
\begin{equation}\label{eqn:compln-eqn-s}
	K_{\star}(x, y) = \langle k_{x, S}^{\star}, k_{y, S}^{\star} \rangle.
\end{equation}
Note that $S \subset X$ is a separator if and only if $S \subset X_{1} \cap X_{2}$, which is to say $X_{1} \cap X_{2}$ is the minimal separator of $\Omega$. 

There is an alternative way of looking at (\ref{eqn:compln-eqn-s}). 
Consider the partially reproducing kernel $K_{\tilde{\Omega}} = K_{\star}|_{\tilde{\Omega}}$ for the $2$-serrated domain $\tilde{\Omega} = \cup_{j=1}^{2} (S_{j} \times S_{j})$ where $S_{1} = X_{1} \cup S$ and $S_{2} = X_{2} \cup S$. Equation (\ref{eqn:compln-eqn-s}) is now equivalent to saying that the canonical completion of $K_{\tilde{\Omega}}$ is same as $K_{\star}$ and the restriction $K_{\tilde{\Omega}}$ can be said to \emph{inherit} the canonical completion of $K_{\Omega}$. In other words, for any $2$-serrated domain $\tilde{\Omega}$ which contains $\Omega$, the canonical completion of $K_{\star}|_{\tilde{\Omega}}$ is $K_{\star}$. We shall now use this insight to prove (\ref{eqn:compln-eqn-s}).

 
\begin{theorem}\label{thm:canonical-separation-2}
	If $S$ separates $x \in X_{1}$ and $y \in X_{2}$, then 
	$K_{\star}(x, y) = \langle k_{x, S}^{\star}, k_{y, S}^{\star} \rangle$.
\end{theorem}
\begin{proof}
	Define $S_{1} = X_{1} \cup S$ and $S_{2} = X_{2} \cup S$. Let $\bar{\Omega} = \cup_{j=1}^{2}(S_{j}\times S_{j})$ and $K_{\bar{\Omega}} = K_{\star}|_{\bar{\Omega}}$. The proof now reduces to showing that the canonical completion of $K_{\bar{\Omega}}$ is $K_{\star}$. It suffices to show that the associated inner products are equal.	By Theorem \ref{thm:canonical-inner-product}, the inner product associated with the canonical completion of $K_{\bar{\Omega}}$ is
	\begin{equation}\label{eqn:sep-rkhs-norm}
		\|f_{S_{1}}\|^{2} - \|f_{S_{1} \cap S_{2}}\|^{2} + \|f_{S_{2}}\|^{2} = \|f_{S_{1}}\|^{2} - \|f_{S}\|^{2} + \|f_{S_{2}}\|^{2} 
	\end{equation}	
	However, $K_{S_{1}}$ is itself the canonical completion of $K_{\Omega_{1}} = K_{\Omega}|_{\Omega_{1}}$ where $\Omega_{1} = (X_{1} \times X_{1}) \cup (S \times S)$ as can be verifed from Equation (\ref{eqn:compln-eqn}) and therefore, we can write
	\begin{equation}\label{eqn:sep-rkhs-norm1}
		\|f_{S_{1}}\|^{2} = \|f_{X_{1}}\|^{2} - \|f_{X_{1} \cap S}\|^{2} + \|f_{S}\|^{2}
	\end{equation}
	by Theorem \ref{thm:canonical-inner-product}. Using the same reasoning for $\Omega_{2} = (S \times S) \cup (X_{2} \times X_{2})$ and $\Omega_{3} = [(X_{1} \cap S) \times (X_{1} \cap S)] \cup [(X_{2} \cap S) \times (X_{2} \cap S)]$, we can write
	\begin{equation}\label{eqn:sep-rkhs-norm2}
		\|f_{S_{2}}\|^{2} = \|f_{S}\|^{2} - \|f_{X_{2} \cap S}\|^{2} + \|f_{X_{2}}\|^{2}
	\end{equation}
	\begin{equation}\label{eqn:sep-rkhs-norm3}
		\|f_{S}\|^{2} = \|f_{X_{1} \cap S}\|^{2} - \|f_{X_{1} \cap X_{2}}\|^{2} + \|f_{X_{2} \cap S}\|^{2}
	\end{equation}
	Substituting Equations (\ref{eqn:sep-rkhs-norm1}), (\ref{eqn:sep-rkhs-norm2}) and (\ref{eqn:sep-rkhs-norm3}) in the expression (\ref{eqn:sep-rkhs-norm}) yields
	\begin{equation*}
		\|f_{X_{1}}\|^{2} - \|f_{X_{1} \cap X_{2}}\|^{2} + \|f_{X_{2}}\|^{2}
	\end{equation*}
	as desired. The conclusion follows.
\end{proof}

\subsection{Canonical Completion for Serrated Domains}

In the last section, the canonical completion $K_{\star}$ of a $2$-serrated domain was defined by construction. We now give a general definition of the canonical completion of any partially reproducing kernel in terms of separation.

\begin{definition}[Canonical Completion]\label{def:canonical-compln}
	A completion $K_{\star}$ of a partially reproducing kernel $K_{\Omega}$ is called a \emph{canonical completion}, if we have
	\begin{equation*}
		K_{\star}(x, y) = \langle k_{x, S}^{\star}, k_{y, S}^{\star}\rangle
	\end{equation*}
	for every $x, y \in X$ separated by $S \subset X$ in $\Omega$, where $k_{u, U}^{\star}: U \to \bbR$ for $u \in X$ and $U \subset X$ is given by $k_{u, U}^{\star}(v) = K_{\star}(u, v)$.
\end{definition}

Our construction of $K_{\star}$ for a partially reproducing kernel $K_{\Omega}$ on a $2$-serrated domain can be iteratively extended to any serrated domain. Observe that we can extend $K_{\Omega}$ by extending its restriction to $(X_{i} \times X_{i}) \cup (X_{i+1} \times X_{i+1})$ using canonical completion to $(X_{i} \cup X_{i+1}) \times (X_{i} \cup X_{i+1})$. This results in a partially reproducing kernel on a $(m-1)$-serrated domain and continuing the procedure results in a completion of $K_{\Omega}$ to $X \times X$ in $m-1$ steps. It turns out that regardless of the order of the $2$-serrated completions, one always recovers the same completion which is actually the unique canonical completion of $K_{\Omega}$ in the sense of Defintion \ref{def:canonical-compln}. The proof is not straightforward, but using some clever argumentation, we shall now reduce this statement to verifying the separation property for a $2$-serrated domain.

\begin{theorem}[Canonical Completion for Serrated Domains]\label{thm:canonical-serrated}
Let $K_{\Omega}$ be a partially reproducing kernel on a serrated domain $\Omega$ on $X$. Then the following statements hold.
\begin{enumerate}
	\item $K_{\Omega}$ admits a unique canonical completion $K_{\star}$.
	\item If $x \in X_{i}$ and $y \in X_{j}$ for some $1 \leq i < j \leq n$, then 
	\begin{equation*}
		K_{\star}(x, y) = \langle \left[\Phi_{j,j-1}\cdots\Phi_{i+2,i+1}\Phi_{i+1,i}\right]k_{x,X_{i}}, k_{y, X_{j}}\rangle
	\end{equation*} 
	where for $|p - q|=1$, the mapping $\Phi_{p,q}:\sH_{q} \to \sH_{p}$ is given by $\Phi_{p,q} = \Phi_{X_{p},X_{p} \cap X_{q}}\Phi_{X_{p} \cap X_{q},X_{q}}$.
	\item The norm $\|\cdot\|_{\star}$ associated with the canonical completion $K_{\star}$ of $K_{\Omega}$ can be expressed as 
	\begin{equation*}
		\|f\|_{\star}^{2} = \sum_{j=1}^{n}\|f_{X_{j}}\|^{2} - \sum_{j=1}^{n-1} \|f_{X_{j} \cap X_{j+1}}\|^{2} 
	\end{equation*}
	where $\|f_{U}\|$ for $U \subset X$ denotes the norm of $f_{U}$ in $\sH(K_{U})$.
\end{enumerate}
\end{theorem}
\begin{proof}
	Let $\Omega$ be an $m$-serrated domain. We proceed by induction on $m$. The base case $m = 2$ follows from Theorem \ref{thm:canonical-compln} and Theorem \ref{thm:canonical-inner-product}. Assuming that the statement holds for every $m \leq n$ for some $n \geq 2$, we shall show that it holds for the case $m = n+1$. 
	
	Consider an $(n+1)$-serrated domain $\Omega = \cup_{j=1}^{n+1} (X_{j} \times X_{j})$. Let $\bar{X}_{1} = \cup_{j=1}^{n} X_{j}$, $\bar{X}_{2} = X_{n+1}$ and $\Omega_{1n} = \cup_{j=1}^{n} (X_{j} \times X_{j})$, $K_{\Omega_{1n}} = K_{\Omega}|_{\Omega_{1n}}$ and $K_{1n}$ be the unique canonical completion of $K_{\Omega_{1n}}$. It follows from the induction hypothesis that the restriction to $\bar{X}_{1} \times \bar{X}_{1}$ of a canonical completion $K_{\star}$ of $K_{\Omega}$ has to be a canonical completion of $K_{\Omega_{1n}}$ and thus, equal to $K_{1n}$. Define the $2$-serrated domain $\bar{\Omega} = (\bar{X}_{1} \times \bar{X}_{1}) \cup (\bar{X}_{2} \times \bar{X}_{2})$ (see Figure \ref{fig:bar-omega})
	and the partially reproducing kernel $K_{\bar{\Omega}}: \bar{\Omega} \to \bbR$ as
	\begin{align*}
		K_{\bar{\Omega}}(x, y) 
		= \begin{cases}
			K_{1n}(x, y) &\mbox{ if } x,y \in \cup_{j=1}^{n} X_{j} \\
			K_{\Omega}(x, y) &\mbox{ if } x, y \in X_{n+1}.
		\end{cases}	
	\end{align*}
	\begin{figure}[htbp]
		\centering
		\begin{subfigure}{0.49\textwidth}
		\centering
		\begin{tikzpicture}[scale=0.8]
			\pgfmathsetmacro{\a}{8}
			\pgfmathsetmacro{\b}{3}
			\pgfmathsetmacro{\c}{3}
				

			\draw[fill=green!15, dashed] (0, 0) rectangle (\a-1, \a-1);
			\foreach \i in {0,...,5} {
				\draw[fill=red!15] (\i, \i) rectangle ++(\b, \b);								
			} 
			\node at (0+0.5, 0+\b-0.5) {\footnotesize $K_{X_{1}}$};
			\node at (1+0.5, 1+\b-0.5) {\footnotesize $K_{X_{2}}$};
			\node at (2+0.5, 2+\b-0.5) {\footnotesize $K_{X_{3}}$};
			\node at (3+0.5, 3+\b-0.5) {\footnotesize $\reflectbox{$\ddots$}$};
			\node at (4+0.5, 4+\b-0.5) {\footnotesize $K_{X_{n}}$};
			\node at (5+0.7, 5+\b-0.5) {\footnotesize $K_{X_{n+1}}$};

			\draw[fill=blue!15] (5, 5) rectangle ++(\b-1, \b-1);
			\node at (5+1, 5+\b/2-1/2) {\footnotesize $K_{X_{n} \cap X_{n+1}}$};

			\draw (0, 0) rectangle (\a, \a);
		\end{tikzpicture}
		\caption{The coloured region represents $\bar{\Omega}$.}
		\label{fig:bar-omega}
		\end{subfigure}
		\begin{subfigure}{0.49\textwidth}
		\centering
		\begin{tikzpicture}[scale=0.8]
			\pgfmathsetmacro{\a}{8}
			\pgfmathsetmacro{\b}{3}
			\pgfmathsetmacro{\c}{3}				

			\draw[fill=green!15, dashed] (0, 0) rectangle (\a-3, \a-3);
			\draw[fill=green!15, dashed] (3, 3) rectangle (\a, \a);
			\foreach \i in {0,...,2} {
				\draw[fill=red!15] (\i, \i) rectangle ++(\b, \b);								
			} 
			\foreach \i in {5,4,3} {
				\draw[fill=red!15] (\i, \i) rectangle ++(\b, \b);								
			} 
			\node at (0+0.5, 0+\b-0.5) {\footnotesize $K_{X_{1}}$};
			\node at (1+0.5, 1+\b-0.5) {\footnotesize $\reflectbox{$\ddots$}$};
			\node at (2+0.5, 2+\b-0.5) {\footnotesize $K_{X_{j}}$};
			\node at (3+0.7, 3+\b-0.5) {\footnotesize $K_{X_{j+1}}$};
			\node at (4+0.5, 4+\b-0.5) {\footnotesize $\reflectbox{$\ddots$}$};
			\node at (5+0.7, 5+\b-0.5) {\footnotesize $K_{X_{n+1}}$};

			\draw[fill=blue!15] (3, 3) rectangle ++(\b-1, \b-1);
			\node at (3+1, 3+\b/2-1/2) {\footnotesize $K_{X_{j} \cap X_{j+1}}$};

			\draw (0, 0) rectangle (\a, \a);
		\end{tikzpicture}
		\caption{The coloured region represents $\tilde{\Omega}$.}
		\label{fig:tilde-omega}
		\end{subfigure}
		\caption{The red and blue region represents $\Omega$.}
		\label{fig:general-serrated-domain}
	\end{figure}
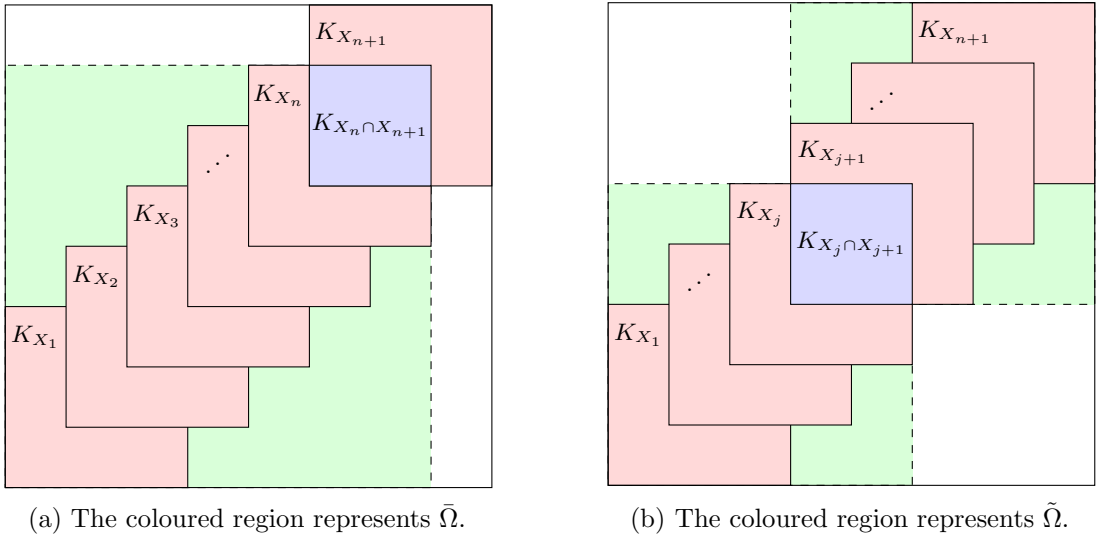

	If $K_{\star}: X \times X \to \bbR$ is a canonical completion of $K_{\bar{\Omega}}$ then for $x \in \cup_{j=1}^{n} X_{j} \setminus X_{n+1}$ and $y \in X_{n+1} \setminus X_{n}$ we must have $K_{\star}(x, y) 
	= \langle \bar{k}_{x, X_{n} \cap X_{n+1}}, \bar{k}_{y, X_{n} \cap X_{n+1}} \rangle$, where $\bar{k}_{x, X_{n} \cap X_{n+1}}(y) = K_{\bar{\Omega}}(x, y)$, by taking $S = X_{n} \cap X_{n+1}$ as the separator in Definition \ref{def:canonical-compln}. It follows that if $K_{\Omega}$ admits a canonical completion, the it must be $K_{\star}$, which is the canonical completion of $K_{\bar{\Omega}}$. Notice that	
	\begin{align*}
		K_{\star}(x, y) 
		&= \langle \bar{k}_{x, X_{n} \cap X_{n+1}}, \bar{k}_{y, X_{n} \cap X_{n+1}} \rangle \\
		&= \langle \Phi_{X_{n} \cap X_{n+1}, X_{n}}\left[\Phi_{n,n-1}\cdots\Phi_{i+1,i}\right]k_{x, X_{i}}, [\Phi_{X_{n} \cap X_{n+1}, X_{n+1}}]k_{y, X_{n+1}} \rangle \\
		&= \langle \Phi_{n+1, n}\left[\Phi_{n,n-1}\cdots\Phi_{i+1,i}\right]k_{x, X_{i}}, k_{y, X_{n+1}} \rangle
	\end{align*}
	for $x \in X_{i}$ and $y \in X_{n+1}$. Furthermore, we can calculate the associated norm $\|\cdot\|_{\star}$ associated with $\sH(K_{\star})$ using Theorem \ref{thm:canonical-inner-product}, as follows 
	\begin{align*}
		\|f\|_{\star}^{2} 
		&= \|f_{\bar{X}_{1}}\|^{2} - \|f_{\bar{X}_{1} \cap \bar{X}_{2}}\|^{2} + \|f_{\bar{X}_{2}}\|^{2} \\
		&= \textstyle \left[ \sum_{j=1}^{n}\|f_{X_{j}}\|^{2} - \sum_{j=1}^{n-1} \|f_{X_{j} \cap X_{j+1}}\|^{2} \right]  - \|f_{X_{n} \cap X_{n+1}}\|^{2} + \|f_{X_{n+1}}\|^{2} \\
		&= \textstyle \sum_{j=1}^{n+1}\|f_{X_{j}}\|^{2} - \sum_{j=1}^{n} \|f_{X_{j} \cap X_{j+1}}\|^{2}
	\end{align*}
	for $f \in \sH(K_{\star})$.

	It remains to be shown that $K_{\star}$ is a canonical completion of $K_{\Omega}$. Let $x, y \in X$ such that they are separated by $S \subset X$ in $\Omega$. Then $x, y$ must also be separated by a minimal separator $X_{i} \cap X_{i+1} \subset S$ for some $1 \leq i \leq n$. Let $\tilde{X}_{1} = \cup_{j=1}^{i} X_{j}$, $\tilde{X}_{2} = \cup_{j=i+1}^{n+1} X_{j}$ and $\tilde{\Omega} = (\tilde{X}_{1} \times \tilde{X}_{1}) \cup  (\tilde{X}_{2} \times \tilde{X}_{2})$ (see Figure \ref{fig:tilde-omega}). Consider the partially reproducing kernel $K_{\tilde{\Omega}}: \tilde{\Omega} \to \bbR$ given by 
	\begin{equation*}
		K_{\tilde{\Omega}}(x, y) = \begin{cases}
			K_{1i}(x, y) & \mbox{ if } x, y \in \tilde{X}_{1} \\
			K_{i,n+1}(x, y) & \mbox{ if } x, y \in \tilde{X}_{2} 
		\end{cases}
	\end{equation*}
	where $K_{1i}$ and $K_{i,n+1}$ are the canonical completions of $K_{\Omega}|_{\tilde{\Omega}_{1}}$ and $K_{\Omega}|_{\tilde{\Omega}_{2}}$ where $\tilde{\Omega}_{1} = \cup_{j=1}^{i} (X_{j} \times X_{j})$ and $\tilde{\Omega}_{2} = \cup_{j=i+1}^{n+1} (X_{j} \times X_{j})$. It is clear that $K_{\star}$ is the canonical completion of $K_{\tilde{\Omega}}$ from the observation 
	\begin{equation*}
		\textstyle \|f\|_{\star}^{2}
		= \sum_{j=1}^{n+1}\|f_{X_{j}}\|^{2} - \sum_{j=1}^{n} \|f_{X_{j} \cap X_{j+1}}\|^{2} 
		= \|f_{\tilde{X}_{1}}\|^{2} - \|f_{\tilde{X}_{1} \cap \tilde{X}_{2}}\|^{2} + \|f_{\tilde{X}_{2}}\|^{2}
	\end{equation*}	
	where
	\begin{align*}
		\textstyle \|f_{\tilde{X}_{1}}\|^{2} &= \textstyle \sum_{j=1}^{i}\|f_{X_{j}}\|^{2} - \sum_{j=1}^{i-1} \|f_{X_{j} \cap X_{j+1}}\|^{2}\\
		\|f_{\tilde{X}_{2}}\|^{2} &= \textstyle  \sum_{j=i+1}^{n+1}\|f_{X_{j}}\|^{2} - \sum_{j=i+1}^{n} \|f_{X_{j} \cap X_{j+1}}\|^{2}
	\end{align*}
	and $\|f_{\tilde{X}_{1} \cap \tilde{X}_{2}}\|^{2} = \textstyle  \|f_{X_{i} \cap X_{i+1}}\|^{2}$ are the quadratic forms associated with the reproducing kernels $K_{1i}$, $K_{i,n+1}$ and $K_{X_{i} \cap X_{i+1}}$. Notice that $S$ can now be thought of as a separator of $x, y \in X$ in the $2$-serrated domain $\tilde{\Omega}$. By Theorem \ref{thm:canonical-separation-2}, we conclude that
	\begin{equation*}
		K_{\star}(x, y) = \langle k_{x, S}^{\star}, k_{y, S}^{\star} \rangle,
	\end{equation*}
	where $k_{u, U}^{\star}: U \to \bbR$ for $u \in X$ and $U \subset X$ is given by $k_{u, U}^{\star}(v) = K_{\star}(u, v)$. Hence proved.
\end{proof}

We can also interpret $K_{\star}$ as the result of sequence of multiple minimum norm interpolations of the kind that appeared in the study of $2$-serrated domains. And similar to $2$-serrated domains, the expression for the norm can be evaluated purely in terms of norms of restrictions of the restrictiosn $K_{X_{j}}$ and $K_{X_{j} \cap X_{j+1}}$ of $K_{\Omega}$. The following result can be proved in the same way as Theorem \ref{thm:proj-2-serrated}. 
\begin{theorem}[Projections]\label{thm:proj-serrated} 
	Let $K$ be a completion of a partially reproducing kernel $K_{\Omega}$ on an $n$-serrated domain $\Omega$. The followings statements are equivalent.
	\begin{enumerate}
		\item $K = K_{\star}$,
		\item $\bI = \sum_{j=i}^{n} \Pi_{X_{j}} - \sum_{j=i}^{n-1} \Pi_{X_{j} \cap X_{j+1}}$, and
		\item for $1 \leq j < n$, $A_{j} = \cup_{k=1}^{j}X_{k}$ and $B_{j} = \cup_{k=j+1}^{n}X_{k}$, $\Pi_{X_{j} \cap X_{j+1}} = \Pi_{A_{j}}^{}\Pi_{B_{j}}^{} =  \Pi_{B_{j}}^{}\Pi_{A_{j}}^{}$
	\end{enumerate}
\end{theorem}

\begin{theorem}
	If $K = K_{\star}$ and $S_{1}, S_{2} \subset X$ such that $S_{1} \cap S_{2}$ separates $S_{1} \setminus S_{2}$ and $S_{2} \setminus S_{1}$, then
	\begin{align}
		\Pi_{S_{1} \cup S_{2}} &= \Pi_{S_{1}} + \Pi_{S_{2}} - \Pi_{S_{1} \cap S_{2}} \\
		\Pi_{S_{1} \cap S_{2}} &= \Pi_{S_{1}}\Pi_{S_{2}} = \Pi_{S_{2}}\Pi_{S_{1}}. \label{eqn:sep-proj-2}
	\end{align}
\end{theorem}
Property (\ref{eqn:sep-proj-2}) is somewhat reminiscent of projection valued measures.

\begin{example}[Matrices with Banded Inverses]
	Let $X = \{1, \dots, m\}$ and $\Omega = \cup_{j=1}^{n} (X_{j} \times X_{j})$ be a serrated domain on $X$. If $K_{\Omega}$ be a partially reproducing kernel on $\Omega$, we can think of it as a matrix $\bA_{\Omega} = [A_{ij}]_{i,j=1}^{m}$ where $A_{ij}$ is unspecified for $(i, j) \notin \Omega$. The inner product associated with a reproducing kernel $K$ on finite $X$ is given by $(\bbf, \bbg) \mapsto \langle \bA^{\dagger} \bbf, \bbg\rangle$ where $\bA$ is the kernel $K$ in matrix form and $\bA^{\dagger}$ denotes the pseudoinverse of $\bA$. Using this fact and the form we have derived for the inner product associated with the canonical completion, we can write the canonical completion of $K_{\Omega}$ in a closed matrix form as
	\begin{equation*}
		\bA^{\dagger} = \bA_{1}^{\dagger} - \bA_{12}^{\dagger} + \bA_{2}^{\dagger} - \cdots - \bA_{n-1,n}^{\dagger}  + \bA_{n}^{\dagger}
	\end{equation*}
	where $\bA_{k}$ is the $m \times m$ matrix with the $(i,j)$th entry $A_{ij}$ for $i,j \in X_{k}$ and $0$ otherwise, and $\bA_{k,k+1}$ is the $m \times m$ matrix with the $(i,j)$th entry $A_{ij}$ for $i,j \in X_{k} \cap X_{k+1}$ and $0$ otherwise. 
\end{example}

\subsection{Canonical Completion for Junction Tree Domains}

The results for the serrated domains on an interval can be extended to a more general setting where the domains are tree-like in a certain sense. A \emph{tree} is an undirected graph in which there exists a unique path connecting every two vertices. As with other graphs, a tree $\sT$ on the set $\{1, \dots, n\}$ can be treated as a subset of $\{1, \dots, n\} \times \{1, \dots, n\}$.  We say that $\Omega$ can be represented as a \emph{junction tree} if there exists for some $n \geq 1$, a tree $\sT$ on $\{1, \dots, n\}$ and subsets $\{X_{j}\}_{j=1}^{n}$ of $X$ such that (a) $X = \cup_{j} X_{j}$, (b) $\Omega = \cup_{j} (X_{j} \times X_{j})$, and (c) for every $(i, j), (j, k) \in \sT$, we have $X_{i} \cap X_{k} \subset X_{j}$ (see Figure \ref{fig:junction-tree-domain}). Essentially, this means that the graph $\Omega$ admits a tree decomposition into cliques \textcite[see][]{diestel2010}.

\begin{figure}[htbp]
\centering
\begin{tikzpicture}[scale=0.7]
	\draw[fill=red!50, opacity=0.5, scale=1.2] plot [smooth cycle, tension=1] coordinates {(0,1) (2, 0) (0,-1) (-2,0)};
	\node at (-1,0.4) {$X_{1}$};
	\draw[shift={(1.6,1.3)}, rotate=50, fill=blue!50, opacity=0.5] plot [smooth cycle, tension=1] coordinates {(0,1) (2, 0) (0,-1) (-2,0)};
	\node at (2.2,2.3) {$X_{2}$};
	\draw[shift={(1.6,-1.3)}, rotate=-50, fill=green!50, opacity=0.5] plot [smooth cycle, tension=1] coordinates {(0,1) (2, 0) (0,-1) (-2,0)};
	\node at (2.2,-2.3) {$X_{3}$};
	\draw[shift={(-1.4,-1.5)}, rotate=-150, fill=yellow!50, opacity=0.5] plot [smooth cycle, tension=1] coordinates {(0,1) (2, 0) (0,-1) (-2,0)};
	\node at (-1.0,-1.8) {$X_{4}$};
	\draw[shift={(-3.4,-1.5)}, rotate=-30, fill=magenta!50, opacity=0.5] plot [smooth cycle, tension=1] coordinates {(0,1) (2, 0) (0,-1) (-2,0)};
	\node at (-4,-1.3) {$X_{5}$};
\end{tikzpicture}
\caption{A junction tree domain $\Omega$ visualized as a graph on $X = \cup_{j=1}^{5}X_{j}$ with the edges given by interpreting the regions $X_{j}$ as cliques. The tree $\sT$ is given by $\{(1, 2), (2, 1), (1,3), (3,1), (1, 4), (4,1), (4, 5), (5,4)\}$.}
\label{fig:junction-tree-domain}
\end{figure}
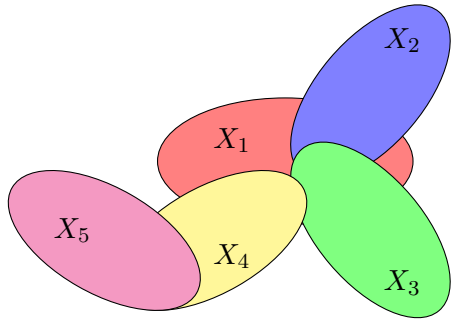

Notice that if we apply the completion formula (\ref{eqn:compln-eqn}) to $X_{i}$ and $X_{j}$ for two adjacent vertices $i,j$ of $\sT$, then we get a partially specified reproducing kernel over a larger domain with a simpler junction tree representation of $n-1$ vertices since $X_{i}$ and $X_{j}$ get replaced by $X_{i} \cup X_{j}$. Iterating the procedure, results in a completion of $K_{\Omega}$.

\begin{theorem}\label{thm:canonical-compln-jtdomains}
	Suppose that $\Omega$ admits a juction tree representation for some $n \geq 1$, a tree $\sT$ on $\{1, \dots, n\}$ and subsets $\{X_{j}\}_{j=1}^{n}$ of $X$. Then the following statements apply.
	\begin{enumerate}
		\item $K_{\Omega}$ admits a canonical completion.
		\item The norm $\|\cdot\|_{\star}$ associated with the canonical completion $K_{\star}$ of $K_{\Omega}$ can be expressed as 
		\begin{equation*}
			\|f\|_{\star}^{2} = \sum_{j} \|f_{X_{j}}\|^{2} - \sum_{(i,j) \in \sT} \|f_{X_{i} \cap X_{j}}\|^{2}
		\end{equation*}
		where $\|f_{U}\|$ for $U \subset X$ denotes the norm of $f_{U}$ in $\sH(K_{U})$.
		\item If $x \in X_{i}$ and $y \in X_{j}$ for some $1 \leq i,j \leq n$, then 
		\begin{equation*}
			K_{\star}(x, y) = \langle \left[\Phi_{ji_{k}}\Phi_{i_{k}i_{k-1}}\cdots\Phi_{i_{1}i}\right]k_{x,X_{i}}, k_{y, X_{j}}\rangle
		\end{equation*} 
		where $(i, i_{1}), (i_{1}, i_{2}), \dots, (i_{k},j) \in \sT$ is the unique path from $i$ to $j$ and for two adjacent vertices $p$ and $q$, the mapping $\Phi_{pq}:\sH_{q} \to \sH_{p}$ is given by $\Phi_{pq} = \Phi_{X_{p},X_{p} \cap X_{q}}\Phi_{X_{p} \cap X_{q},X_{q}}$.
	\end{enumerate}
\end{theorem}
The proof of Theorem \ref{thm:canonical-compln-jtdomains} is very similar to that of Theorem \ref{thm:canonical-serrated} and is hence omitted. Almost all of the following results in this article which apply to serrated domains can be easily generalized to junction tree domains, although we shall refrain from doing so for the sake of simplicity.



\subsection{Dual and Variational Characterization}

Theorem \ref{thm:canonical-serrated} provides an iterative procedure for calculating $K_{\star}$ for a partially reproducing kernel on $K_{\Omega}$ on a serrated domain $K_{\Omega}$. We now provide a more direct and elegant characterization of the canonical completion which relies on the simple form of its associated norm. The key idea is that the quadratic form associated with a reproducing kernel and the square of its associated norm are, in a certain sense, convex conjugates of each other.


\begin{theorem}[Duality Relations]\label{thm:duality-relation}
Let $K$ be a reproducing kernel on $X$ and with the associated Hilbert space $\sH = \sH(K)$ equipped with the norm $\|\cdot\|$. For every $n \geq 1$, $\{\alpha_{i}\}_{i=1}^{n} \subset \bbR$ and $\{x_{i}\}_{i=1}^{n} \subset X$,
\begin{equation}\label{eqn:conjugate}
	\frac{1}{2}\sum_{i,j=1}^{n} \alpha_{i}\alpha_{j}K(x_{i}, x_{j}) = \max_{f} \left[ \sum_{i=1}^{n} \alpha_{i}f(x_{i}) - \frac{1}{2} \|f\|^{2} \right]
\end{equation}
where the maximum is taken over functions $f: X \to \bbR$. Furthermore, for every $f \in \sH$,
\begin{equation}\label{eqn:biconjugate}
	\frac{1}{2} \|f\|^{2} = \sup_{\balpha, \bx} \left[ \sum_{i=1}^{n} \alpha_{i}f(x_{i}) - \frac{1}{2}\sum_{i,j=1}^{n} \alpha_{i}\alpha_{j}K(x_{i}, x_{j}) \right]
\end{equation}
where the supremum is taken over $n \geq 1$, $\balpha = \{\alpha_{i}\}_{i=1}^{n} \subset \bbR$ and $\bx = \{x_{i}\}_{i=1}^{n} \subset X$.
\end{theorem}
\begin{proof}
	Let $g = \sum_{i=1}^{n} \alpha_{i} k_{x_{i}}$. Then $\sum_{i=1}^{n} \alpha_{i}f(x_{i}) = \langle f, g \rangle$ and we can write the right-hand side of Equation \ref{eqn:conjugate} as
	\begin{align*}
		\max_{f} \left[ \langle f, g \rangle - \tfrac{1}{2} \langle f, f \rangle \right]
		&= \tfrac{1}{2}\max_{f} \left[\|g\|^{2} - \|f - g\|^{2}\right] \\
		&= \tfrac{1}{2}\|g\|^{2} - \tfrac{1}{2}\min_{f} \left[ \|f - g\|^{2}\right]\\
		&= \tfrac{1}{2}\|g\|^{2}
	\end{align*}
	which is equal to $\tfrac{1}{2}\sum_{i,j=1}^{n} \alpha_{i}\alpha_{j}K(x_{i}, x_{j})$. Similarly, we can simplify the right-hand side of Equation \ref{eqn:biconjugate} as
	\begin{eqnarray*}
		\sup_{\balpha, \bx} \left[\langle f, g \rangle - \tfrac{1}{2} \|g\|^{2}\right] = \tfrac{1}{2}\|f\|^{2} - \tfrac{1}{2}\inf_{\balpha, \bx} \left[ \|f - g\|^{2}\right] = \tfrac{1}{2}\|f\|^{2}
	\end{eqnarray*}
	since the set of linear combinations $g = \sum_{i=1}^{n} \alpha_{i} k_{x_{i}}$ is dense in $\sH$. Hence proved.
\end{proof}
Given a space of functions and an inner product which under which point evaluations are continuous, Equation (\ref{eqn:conjugate}) reduces the problem of calculating the corresponding reproducing kernel to a calculus of variations problem, thus providing a direct method for calculating the kernel. Indeed, once we calculate the quadratic form $q(\alpha, \beta) = \tfrac{1}{2} \left[ \alpha^{2}K(x, x) + 2 \alpha \beta K(x, y) + \beta^{2} K(y, y) \right]$ using
\begin{equation*}
	q(\alpha, \beta) = \max_{f} \left[ \alpha f(x) + \beta f(y) - \tfrac{1}{2}\|f\|^{2}\right]
\end{equation*}
we have $K(x, y) = \tfrac{1}{4}[q(1, 1) - q(1, -1)]$. Using this observation, the duality relation (\ref{eqn:conjugate}) can be used to derive a very elegant formula for calculating the canonical completion.
\begin{corollary}
Let $K_{\star}$ be the canonical completion of a partially reproducing kernel $K_{\Omega}$ on a serrated (or junction tree) domain $\Omega$. Then,
\begin{equation}
	K_{\star}(x, y) = -\tfrac{1}{2}\left[ K_{\Omega}(x, x) + K_{\Omega}(y, y) \right] + \max_{f} \left[ f(x) + f(y) - \tfrac{1}{2} \|f\|_{\Omega}^{2} \right]
\end{equation}
for $x, y \in X$.
\end{corollary}


The connection to convex analysis also makes obvious many fundamental results in the theory of reproducing kernels. Take for example, the fact that the norm associated with the sum $K_{1} + K_{2}$ of two reproducing kernels $K_{1}$ and $K_{2}$, can be expressed as
\begin{equation*}
	\|f\|^{2} = \inf_{h} \left[ \|f - h\|_{1}^{2} + \|h\|_{2}^{2} \right]
\end{equation*}
where $\|\cdot\|_{1}$ and $\|\cdot\|_{2}$ are the norms associated with $K_{1}$ and $K_{2}$. In light of the previous observation, this can be seen as an corollary to the fact that for two convex functionals $F$ and $G$, the sum of their convex conjugates $F^{\ast} + G^{\ast}$ is equal to the convex conjugate of the infimal convolution $F \square G$ of $F$ and $G$ given by $(F \square G)(x) = \inf_{y} \left[F(x - y) + G(y)\right]$. Indeed, let $g_{i} = \sum_{j=1}^{n} \alpha_{j}k_{x_{j}, i}$ for $n \geq 1$, $\{\alpha_{j}\}_{j=1}^{n} \subset \bbR$ and $\{x_{j}\}_{j=1}^{n} \subset X$, where $k_{x, i}(y) = K_{i}(x, y)$ for $x, y \in X$. The convex conjugate  of $\tfrac{1}{2}\|f\|^{2}$ can then be written as
\begin{align*}
	\textstyle \max_{f} \left[ \sum_{j=1}^{n} \alpha_{j}f(x_{j}) - \tfrac{1}{2} \|f\|^{2} \right]  &= \sup_{f, h} \left[ \langle f-h, g_{1} \rangle_{1} + \langle h, g_{2} \rangle_{2}- \frac{1}{2} \left[ \|f-h\|_{1}^{2} + \|h\|_{2}^{2} \right] \right] \\
	&= \tfrac{1}{2} \sup_{f, h} \left[ \|g_{1}\|_{1}^{2} + \|g_{2}\|_{2}^{2} - \|f-h - g_{1}\|_{1}^{2} - \|h - g_{2}\|_{2}^{2} \right] \\
	&= \tfrac{1}{2}\left[ \|g_{1}\|_{1}^{2} + \|g_{2}\|_{2}^{2} \right] \\
	&= \tfrac{1}{2}\sum_{i,j=1}^{n} \alpha_{i}\alpha_{j}[K_{1}(x_{i}, x_{j}) + K_{2}(x_{i}, x_{j})].
\end{align*}
which implies that the reproducing kernel corresponding to the norm $\|\cdot\|$ is indeed $K_{1} + K_{2}$. Similarly, we can show that the norm corresponding to the subkernel $K_{A} = K|_{A \times A}$ for some kernel $K$ on $X$ and $A \subset X$ is $\|f\| = \inf\nolimits_{h} \{\|h\|: h \in \sH(K) \mbox{ and }h|_{A} = f\}$.

\begin{theorem}[Variational Characterization]\label{thm:variational-char}
Let $K$ be a reproducing kernel on $X$ and with the associated Hilbert space $\sH = \sH(K)$ equipped with the norm $\|\cdot\|$. Then
\begin{equation*}
	k_{x} = \argmin_{f} \left[\|f\|^{2} - 2f(x)\right]
\end{equation*}
where the minimum is taken over all functions $f: X \to \bbR$. Additionally, 
\begin{equation*}
	k_{x} = \argmin \Big\lbrace \|f\|~\Big|~ f: X \to \bbR \mbox{ such that } f(x) = K(x, x) \Big\rbrace.
\end{equation*}
\end{theorem}
\begin{proof}
	Clearly, $\|f\|^{2} - 2f(x) = \|f - k_{x}\|^{2} - K(x,x)$ is minimized for $f = k_{x}$. Hence proved.
\end{proof}

\begin{corollary}
	Let $K_{\star}$ be the canonical completion of a partially reproducing kernel $K_{\Omega}$ on a serrated (or junction tree) domain $\Omega$. Then,
	\begin{align*}
		k_{x}^{\star} &= \argmin_{f} \left[\|f\|_{\Omega}^{2} - 2f(x)\right], \mbox{ and } \\
		k_{x}^{\star} &= \argmin \Big\lbrace \|f\|_{\Omega}~\Big|~ f: X \to \bbR \mbox{ such that } f(x) = K_{\Omega}(x, x) \Big\rbrace
	\end{align*}
	for $x\in X$ where $k_{x}^{\star}(y) = K_{\star}(x, y)$ for $y \in X$.
\end{corollary}
\begin{proof}
	The first characterization is an immediate corollary of Theorem \ref{thm:variational-char}, while the second one follows from the observation that for a reproducing kernel $K$ on $X$ and $x \in X$, the generator $k_{x}$ minimizes the associated norm among all functions $f \in \sH(K)$ satifying $f(x) = K(x, x)$. 
\end{proof}

We conclude this section by demonstrating how Theorem \ref{thm:variational-char} can be used to recover the reproducing kernel of a Hilbert space of functions from the norm.
\begin{example}[Brownian Motion]
	Let $\sH$ be the Sobolev space of functions $f: [0,1] \to \bbR$ such that $f(0) = 0$, $f$ is absolutely continuous and $f' \in L^{2}[0,1]$ equipped with the inner product $\langle f, g \rangle = \int_{0}^{1} f(u)g(u) ~du$. We would like to calculate the corresponding reproducing kernel $K$. Pick $x \in [0, 1]$ and by \emph{Theorem \ref{thm:variational-char}}, we can write
	\begin{equation*}
		k_{x} = \argmin_{f} \left[ \int_{0}^{1} [f'(u)]^{2} ~du - 2f(x)\right].
	\end{equation*}	
	Fix $f(x) = z$ for some $z \in \bbR$. Then the problem reduces to minimizing the integral $\int_{0}^{1} [f'(u)]^{2} ~du$ and the minimum occurs when $f$ is given by the linear interpolation
	\begin{equation*}
		f(u) = \begin{cases}
			[z/x]u & u \in [0, x] \\
			z & u \in [x, 1]
		\end{cases}
	\end{equation*}
	when $x \neq 0$ and $f(u) = 0$ for $u \in [0, 1]$ when $x = 0$. The problem reduces to finding the value of $z$ which minimizes
	\begin{equation*}
		\int_{0}^{1} [f'(u)]^{2} ~du - 2f(x) = \left[\tfrac{z}{x}\right]^{2}x - 2z
	\end{equation*}
	which is when $z = x$! Thus, $k_{x} = f$ with $z = x$, or in other words, $K(x, y) = x \wedge y$, which is the Brownian motion covariance.
\end{example}

\subsection{Vanishing Trace and Determinant Maximization}

We are now in position to establish that the canonical completion is the reproducing kernel counterpart of the determinant maximizing completion whose inverse vanishes outside the specified region which appears in the classical theory of completions of partially specified matrices. The role of the matrix inverse and matrix determinant is played by the trace and the Fredholm determinant of an operator in the reproducing kernel Hilbert space.

\begin{theorem}\label{thm:trace-determinant}
	Let $K_{\star}$ be the canonical completion of a partially reproducing kernel $K_{\Omega}$ on a serrated or junction tree domain $\Omega$ on $X$ and $\Psi: \sH(K_{\star}) \to \sH(K_{\star})$ be a trace-class operator. If $\langle \Psi k_{x}^{\star}, k_{y}^{\star} \rangle = 0$ for $(x, y) \in \Omega$, then
	\begin{enumerate}
		\item $\tr \Psi = 0$, and
		\item $\det (\bI + \Psi) \leq 0$ with equality if and only if $\Psi = \bzero$.
	\end{enumerate}	
\end{theorem}
\begin{proof}
By Theorem \ref{thm:proj-serrated}, we can write $\tr \Psi$ for a serrated domain $\Omega$ as
\begin{align*}
	\textstyle \tr (\bI \Psi)
	&= \textstyle\sum_{j=1}^{n} \tr(\Pi_{X_{i}} \Psi) - \sum_{j=1}^{n-1} \tr (\Pi_{X_{i} \cap X_{i+1}} \Psi) \\
	&= \textstyle\sum_{j=1}^{n} \tr(\Pi_{X_{i}}^{} \Psi \Pi_{X_{i}}^{}) - \sum_{j=1}^{n-1} \tr (\Pi_{X_{i} \cap X_{i+1}}^{} \Psi \Pi_{X_{i} \cap X_{i+1}}^{}).
\end{align*}
Clearly, the operators $\Pi_{X_{i}} \Psi \Pi_{X_{i}}$ for $1 \leq i \leq n$ and $\Pi_{X_{i} \cap X_{i+1}} \Psi \Pi_{X_{i} \cap X_{i+1}}$ for $1 \leq i \leq n-1$ are both zero, since $\langle \Psi k_{x}^{\star}, k_{y}^{\star} \rangle = 0$ for $x, y \in X_{i}$. Therefore, $\tr \Psi = 0$. 

The second conclusion follows from noticing that $\Psi \mapsto -\log \det (\bI + \Psi)$ is a strictly convex function whose Gateux derivative vanishes at $\bzero$ because
	\begin{equation*}
		\frac{d}{dt} \left[ \log \det (\bI + t\Psi) \right]\Big|_{t = 0} = \tr \Psi
	\end{equation*}
and that the operators $\Psi$ form a linear subspace. The proof for junction tree domains is analogous.
\end{proof}

To see the analogy with the matrix setting, it helps to write down the trace $\tr \Psi$ in terms of matrix trace assuming that $X$ is finite. For $X = \{x_{i}\}_{i=1}^{n}$, the statement $\tr \Psi = 0$ (Theorem \ref{thm:trace-determinant} (1)) can be expressed as
\begin{equation*}
	\textstyle \tr \bK^{-1} \bP = \sum_{i,j} (\bK^{-1})_{ij}\bP_{ij} = 0
\end{equation*}
if $\bP_{ij} = 0$ for $(x_{i}, x_{j}) \in \Omega$, where $\bK = [K_{\star}(x_{i}, x_{j})]_{i,j=1}^{n}$ and $\bP = [\bP_{ij}]_{i,j=1}^{n}$ with $\bP_{ij} = \langle \Psi k_{x_{i}}, k_{x_{j}} \rangle$. This implies that $(\bK^{-1})_{ij} = 0$ for $(x_{i}, x_{j}) \notin \Omega$. To make sense of Theorem \ref{thm:trace-determinant} (2) in the same way, we need the following lemma.
\begin{lemma}
	Let $K$ be a strictly positive reproducing kernel and $\Psi: \sH(K) \to \sH(K)$ be a trace-class operator. Consider the nets $\{\bK_{\sF}\}_{\sF}$ and $\{\bH_{\sF}\}_{\sF}$ of matrices indexed by finite subsets $\sF$ of $X$ ordered by inclusion where
	\begin{equation*}
		\bK_{\sF} = [K(x, y)]_{x,y \in \sF} \quad \mbox{ and } \quad \bH_{\sF} = [\langle \Psi k_{x}, k_{y}\rangle]_{x,y \in \sF} \quad \mbox{ for } j \geq 1.
	\end{equation*}
	Then $\lim\nolimits_{\sF} \left[\log \det (\bK_{\sF} + \bH_{\sF}) - \log \det (\bK_{\sF})\right] = \log \det (\bI + \Psi)$.
\end{lemma}
\begin{proof}
	We have by Gr\"{u}mm's Convergence Theorem that $\Pi_{\sF} \Psi \Pi_{\sF} \to \Psi$ in trace norm (see \textcite[Theorem 2.19]{simon2005} and also \textcite[Theorem 3.8]{simon1977}) because $\Pi_{\sF}$ strongly converges to $\bI$ (\textcite[Proposition 3.9]{paulsen2016}). Therefore, $\det (\bI + \Pi_{\sF}\Psi\Pi_{\sF}) \to \det (\bI + \Psi)$. It suffices to show that $\log \det (\bI + \Pi_{\sF}\Psi\Pi_{\sF}) = \log \det (\bK_{\sF} + \bH_{\sF}) - \log \det (\bK_{\sF})$.
	
	For finite rank $\Psi$, we can write $\Psi = \sum_{i=1}^{r}f_{i} \otimes g_{i}$ for $r \geq 1$ and $\{f_{i}\}_{i=1}^{r}, \{g_{i}\}_{i=1}^{r} \subset \sH(K)$. So,
	\begin{align*}
		\log \det (\bI + \Pi_{\sF}\Psi\Pi_{\sF}) = \log \det \left(\bI + \sum_{i=1}^{r}\Pi_{\sF}f_{i} \otimes \Pi_{\sF}g_{i}\right).
	\end{align*}
	By Plemelj's formula,
	\begin{align}\label{eqn:plemlj-1}
		\log \det \left(\bI + \sum_{i=1}^{r}\Pi_{\sF}f_{i} \otimes \Pi_{\sF}g_{i}\right) 
		&= \sum_{k=1}^{\infty} \frac{(-1)^{k-1}}{k} \tr \left\lbrace \left[\sum_{i=1}^{r}\Pi_{\sF}f_{i} \otimes \Pi_{\sF}g_{i}\right]^{n} \right\rbrace
	\end{align}
	where the trace terms can be written as sums of products of the inner products $\langle \Pi_{\sF}f_{i}, \Pi_{\sF}g_{k} \rangle$ which can be simplified as
	\begin{equation*}
		\langle \Pi_{\sF}f_{i}, \Pi_{\sF}g_{k} \rangle = \langle f_{i}|_{\sF}, g_{k}|_{\sF} \rangle = \bbf_{i}^{\top}\bK_{\sF}^{-1} \bbg_{k} = (\bK_{\sF}^{-1/2}\bbf_{i})^{\top}\bK_{\sF}^{-1/2} \bbg_{k},
	\end{equation*}	
	where $\bbf_{i} = [f_{i}(x)]_{x \in \sF}$ and $\bbg_{j} = [g_{k}(x)]_{x \in \sF}$ are to be thought of as column vectors. By working our way backwards with the matrix form, we can rewrite (\ref{eqn:plemlj-1}) as
	\begin{align*}
		\log \det \left(\bI + \sum_{i=1}^{r}\Pi_{\sF}f_{i} \otimes \Pi_{\sF}g_{i}\right) 
		&= \sum_{k=1}^{\infty} \frac{(-1)^{k-1}}{k} \tr \left\lbrace \left[\sum_{i=1}^{r} (\bK_{\sF}^{-1/2}\bbf_{i}) \left(\bK_{\sF}^{-1/2} \bbg_{i}\right)^{\top}\right]^{n} \right\rbrace \\
		&= \log \det \left(\bI + \sum_{i=1}^{r}\Big(\bK_{\sF}^{-1/2}\bbf_{i}\Big) \Big(\bK_{\sF}^{-1/2} \bbg_{i}\Big)^{\top}\right)\\ 
		&= \log \det \left(\bI + \bK_{\sF}^{-1/2} \left[\sum_{i=1}^{r}\bbf_{i} \bbg_{i}^{\top}\right]\bK_{\sF}^{-1/2}\right) \\
		&= \log \det \left(\bI + \bK_{\sF}^{-1/2} \bH_{\sF}^{} \bK_{\sF}^{-1/2}\right) \\
		&= \log \det (\bK_{\sF} + \bH_{\sF}) - \log \det (\bK_{\sF}).
	\end{align*}
	Even if $\Psi$ is not finite rank, we can approximate by finite rank operators in trace norm and the conclusion follows from the continuity of the Fredholm determinant in trace norm.	
\end{proof}

Roughly speaking, Theorem \ref{thm:trace-determinant} (2) seems to say that trace-class self-adjoint perturbations of the canonical solution tend to decrease the determinant.

\subsection{Canonical Completion for Regular Domains}

In this section, we shall study the problem of canonical completion for a different class of domains which can be thought of as the limit of a sequence of serrated domain.
\begin{definition}[Regular Domain]\label{def:regular-domains}
	Let $X = [0, 1] \subset \bbR$. We say that a domain $\Omega$ on $X$ is a \emph{regular domain} if we can write 
	\begin{equation*}
		\Omega = \cup_{t \in T} (I_{t} \times I_{t})
	\end{equation*}
	where $T = [0, t_{0}]$ for some $t_{0} \in (0, 1)$ and $I_{t} = [t, b(t)]$ for a strictly increasing function $b: [0, w] \to \bbR$ satisfying $b(t) > t$ for $t \in T$ and $b(w) = 1$.
\end{definition}
Regular domains are particularly nice in that $I_{t}$ for $t \in (0, t_{0})$ are all minimal separators of $\Omega$. Note that by appropriately rescaling $X$, we can make any regular domain $\Omega$ in to the \emph{band} $\{(x, y): |x-y| \leq a\}$ for some $a > 0$. Thus regular domains are domains which are equivalent to the band. 

We shall prove the existence of a canonical completion $K_{\star}$ of every partially reproducing kernel $K_{\Omega}$ on a regular domain $\Omega$. Roughly speaking, our proof relies on approximating $\Omega$ with a sequence of serrated domains $\Omega_{j} \subset \Omega$ and a canonical completion $K_{\star}$ as a limit of the canonical completions $K_{j}$ of the partially reproducing kernels $K_{\Omega_{j}} = K_{\Omega}|_{\Omega_{j}}$ on the serrated domains $\Omega_{j}$. In spite of their simple definition, the norms $\|\cdot\|_{\Omega_{j}}$ we do not have much insight into their limiting behaviour as $j \to \infty$. We circumvent this problem by relying on sequential compactness properties of the sequence $\|\cdot\|_{\Omega_{j}}$ under a special notion of convergence known as $\Gamma$-convergence or epiconvergence (\textcite{DalMaso2012, Braides2002}). 

The opaque nature of our construction makes it difficult to show that the canonical completion is unique, although our experience with serrated domains is a compelling reason to believe that this is certainly the case. Regardless, we are still able to derive an interesting algebraic characterization of canonical completion in terms of semigroupoids of contraction maps. In the next section, we establish the uniqueness of canonical completion for some stationary partially reproducing kernels.

\subsubsection{$\Gamma$-Convergence in Separable Hilbert Spaces}

As a result, we are forced to rely on general properties of the norm, such as the fact that $\|f\|^{2}$ is a quadratic form on the space of functions $f: X \to \bbR$. Let $\bar{\bbR}$ denote $\bbR \cup \{\infty\}$. 
\begin{definition}[$\Gamma$-Convergence in Hilbert Space]
	Let $\sX$ be a Hilbert space. We say that a sequence $\{\Lambda_{j}\}_{j=1}^{\infty}$ of functionals $\Lambda_{j}:\sX \to \bar{\bbR}$ converges in the $\Gamma$ sense or simply, $\Gamma$-converges to $\Lambda: \sX \to \bar{\bbR}$ if for every $f \in \sX$ we have:
\begin{enumerate}
	\item for every $\{f_{j}\}_{j=1}^{\infty} \subset \sX$, $\lim_{j\to \infty} f_{j} = f$ implies $\Lambda(f) \leq \liminf_{j} \Lambda_{j}(f_{j})$, and
	\item there exists $\{f_{j}\}_{j=1}^{\infty} \subset \sX$ such that $\lim_{j\to \infty} f_{j} = f$ and $\Lambda(f) \geq \limsup_{j} \Lambda_{j}(f_{j})$.
\end{enumerate}
If $\Lambda_{j}$ $\Gamma$-converges to $\Lambda$, we write $\glim_{j} \Lambda_{j} = \Lambda$.
\end{definition}

$\Gamma$-convergence is quite different from other modes of convergence with which the reader may be familiar. Perhaps most strikingly, even the limit $\glim_{j} \Lambda_{j}$ of a constant sequence $\Lambda_{j} = \Lambda$ is not necessarily equal to $\Lambda$ unless $\Lambda$ is \emph{lower semicontinuous}. A functional $\Lambda: \sX \to \bbR$ is said to be lower semicontinuous if for every $f \in \sX$ and $\{f_{j}\}_{j=1}^{\infty} \subset \sX$ such that $f_{j} \to f$, we have $\liminf_{j} \Lambda(f_{j}) \geq \Lambda(f)$ (\textcite[Remark 1.8]{Braides2002}). Note that $\Gamma$-limits themselves are always lower semicontinuous (\textcite[Proposition 6.8]{DalMaso2012}).

Somewhat unsurprisingly, the limit $\glim_{j} \Lambda_{j}$ is not necessarily equal to the pointwise limit $f \mapsto \lim_{j} \Lambda_{j}(f)$. Thankfully, many intuitive properties continue to hold, if only under certain conditions. Let $\{\Lambda_{j}\}_{j=1}^{\infty}$ be sequence of functionals on $\sX$ such that $\glim_{j} \Lambda_{j} = \Lambda$. For a continuous and increasing function $\varphi:\bar{\bbR} \to \bar{\bbR}$ and  then $\glim_{j} \varphi \circ \Lambda_{j} = \varphi \circ \Lambda$ (\textcite[Proposition 6.16]{DalMaso2012}). We also have \emph{monotonicity}. Let $\{\tilde{\Lambda}_{j}\}_{j=1}^{\infty}$ be another sequence of functionals $\tilde{\Lambda}_{j}:\sX \to \bar{\bbR}$ such that $\glim_{j} \tilde{\Lambda}_{j} = \tilde{\Lambda}$. If $\Lambda_{j}(f) \leq \tilde{\Lambda}_{j}(f)$ for $f \in \sX$ and $j \geq 1$, then $\Lambda \leq \tilde{\Lambda}$ (\textcite[Proposition 6.7]{DalMaso2012}). Naturally, this means that if $\bar{\Lambda}(f) \leq \tilde{\Lambda}_{j}(f)$ for $f \in \sX$ and $j \geq 1$ for some lower semicontinuous functional $\bar{\Lambda}: \sX \to \bar{\bbR}$, then $\bar{\Lambda} \leq \tilde{\Lambda}$. Furthermore, $\Gamma$-limits are \emph{superadditive}, in that 
\begin{equation*}
	\Lambda + \tilde{\Lambda} \leq \glim_{j} \left[\Lambda_{j} + \tilde{\Lambda}_{j}\right]
\end{equation*}
so long as the limit $\glim_{j} \left[\Lambda_{j} + \tilde{\Lambda}_{j}\right]$ exists and the sums $\Lambda_{j} + \tilde{\Lambda}_{j}$ and $\Lambda + \tilde{\Lambda}$ are well-defined, in the sense that for no point in $\sX$ is one of the functionals in the sum equal to $\infty$ when the other is $-\infty$. 

Interestingly, $\Gamma$-convergence is sequentially compact on second-countable spaces that is, every sequence $\{\Lambda_{j}\}_{j=1}^{\infty}$ of functionals $\Lambda_{j}:\sX \to \bar{\bbR}$ has a $\Gamma$-convergent subsequence (\textcite[Theorem 8.5]{DalMaso2012}).
Note that because separability and second-countability are equivalent for metric spaces, this also holds true for our setting of separable Hilbert spaces. We shall use this property to construct our canonical completion from a sequence of canonical completions on serrated domains. To this end, we shall need another importantly property of $\Gamma$-convergence, which is that the $\Gamma$-limit of non-negative quadratic forms is a non-negative quadratic form (\textcite[Theorem 11.10]{DalMaso2012}).

\subsubsection{Existence of Canonical Completion on Regular Domains} 
\label{sec:existence-regular-proof}
In this section, we shall establish the existence of the canonical completion for partially reproducing kernel on regular domains. 

\begin{theorem}
	Let $K_{\Omega}$ be a partially reproducing kernel on a regular domain $\Omega$. Then $K_{\Omega}$ admits a canonical completion.
\end{theorem}

Let $K_{\Omega}$ be a continuous partially reproducing kernel on a regular domain $\Omega$ on $[0, a]$ with $\Omega = \cup_{t \in T} (I_{t} \times I_{t})$ where $T = [0, t_{0}]$ as in Definition \ref{def:regular-domains}. Define an increasing sequence $\{\Omega_{j}\}_{j=1}^{\infty}$ of serrated domains on $[0, a]$ by $\Omega_{j} = \cup_{t \in T_{j}} (I_{t} \times I_{t})$ where $T_{j} = \{\tfrac{i}{2^{j}}t_{0}: 0 \leq i \leq 2^{j}\}$. To retain the serrated domain notation, we denote $X_{ij} = I_{t_{ij}}$ where $t_{ij} = \tfrac{i-1}{2^{j}}t_{0}$, and now we can write $\Omega_{j} = \cup_{i=1}^{m_{j}} (X_{ij} \times X_{ij})$ where $m_{j} = 2^{j}+1$.



For $j \geq 1$, let $K_{j}$ denote the canonical completion of the restriction $K_{\Omega_{j}} = K_{\Omega}|_{\Omega_{j}}$ with the associated Hilbert space $\sH_{j} = \sH(K_{j})$. Define the quadratic forms $\Lambda_{j}$ as the squares of the associated norms $\|f\|_{j}$ of $\sH_{j}$:
\begin{equation*}
	\textstyle \Lambda_{j}(f) = \|f\|_{j}^{2} = \sum_{i=1}^{m_{j}} \|f_{X_{ij}}\|^{2} - \sum_{i=1}^{m_{j}-1} \|f_{X_{ij} \cap X_{i+1,j}}\|^{2}
\end{equation*}
where $\|f\|_{j}$ is the norm associated with $\sH_{j}$. 

Let $\sX = \sH_{1}$ and $\|\cdot\| = \|\cdot\|_{1}$. We shall treat $\Lambda_{j}$ as functionals on $\sX$. Notice that $\{\Lambda_{j}\}_{j=1}^{\infty}$ is \emph{equicoercive} in that $\frac{1}{m_{1}}\|f\|^{2} \leq \Lambda_{j}(f)$ for every $f \in \sX$ and $j \geq 1$, since
\begin{equation}\label{eqn:equicoercivity}
	\textstyle \frac{1}{m_{1}}\|f\|^{2} \leq \frac{1}{m_{1}}\sum_{i=1}^{m_{1}} \|f_{X_{i1}}\|^{2} \leq \frac{1}{m_{1}}\sum_{i=1}^{m_{1}} \|f\|_{j}^{2} = \|f\|_{j}^{2} = \Lambda_{j}(f)
\end{equation}
as $X_{i1} \times X_{i1} \subset \Omega_{j}$ for every $1 \leq i \leq m_{1}$. Moreover, note that $K_{1}$ is continuous because $K_{\Omega_{1}}$ is continuous. This implies that $\sX = \sH(K_{1})$ is a second-countable space, since continuous kernels induce separable Hilbert spaces and for Hilbert spaces, separability and second-countability are equivalent. 

By the sequential compactness property of $\Gamma$-convergence (\textcite[Theorem 8.5]{DalMaso2012}), there exists a subsequence $\{j_{k}\}_{k=1}^{\infty} \subset \{j\}_{j=1}^{\infty}$ such that $\glim_{k}\Lambda_{j_{k}} = \Lambda$ for some lower-semicontinuous functional $\Lambda$. Because every $\Lambda_{j_{k}}$ is a non-negative quadratic form so is $\Lambda$ (\textcite[Theorem 11.10]{DalMaso2012}). Define $\sH_{\Lambda} = \{f \in \sX: \Lambda(f) < \infty\}$. By virtue of being a non-negative quadratic form, $\Lambda(f)$ defines an inner product $\langle \cdot, \cdot \rangle_{\Lambda}$ on $\sH_{\Lambda}$ given by
\begin{equation*}
	\langle f, g \rangle_{\Lambda} = \tfrac{1}{4}[\Lambda(f + g) - \Lambda(f - g)].
\end{equation*}
The inner product induces the norm $\|f\|_{\Lambda} = \sqrt{\Lambda(f)}$ on $\sH_{\Lambda}$. 

\begin{lemma}
	The space $\sH_{\Lambda}$ equipped with the inner product $\langle \cdot, \cdot \rangle_{\Lambda}$ is a reproducing kernel Hilbert space.
\end{lemma}
\begin{proof}
	$\sH_{\Lambda}$ is clearly an inner product space. We need to show that $\sH_{\Lambda}$ is complete with respect to the norm $\|\cdot\|_{\Lambda}$. Let $\{f_{j}\}_{j=1}^{\infty} \subset \sX$ be a $\|\cdot\|_{\Lambda}$-Cauchy sequence. This implies that $\{\|f_{j}\|_{\Lambda}\}_{j=1}^{\infty}$ is bounded. Furthermore, because (\ref{eqn:equicoercivity}) we can conclude using Proposition 6.7 of \textcite{DalMaso2012} that
	\begin{equation*}
		\frac{1}{m_{1}}\|f\|^{2} \leq \|f\|_{\Lambda}^{2} = \Lambda(f)
	\end{equation*}
	implying that $\{f_{j}\}_{j=1}^{\infty}$ is also Cauchy with respect to $\|\cdot\|$ and therefore it must converge to some $f \in \sX$. We conclude from the lower semicontinuity of $\Lambda$ that $\Lambda(f) \leq \lim_{j\to \infty} \|f_{j}\|_{\Lambda} < \infty$ and thus, $f \in \sH_{\Lambda}$. Notice that $f_{j} - f_{i} \to f_{j} - f$ in $\sX$ as $i \to \infty$, so we can write again using  the lower semicontinuity of $\Lambda$ that $\|f_{j} - f\|_{\Lambda} \leq \lim_{i \to \infty} \|f_{j} - f_{i}\|_{\Lambda}$ for every $j \geq 1$. Taking the limit as $j \to \infty$ gives 
	\begin{equation*}
		\lim_{j \to \infty}\|f_{j} - f\|_{\Lambda} \leq \lim_{j,i \to \infty}\|f_{j} - f_{i}\|_{\Lambda} = 0.
	\end{equation*}
	Thus $f_{j} \to f$ in the norm $\|\cdot\|_{\Lambda}$ and $\sH_{\Lambda}$ is a Hilbert space. To show that $\sH_{\Lambda}$ is a reproducing kernel Hilbert space we need only observe that for every $f \in \sH_{j_{k}} \subset \sX$ and $x \in X$, 
	\begin{equation*}
		|f(x)| = |\langle f, k_{x, j_{k}} \rangle| \leq \|f\|_{j_{k}}\|k_{x,j_{k}}\|_{j_{k}} = \sqrt{K_{\Omega}(x, x)} \cdot \|f\|_{j_{k}} =  \sqrt{K_{\Omega}(x, x) \cdot\Lambda_{j_{k}}(f)}
	\end{equation*}
	by Cauchy-Schwarz inequality, where $k_{x,j_{k}}(y) = K_{j_{k}}(x, y)$ for $y \in X$. Proposition 6.7 of \textcite{DalMaso2012} allows us to conclude that the same is true for $\Lambda$, and thus
	\begin{equation*}
		|f(x)| \leq \sqrt{K_{\Omega}(x, x)} \cdot \|f\|_{\Lambda} =  \sqrt{K_{\Omega}(x, x) \cdot\Lambda(f)}
	\end{equation*}
	which implies that point evaluations are continuous in $\sH_{\Lambda}$. Hence proved.
\end{proof}

Let $K_{\Lambda}$ denote the reproducing kernel of $\sH_{\Lambda}$.
\begin{lemma}\label{lem:lambda-pw}
	$K_{\Lambda}$ is a completion of $K_{\Omega}$ and $\lim_{k\to\infty} K_{j_{k}}(x, y) = K_{\Lambda}(x, y)$ for every $x,y \in X$.
\end{lemma}
\begin{proof}
Pick $x, y \in X$. Proposition 6.21 of \textcite{DalMaso2012} tells us that if $G$ is a continuous functional on $\sX$, then $\glim_{j\to\infty}\Lambda_{j} = \Lambda$ implies $\glim_{j\to\infty}\Lambda_{j}+G = \Lambda+G$. Notice that the point evaluations $f \mapsto f(u)$ for $u \in X$ are continuous linear functionals on $\sX$. We deduce from $\glim_{k\to\infty}\Lambda_{j_{k}} = \Lambda$ that $\glim_{k\to\infty}(-\tilde{\Lambda}_{j_{k}}) = (-\tilde{\Lambda})$, where 
\begin{equation*}
	\tilde{\Lambda}_{j_{k}}(f) = f(x) + f(y) - \frac{1}{2}\|f\|_{j_{k}}^{2} \quad\mbox{ and } \quad \tilde{\Lambda}(f) = f(x) + f(y) - \frac{1}{2}\|f\|_{\Lambda}^{2}.
\end{equation*}
Because $\{\Lambda_{j_{k}}\}_{k=1}^{\infty}$ is equicoercive (\ref{eqn:equicoercivity}), we can conclude using Theorem 7.8 of \textcite{DalMaso2012}, that $\min (-\tilde{\Lambda}_{j_{k}}) \to \min (-\tilde{\Lambda})$ as $k \to \infty$ or alternatively, $\max \tilde{\Lambda}_{j_{k}} \to \max \tilde{\Lambda}$ as $k \to \infty$. By Theorem \ref{thm:duality-relation}, we have
\begin{equation*}
	\tfrac{1}{2}\left[K_{\Omega}(x, x) + K_{\Omega}(y, y) + 2 K_{j_{k}}(x, y)\right] \to \tfrac{1}{2} \left[K_{\Omega}(x, x) + K_{\Omega}(y, y) + 2 K_{\Lambda}(x, y)\right]
\end{equation*}
or $K_{j_{k}}(x, y)  \to K_{\Lambda}(x, y)$ as $k \to \infty$. Thus the kernel $K_{\Lambda}$ of $\sH_{\Lambda}$ is actually the pointwise limit of the kernels $K_{j_{k}}$ of $\sH_{j_{k}}$. Notice that $K_{\Lambda}$ is a completion of $K_{\Omega}$. Indeed, for $(x, y)$ in the interior of $\Omega$, $\lim_{k\to\infty} K_{j_{k}}(x, y) = K_{\Omega}(x, y)$ since $K_{j_{k}}(x, y) = K_{\Omega}(x, y)$ for large enough $k$. For $(x, y) \in \Omega$ which lie on the boundary of $\Omega$, the same conclusion follows from the continuity of the kernels.
\end{proof}

\begin{lemma}
	$K_{\Lambda}$ is a canonical completion of $K_{\Omega}$.
\end{lemma}
\begin{proof}
It suffices to show that the separation property is satisfied for minimal separators, for otherwise we can argue as in Theorem \ref{thm:canonical-serrated}. Let $S = X_{pq}$. Then $S$ is a minimal separator of $\Omega$ which separates $\Omega$ into two connected components $Y_{1}, Y_{2} \subset X$. Define $S_{1} = S \cup Y_{1}$ and $S_{2} = S \cup Y_{2}$ (see Figure \ref{fig:S-is-Xpq}). Note that $S \times S \subset \Omega_{j_{k}}$ for $j_{k} \geq q$. Consider the sequence $\{\Lambda_{j_{k}}\}$ for $j_{k} \geq q$. 

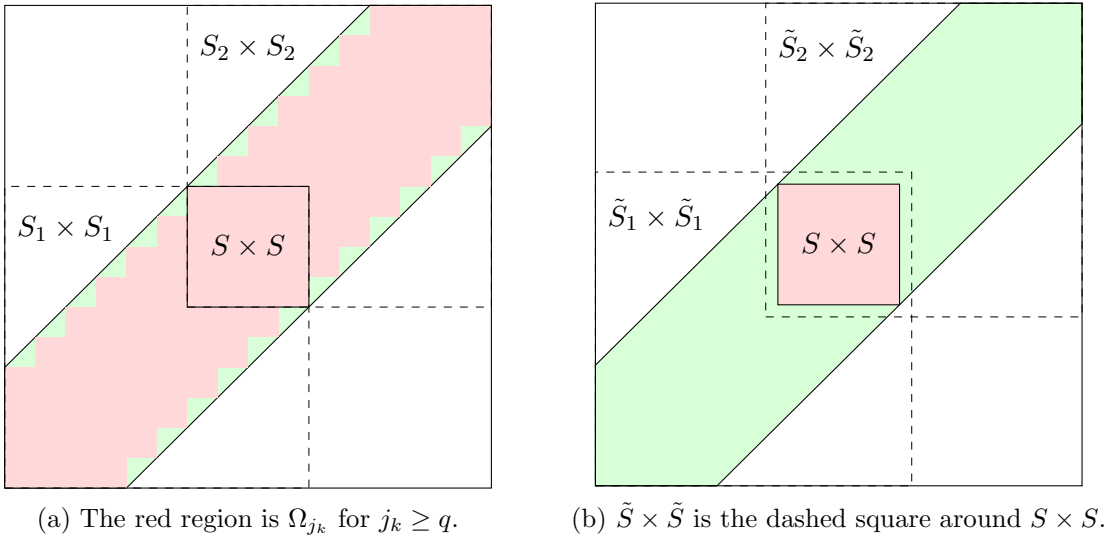
\begin{figure}[htbp]
\centering
\begin{subfigure}{0.48\textwidth}
\centering
\begin{tikzpicture}[scale=0.8]
	\pgfmathsetmacro{\a}{8}
	\pgfmathsetmacro{\b}{2}
	\pgfmathsetmacro{\c}{3}

	\draw[fill=green!15] (0, 0) -- (\b, 0) -- (\a, \a-\b) -- (\a, \a) -- (\a-\b, \a) -- (0, \b) -- cycle;

	\foreach \i in {0,...,12} {
		\fill[red!15] (\i/2, \i/2) rectangle ++(\b, \b);
    } 
	\draw (0, 0) rectangle (\a, \a);
	\draw[fill=red!15] (\c, \c) rectangle ++(\b, \b);
	\node at (\c+\b/2, \c+\b/2) {$S \times S$};
	\draw[dashed, line width = 0.1 mm] (0, 0) rectangle ++(\c+\b, \c+\b);
	\node at (1, \c+\b-0.7) {$S_{1} \times S_{1}$};
	\draw[dashed, line width = 0.1 mm] (\c, \c) rectangle (\a, \a);
	\node at (\c+1, \a-0.7) {$S_{2} \times S_{2}$};
\end{tikzpicture}
\caption{The red region is $\Omega_{j_{k}}$ for $j_{k} \geq q$.}
\label{fig:S-is-Xpq}
\end{subfigure}
\begin{subfigure}{0.48\textwidth}
\centering
\begin{tikzpicture}[scale=0.8]
	\pgfmathsetmacro{\a}{8}
	\pgfmathsetmacro{\b}{2}
	\pgfmathsetmacro{\c}{3}
	\pgfmathsetmacro{\ba}{2.4}
	\pgfmathsetmacro{\ca}{2.8}

	\draw[fill=green!15] (0, 0) -- (\b, 0) -- (\a, \a-\b) -- (\a, \a) -- (\a-\b, \a) -- (0, \b) -- cycle;

	\draw (0, 0) rectangle (\a, \a);
	\draw[fill=red!15] (\c, \c) rectangle ++(\b, \b);
	\node at (\c+\b/2, \c+\b/2) {$S \times S$};

	\draw[dashed, line width = 0.1 mm] (0, 0) rectangle ++(\ca+\ba, \ca+\ba);
	\node at (1, \ca+\ba-0.7) {$\tilde{S}_{1} \times \tilde{S}_{1}$};
	\draw[dashed, line width = 0.1 mm] (\ca, \ca) rectangle (\a, \a);
	\node at (\ca+1, \a-0.7) {$\tilde{S}_{2} \times \tilde{S}_{2}$};
\end{tikzpicture}
\caption{$\tilde{S} \times \tilde{S}$ is the dashed square around $S \times S$.}
\label{fig:S-is-minimal-separator}
\end{subfigure}
\caption{Canonical Completion. The colored regions represent $\Omega$.}
\label{fig:regular-domain}
\end{figure}

The norm of the function $f_{S_{1}}$ in $\sH(K_{j_{k}}|_{S_{1} \times S_{1}})$ can be expressed as
\begin{equation*}
	\|f_{S_{1}}\|_{j_{k}, S_{1}}^{2} = \sum_{i=1}^{p_{j_{k}}} \|f_{X_{ij_{k}}}\|^{2} - \sum_{i=1}^{p_{j_{k}}-1} \|f_{X_{ij_{k}} \cap X_{i+1,j_{k}}}\|^{2} 
\end{equation*}
where $p_{j_{k}}$ is given by $X_{p_{j_{k}},j_{j}} = X_{pq}$, that is $p_{j_{k}} = 1 + (p-1)2^{j_{k} - q}$. Similarly, the norm of the function $f_{S_{2}}$ in $\sH(K_{j_{k}}|_{S_{2} \times S_{2}})$ is 
\begin{equation*}
	\|f_{S_{2}}\|_{j_{k}, S_{2}}^{2} = \sum_{i=p_{j_{k}}}^{m_{j_{k}}} \|f_{X_{ij_{k}}}\|^{2} - \sum_{i=p_{j_{k}}}^{m_{j_{k}}-1} \|f_{X_{ij_{k}} \cap X_{i+1,j_{k}}}\|^{2}. 
\end{equation*}
Because $\tilde{\Omega} = (S_{1} \times S_{1}) \cup (S_{2} \times S_{2}) \supset \Omega_{j_{k}}$, we can write using the inheritance property of canonical completion that
\begin{equation*}
	\|f\|_{j_{k}}^{2} = \|f_{S_{1}}\|_{j_{k}, S_{1}}^{2} + \|f_{S_{2}}\|_{j_{k}, S_{2}}^{2} - \|f_{S}\|_{j_{k}, S}^{2} 
\end{equation*} 
where $f_{S_{1}} = f|_{S_{1}}$, $f_{S_{2}} = f|_{S_{2}}$ and of course, $f_{S} = f|_{S}$. Notice that $S \times S \subset \Omega_{j_{k}}$, so $\|f_{S}\|_{j_{k}, S} = \|f_{S}\|_{S}$. Since $f \mapsto \|f_{S}\|_{S}^{2}$ is a continuous function on $\sX$, we get by taking the $\Gamma$-limit of both sides that
\begin{eqnarray*}
	&\|f\|_{\Lambda}^{2} 
	&= \glim_{k\to \infty} \left[ \|f_{S_{1}}\|_{j_{k}, S_{1}}^{2} + \|f_{S_{2}}\|_{j_{k}, S_{2}}^{2} \right] - \|f_{S}\|_{S}^{2} \\
	&&\geq \glim_{k\to \infty} \|f_{S_{1}}\|_{j_{k}, S_{1}}^{2} + \glim_{k\to \infty} \|f_{S_{2}}\|_{j_{k}, S_{2}}^{2} - \|f_{S}\|_{S}^{2}
\end{eqnarray*}
because we taking the limit of a sum with a continuous functional $- \|f_{S}\|_{S}^{2}$ (\textcite[Proposition 6.21]{DalMaso2012}) and superadditivity of the $\Gamma$-limits of $\|f_{S_{1}}\|_{j_{k}, S_{1}}^{2}$ and $\|f_{S_{2}}\|_{j_{k}, S_{2}}^{2}$ (\textcite[Proposition 6.7]{DalMaso2012}). Furthermore, 
\begin{equation*}
	\glim_{k\to \infty} \|f_{S_{1}}\|_{j_{k}, S_{1}}^{2} = \|f_{S_{1}}\|_{\Lambda, S_{1}}^{2} \quad\mbox{ and }\quad  \glim_{k\to \infty} \|f_{S_{2}}\|_{j_{k}, S_{2}}^{2} = \|f_{S_{2}}\|_{\Lambda, S_{2}}^{2}
\end{equation*}
where $\|f_{S_{1}}\|_{\Lambda, S_{1}}^{2}$ and $\|f_{S_{2}}\|_{\Lambda, S_{2}}^{2}$ denote the norms of $f_{S_{1}}$ in $\sH(K_{\Lambda}|_{S_{1} \times S_{1}})$ and $f_{S_{2}}$ in $\sH(K_{\Lambda}|_{S_{2} \times S_{2}})$ respectively. To see why, note that using the same arguments as before, we can show that every subsequence of $f_{S_{1}} \to \|f_{S_{1}}\|_{j_{k}, S_{1}}^{2}$ will admit a $\Gamma$-convergent subsequence which converges to the square of the norm of a completion of $K_{\Omega}|_{S_{1} \times S_{1}}$ to $S_{1} \times S_{1}$ which is the pointwise limit of the corresponding subsequence of the completions $K_{j_{k}}|_{S_{1} \times S_{1}}$ corresponding to the norms $\|\cdot\|_{j_{k}, S_{1}}^{2}$. But $K_{j_{k}}|_{S_{1} \times S_{1}}$ converges pointwise to $K_{\Lambda}|_{S_{1} \times S_{1}}$ by Lemma \ref{lem:lambda-pw}, so it follows that all such subsequences of norms squared converge to $\|f_{S_{1}}\|_{\Lambda, S_{1}}^{2}$ implying that $\glim_{k\to \infty} \|f_{S_{1}}\|_{j_{k}, S_{1}}^{2} = \|f_{S_{1}}\|_{\Lambda, S_{1}}^{2}$ by the Urysohn property (\textcite[Proposition 8.3]{DalMaso2012}) of $\Gamma$-convergence which states that if every subsequence of a sequence $\Gamma$-converges to the same limit then the sequence $\Gamma$-converges to that limit. Similarly, we can show that same for $\|f_{S_{2}}\|_{j_{k}, S_{2}}^{2}$. It follows that
\begin{equation*}
	\|f\|_{\Lambda}^{2} \geq \|f_{S_{1}}\|_{\Lambda, S_{1}}^{2} + \|f_{S_{2}}\|_{\Lambda, S_{2}}^{2} - \|f_{S}\|_{\Lambda, S}^{2}.
\end{equation*}
The expression of the right hand side is actually the squared norm of the canonical completion $\tilde{K}$ of the restriction of $K_{\Lambda}$ to $\tilde{\Omega} = S_{1}^{2} \cup S_{2}^{2}$. Taking the convex conjugates of the two sides gives for every $n\geq 1$, $\{\alpha_{i}\}_{i=1}^{n} \subset \bbR$ and $\{x_{i}\}_{i=1}^{n} \subset X$ that
\begin{equation*}
	\sum_{i, j=1}^{n} \alpha_{i}\alpha_{j} K_{\Lambda}(x_{i}, x_{j}) \leq \sum_{i, j=1}^{n} \alpha_{i}\alpha_{j} \tilde{K}(x_{i}, x_{j})
\end{equation*}
which implies that the difference $\tilde{K} - K_{\Lambda}$ is a reproducing kernel. But $(\tilde{K} - K_{\Lambda})(x, x) = 0$ for every $x \in X$, which implies that $(\tilde{K} - K_{\Lambda})(x, y) = 0$ for every $x, y \in X$. Since $\tilde{K} = K_{\Lambda}$, we have
\begin{equation*}
	\|f\|_{\Lambda}^{2} 	
	= \|f_{1}\|_{\Lambda, S_{1}}^{2} + \|f_{2}\|_{\Lambda, S_{2}}^{2} - \|f_{S}\|_{\Lambda, S}^{2}.
\end{equation*}
This implies the separation property for the minimal separators of the form $S = X_{pq}$ and by inheritance, the separators which contains these minimal separators. To extend the result to all minimal separators, let $S$ be a minimal separator. Then it is contained inside a separator $\tilde{S}$ which separates $\Omega$ into two connected components $\tilde{Y}_{1}, \tilde{Y}_{2} \subset X$ which contains a minimal separator of the form $X_{pq}$ (see Figure \ref{fig:S-is-minimal-separator}). Let $\tilde{S}_{1} = \tilde{S} \cup Y_{1}$ and $\tilde{S}_{2} = \tilde{S} \cup Y_{2}$. We can write the separation property for $\tilde{S}$ in terms of the projections $\Pi_{\tilde{S_{1}}}$, $\Pi_{\tilde{S_{2}}}$ and $\Pi_{\tilde{S}}$ in $\sH_{\Lambda} = \sH(K_{\Lambda})$ as:
\begin{equation*}
	\bI = \Pi_{\tilde{S}_{1}} + \Pi_{\tilde{S}_{2}} - \Pi_{\tilde{S}}.
\end{equation*}
Let $\tilde{S} \downarrow S$ in the sense of sets. The conclusion now follows from the continuity of $K_{\Lambda}$ and the strong convergence of the projection $\Pi_{\tilde{S}}$ to the projection on to the intersection of the closed subspaces generated by $\{k_{x, \Lambda}: x \in \tilde{S}\}$ in $\sH_{\Lambda}$ where $k_{x, \Lambda}(y) = K_{\Lambda}(x, y)$ for $y \in X$(and likewise for $\Pi_{\tilde{S}_{1}}$, $\Pi_{\tilde{S}_{2}}$).
\end{proof}

We have thus shown the following:
\begin{theorem}
	Let $K_{\Omega}$ be a partially reproducing kernel on a regular domain. Then $K_{\Omega}$ admits a canonical completion $K_{\star}$. Furthermore, there exists an increasing sequence $\{\Omega_{j}\}_{j=1}^{\infty}$ of serrated domains with $\Omega_{j} \subset \Omega$ and $\cup_{j=1}^{\infty} \Omega_{j} = \Omega$ such that the canonical completions $K_{j}$ of $K_{\Omega_{j}} = K_{\Omega}|_{\Omega_{j}}$ with the associated norms $\|\cdot\|_{j}$ such that $K_{j}$ converges pointwise to $K_{\star}$ as $j \to \infty$ and $\glim_{j \to \infty} \|\cdot\|_{j} = \|\cdot\|_{\star}$, where $\|\cdot\|_{\star}$ is norm associated with $K_{\star}$.
\end{theorem}

Notice that for every serrated domain $\tilde{\Omega} \supset \Omega$, the canonical completion of $K_{\tilde{\Omega}} = K_{\star}|_{\tilde{\Omega}}$ is still $K_{\star}$ by inheritance. And as a result, the determinant maximization and trace zero properties hold for $K_{\star}$ as the completion of $K_{\tilde{\Omega}}$.

Although, we are unable to draw any conclusion about the uniqueness of canonical completion, using arguments similar to those in the proof of Lemma \ref{lem:lambda-pw}, we can show the following result.
\begin{theorem}
	Let $K_{\Omega}$ be a partially reproducing kernel on domain $\Omega$ such that there exists an increasing sequence of serrated domains $\{\Omega_{j}\}_{j=1}^{\infty}$ with $\Omega_{j} \subset \Omega$ and $\cup_{j=1}^{\infty} \Omega_{j} = \Omega$. Let $K_{j}$ denote the canonical completions of $K_{\Omega_{j}} = K_{\Omega}|_{\Omega_{j}}$ with the associated norms $\|\cdot\|_{j}$. 
	
	If (a) $\Omega$ is serrated or (b) $\Omega$ is regular and $K_{\Omega}$ admits a unique canonical completion, then $K_{j}$ converges pointwise to $K_{\star}$ as $j \to \infty$ and $\glim_{j \to \infty} \|\cdot\|_{j} = \|\cdot\|_{\star}$, where $K_{\star}$ denotes the canonical completion and $\|\cdot\|_{\star}$ denotes its associated norm.
\end{theorem}

\subsection{Canonical Semigroupoids}

The contraction maps of a canonical completion $K_{\star}$ of $K_{\Omega}$ are remarkable in that they mimic the structure of the underlying graph $\Omega$. Suppose that $A$, $S$ and $B$ be subsets of $X$. If $S$ separates $A$ and $B$, then every path from $A$ to $B$ can be decomposed into two paths: one from $A$ to $S$, followed by another from $S$ to $B$. Then the contraction $\Phi_{BA}: \sH_{A} \to \sH_{B}$ can be written as product of the contractions $\Phi_{SA}: \sH_{A} \to \sH_{S}$ and $\Phi_{BS}: \sH_{B} \to \sH_{S}$ as $\Phi_{BA} = \Phi_{BS}\Phi_{SA}$. Indeed, let $x \in A$ and $y \in B$. Naturally, $K_{\star}(x, y) = \langle \Phi_{BA} k_{x, A}, k_{y, B} \rangle$. The separation property gives us another way of writing $K_{\star}(x, y)$ which is
\begin{equation*}
	\langle k_{x, S}, k_{y, S} \rangle = \langle \Phi_{SA}k_{x, A}, \Phi_{SB}k_{y, B}\rangle = \langle \Phi_{BS}\Phi_{SA}k_{x, A}, k_{y, B}\rangle.
\end{equation*}
Thus $\langle \Phi_{BA} k_{x, A}, k_{y, B} \rangle = \langle \Phi_{BS}\Phi_{SA}k_{x, A}, k_{y, B}\rangle$. Since $x$ and $y$ can be chosen arbitrarily, it follows that $\Phi_{BA} = \Phi_{BS}\Phi_{SA}$. It is not difficult to see that the converse is also true. 

\begin{lemma}\label{thm:contract-semigroupoid}
	The contraction maps corresponding to a completion $K$ of $K_{\Omega}$ satisfy for the subsets $A$, $B$ and $S$ of $X$	
	$$\Phi_{BA} = \Phi_{BS}\Phi_{SA} \quad\mbox{ if } S \mbox{ separates } A \mbox{ and } B,$$
	if and only if $K$ is a canonical completion of $K_{\Omega}$.
\end{lemma}

Recall that the spaces $\sH_{A}$, $\sH_{B}$ and $\sH_{S}$ can be thought of as subspaces in $\sH$. The subspace $\sH_{S}$ can thus be said to ``separate'' $\sH_{A}$ and $\sH_{B}$ in $\sH$ in a way similar to how $S$ separates $A$ and $B$ in $\Omega$. 

Consider a regular domain $\Omega = \cup_{t \in T} (I_{t} \times I_{t})$. Define for $x,y \in T$ such that $x \geq y$, $\Phi_{xy}: \sH_{y} \to \sH_{x}$ as $\Phi_{xy}k_{z, I_{y}} = k_{z, I_{x}}$. Thus, $\Phi_{xy} = \Phi_{I_{x}I_{y}}$ and $\Phi_{xx} = \bI$. Because of the separation property, we have $\Phi_{xz} = \Phi_{xy}\Phi_{yz}$ for  $x \geq y \geq z$. Hence, we can write
\begin{equation*}
	\Phi_{wz} = \Phi_{wx}(\Phi_{xy}\Phi_{yz}) = (\Phi_{wx}\Phi_{xy})\Phi_{yz}
\end{equation*}
for $w \geq x \geq y \geq z$. Thus multiplication in $\{\Phi_{xy}\}_{x \geq y}$ is associative when defined. The maps $\{\Phi_{xy}\}_{x \geq y}$ form a group-like algebraic structure which is called a \emph{semigroupoid}. Much like a group, a semigroupoid consists of a set of elements along with an associative binary operation, but unlike a group, the operation need not be defined for all pairs of elements and moreover, there need not be an inverse. In our case, the operation is operator multiplication. 

We say that a set $\{\Phi_{xy}\}_{x \geq y}$ of contractions $\Phi_{xy}: \sH_{y} \to \sH_{x}$ is a \emph{canonical semigroupoid} of $K_{\Omega}$ if $\Phi_{xz} = \Phi_{xy}\Phi_{yz}$ for $x \geq y \geq z$ and $\Phi_{xy}k_{z,I_{y}} = k_{z,I_{x}}$ for $z \in I_{x} \cap I_{y}$. Notice that the second condition ensures that $\Phi_{xx} = \bI$. Moreover, for $x < y$, we can define $\Phi_{xy} = \Phi_{yx}^{\ast}$.
\begin{theorem}\label{thm:semigroupoid-completion}
	Let $K_{\Omega}$ be a partially reproducing kernel on a regular domain $\Omega$. Then there is a bijective correspondence between canonical completions $\tilde{K}$ and canonical semigroupoids $\{\Phi_{xy}\}_{x \geq y}$ of $K_{\Omega}$ given by
	\begin{equation}\label{eqn:semigroupoid}
		\tilde{K}(x, y) = \langle \Phi_{ts}k_{x, I_{s}}, k_{y, I_{t}} \rangle
	\end{equation}
	for $x \in I_{s}$ and $y \in I_{t}$, where $k_{u, I_{v}}(w) = K_{\Omega}(u, w)$ for $v \in T$ and $u, w \in I_{v}$.
\end{theorem}
\begin{proof}
	Given a canonical semigroupoid $\{\Phi_{xy}\}_{x \geq y}$ of $K_{\Omega}$ we can define $\tilde{K}$ using (\ref{eqn:semigroupoid}). Notice that $\tilde{K}(x, y)$ does not depend on the choice of $s, t \in T$ so long as $x \in I_{s}$ and $y \in I_{t}$. Indeed, for $x \in I_{s}, I_{s'}$ and $y \in I_{t}$ we have without loss of generality that $s < s' < t$ and 
	\begin{equation*}
		\langle \Phi_{ts}k_{x, I_{s}}, k_{y, I_{t}} \rangle 
		= \langle \Phi_{ts'}\Phi_{s's}k_{x, I_{s}}, k_{y, I_{t}} \rangle
		= \langle \Phi_{ts'}k_{x, I_{s'}}, k_{y, I_{t}} \rangle.
	\end{equation*}
	Thus $\tilde{K}$ is well-defined. 
	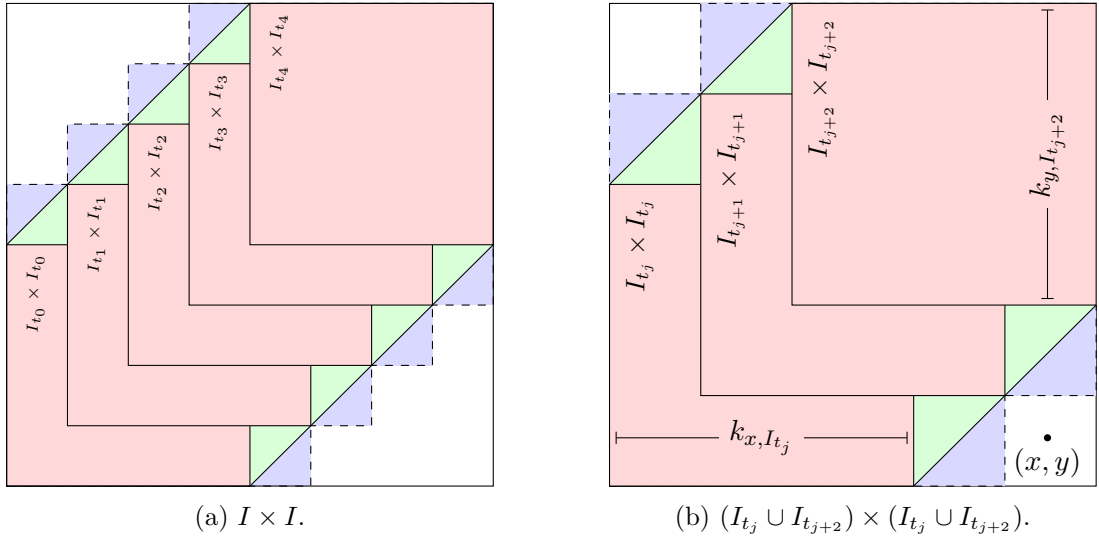
\begin{figure}[htbp]
		\centering
		\begin{subfigure}{0.49\textwidth}
		\centering
		\begin{tikzpicture}[scale=0.8]
			\pgfmathsetmacro{\a}{8}
			\pgfmathsetmacro{\b}{4}
			\pgfmathsetmacro{\c}{3}
				
			\foreach \i in {0,...,3} {
				\draw[fill=blue!15, dashed] (\i, \i) rectangle ++(\b+1, \b+1);
			} 
			\draw[fill=green!15] (0, 0) -- (\b, 0) -- (\a, \a-\b) -- (\a, \a) -- (\a-\b, \a) -- (0, \b) -- cycle;

			\foreach \i in {0,...,4} {
				\draw[fill=red!15] (\i, \i) rectangle ++(\b, \b);
				\node[rotate=90] at (\i+0.12*\b, \i+0.8*\b) {\tiny $I_{t_{\i}} \times I_{t_{\i}}$};
			} 
			
			\draw (0, 0) rectangle (\a, \a);
		\end{tikzpicture}
		\caption{$I \times I$.}
		\label{fig:I-I}
		\end{subfigure}
		\begin{subfigure}{0.49\textwidth}
		\centering
		\begin{tikzpicture}[scale=0.8]
			\pgfmathsetmacro{\a}{8}
			\pgfmathsetmacro{\b}{5}
			\pgfmathsetmacro{\c}{3}
			\pgfmathsetmacro{\hx}{0.5}
			\pgfmathsetmacro{\hy}{1}
		
			\draw (0, 0) rectangle (\a, \a);	
			\draw[dashed, fill=blue!15] (0, 0) rectangle (\a/2+\b/2, \a/2+\b/2);	
			\draw[dashed, fill=blue!15] (\a/2-\b/2, \a/2-\b/2) rectangle (\a, \a);		
		
			\draw[fill=green!15] (0, 0) -- (\b, 0) -- (\a, \a-\b) -- (\a, \a) -- (\a-\b, \a) -- (0, \b) -- cycle;
			\draw[fill=red!15] (0, 0) rectangle (\b, \b);
			\node[rotate=90] at (\hx, \b-\hy) {\small $I_{t_{j}} \times I_{t_{j}}$};
			\draw[fill=red!15] (\a/2-\b/2, \a/2-\b/2) rectangle ++(\b, \b);
			\node[rotate=90] at (\a/2-\b/2+\hx, \a/2+\b/2-1.4*\hy) {\small $I_{t_{j+1}} \times I_{t_{j+1}}$};
			\draw[fill=red!15] (\a-\b, \a-\b) rectangle (\a, \a);
			\node[rotate=90] at (\a-\b+\hx, \a-1.4*\hy) {\small $I_{t_{j+2}} \times I_{t_{j+2}}$};

			\pgfmathsetmacro{\x}{0.1*\a}
			\pgfmathsetmacro{\y}{0.9*\a}
			\pgfmathsetmacro{\h}{0.1}
			\draw [fill] (\y, \x) circle [radius=0.05] node [below] {$(x, y)$};
			\draw[|-|] (0+\h, \x) -- (\b-\h, \x) node [midway, fill = red!15, text opacity=1, opacity = 1] {$k_{x, I_{t_{j}}}$}; 
			\draw[|-|] (\y, \a-\b+\h) -- (\y,\a-\h) node [midway, fill = red!15, text opacity=1, opacity = 1, rotate=90] {$k_{y, I_{t_{j+2}}}$};
		\end{tikzpicture}
		\caption{$(I_{t_{j}} \cup I_{t_{j+2}}) \times (I_{t_{j}} \cup I_{t_{j+2}})$.}
		\label{fig:IuI-IuI}
		\end{subfigure}
		\caption{Semigroupoid characterization of Canonical Completion. The red and green regions represent $\Omega$ while the red, green and blue regions together is $\tilde{\Omega}$.}
		\label{fig:regular-domain1}
	\end{figure}
	To see why $\tilde{K}$ is a reproducing kernel, pick an increasing sequence $\{t_{j}\}_{j=1}^{n} \subset T$ such that $t_{j+1} \in I_{t_{j}}$ for $j \geq 1$ and $\cup_{j} I_{t_{j}} = I$. By Theorem \ref{thm:2sq-contraction}, the restriction of $\tilde{K}$ to $(I_{t_{j}} \cup I_{t_{j+1}}) \times (I_{t_{j}} \cup I_{t_{j+1}})$ is a reproducing kernel for $j \geq 1$, since the contraction $\Phi_{t_{j}t_{j+1}}$ satisfies the necessary conditions. Thus the restriction $\tilde{K}|_{\tilde{\Omega}}$ is a partially reproducing kernel on a serrated domain where $\tilde{\Omega} = \cup_{j} [(I_{t_{j}} \cup I_{t_{j+1}}) \times (I_{t_{j}} \cup I_{t_{j+1}})]$ (see Figure \ref{fig:I-I}). It is now becomes clear from Figure \ref{fig:IuI-IuI} that $\tilde{K}$ is merely the canonical completion of $\tilde{K}|_{\tilde{\Omega}}$. For every $x \in I_{t_{j}} \setminus I_{t_{j+1}}$ and $y \in I_{t_{j+2}} \setminus I_{t_{j+1}}$ are separated by $I_{t_{j+1}}$ in $\tilde{\Omega}$ and we have 
	\begin{align*}
		\tilde{K}(x, y) 
		&= \langle \Phi_{t_{j+2}t_{j}}k_{x, I_{t_{j}}}, k_{y, I_{t_{j+2}}} \rangle \\
		&= \langle \Phi_{t_{j+2}t_{j+1}}\Phi_{t_{j+1}t_{j}}k_{x, I_{t_{j}}}, k_{y, I_{t_{j+2}}} \rangle \\
		&= \langle \Phi_{t_{j+1}t_{j}}k_{x, I_{t_{j}}}, \Phi_{t_{j+1}t_{j+2}}k_{y, I_{t_{j+2}}} \rangle \\
		&= \langle \tilde{k}_{x, I_{t_{j+1}}}, \tilde{k}_{y, I_{t_{j+1}}} \rangle
	\end{align*}
	where 
	\begin{equation*}
		\tilde{k}_{u, I_{t_{j+1}}}(v) = \langle \Phi_{t_{j+1}t_{j}}k_{u, I_{t_{j}}}, k_{v, I_{t_{j+1}}} \rangle = \tilde{K}(u, v)
	\end{equation*}
	for $u \in I_{t_{j}}$ and $v \in I_{t_{j+1}}$ and we can argue similarly for $\tilde{k}_{y, I_{t_{j+1}}}$. One can now show that $\tilde{K}$ is a canonical completion by proving the separation propoerty for minimal separators $I_{t}$ for $t \in T$ using arguments similar to the ones above. The converse follows from Lemma \ref{thm:contract-semigroupoid}.
\end{proof}
Theorem \ref{thm:semigroupoid-completion} can be thought of as a generalization of Theorem \ref{thm:2sq-contraction}. It is a more algebraic view which allows us to see canonical completion as simply a consistent way of extending the generators $k_{x, I_{y}}$ to $I$. In Section \ref{sec:pdfunctions}, we shall see that this semigroupoid can actually be reduced to a nicer algebraic structure called \emph{semigroup} if $K_{\Omega}$ is stationary. The semigroupoids (semigroups) that we are dealing with are equipped with identities (identity) and strictly speaking, should thus be described as \emph{small categories} (\emph{monoids}). Regardless, we shall stick to our chosen terminology for the sake of simplicity.

\section{Canonical Extensions of Positive-Definite Functions}
\label{sec:pdfunctions}
In this section, we show that every continuous positive-definite function $F$ on an interval $[-a, a]$ ($a > 0$) admits a canonical extension to the entire real line corresponding to a canonical completion of the partial kernel $K_{\Omega}(x, y) = f(x - y)$ for $|x - y| \leq a$. Furthermore, this extension admits a representation in terms of a strongly continuous one-parameter semigroup on the RKHS of the kernel $K(x, y) = F(x - y)$ for $x,y \in [0, a]$. The canonical extension can be described in terms of the generator of this semigroup using many classical formulas such as the post-Widder inversion formula. 

In addition to proving the existence, we shall also establish the uniqueness of canonical extension under certain technical conditions, although experience with serrated domains suggests that this is generally the case.

\subsection{The Canonical Extension}

The extension problem for positive-definite functions can be framed in terms of completion of a partially reproducing kernel. Let $a > 0$ and $F$ be a positive-definition function on $[-a, a]$. Define $X = [0, b]$ for some $b > a$ and $\Omega = \{(x, y): |x - y| \leq a\} \subset X \times X$ and $K_{\Omega}: \Omega \to \bbR$ as $K_{\Omega}(x, y) = F(x - y)$. The stationary completions $\tilde{K}$ of $K_{\Omega}$ correspond precisely to the extensions $\tilde{F}$ of $F$. 

A celebrated result of \textcite{krein1940} states that every such function $F$ admits an extension $\tilde{F}$ to $\bbR$. We shall now prove a stronger statement by showing that there exists a canonical extension which correponds to a stationary canonical completion of $K_{\Omega}$. Let $K: [0, a] \times [0, a] \to \bbR$ be the reproducing kernel $K(x, y) = F(x - y)$ and $\sH = \sH(K)$. Let $I_{t} = [t, t+a] \cap X$ for $t \geq 0$ and notice that $\Omega$ is a regular domain and we can write $\Omega = \cup_{t \in T} (I_{t} \times I_{t})$ for $T = [0, b-a]$.

\begin{theorem}
	Let $F$ be a continuous positive-definite function on $(-a, a)$. Define $$K_{\Omega}(x, y) = F(x - y) \mbox{ for } (x, y) \in \Omega = \{(u, v): |u-v| < a\}.$$ Then $K_{\Omega}$ admits a stationary canonical completion.  
\end{theorem}
\begin{proof}
	Recall that $I_{t} = [t, t+a]$. Let $\Omega_{j} = \cup_{k \geq 0} (I_{k/2^{j}} \times I_{k/2^{j}})$ for $j \geq 1$.
	For $j \geq 1$, let $K_{j}$ denote the canonical completion of $K_{\Omega_{j}} = K_{\Omega}|_{\Omega_{j}}$. 
	By the arguments of Section \ref{sec:existence-regular-proof}, we can show that there exists a subsequence $\{K_{j_{k}}\}$ which converges pointwise to a canonical completion $\tilde{K}$ of $K_{\Omega}$. Note that for $1 \leq m \leq j_{k}$,
	\begin{equation*}
		\tilde{K}_{j_{k}}(x, y) = \tilde{K}_{j_{k}}(x+j/2^{m},y+j/2^{m})
	\end{equation*}
	for $x, y \in \bbR$ because of the construction (see Figure \ref{fig:canonical-completion-stationary}). It follows that $\tilde{K}(x, y) = \tilde{K}(x + j/2^{n}, y + j/{2^n})$ for every $j \in \bbZ$ and $n \geq 1$. By continuity, $\tilde{K}(x, y) = \tilde{K}(x+h,y+h)$ for $h \in \bbR$. So $\tilde{K}$ is indeed stationary. Hence proved.
\end{proof}

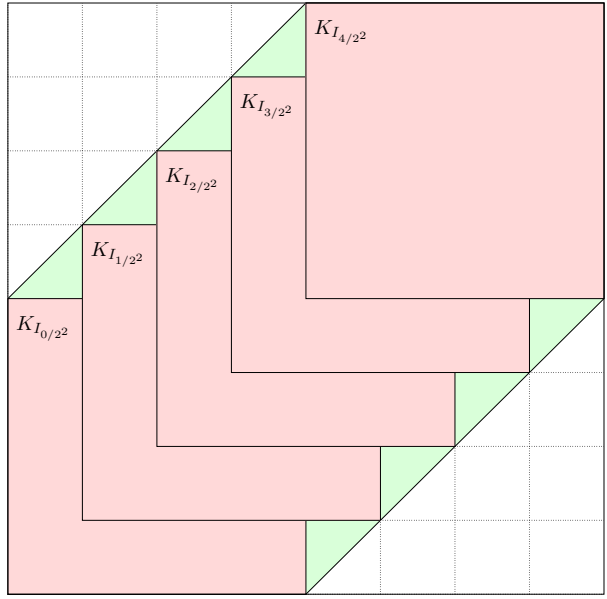
\begin{figure}[htbp]
	\centering
	\scalebox{.7}{
	\begin{tikzpicture}[scale=1.4]
		\pgfmathsetmacro{\a}{8}
		\pgfmathsetmacro{\b}{4}
		\pgfmathsetmacro{\c}{3}

		\foreach \i in {0,...,7} {
            \draw[densely dotted, line width = 0.1 mm] (\i, 0) -- (\i, \a);
            \draw[densely dotted, line width = 0.1 mm] (0, \i) -- (\a, \i);
        }
			
		\draw[fill=green!15] (0, 0) -- (\b, 0) -- (\a, \a-\b) -- (\a, \a) -- (\a-\b, \a) -- (0, \b) -- cycle;

		\foreach \i in {0,...,4} {
			\draw[fill=red!15] (\i, \i) rectangle ++(\b, \b);
			\node at (\i+0.12*\b, \i+0.9*\b) {$K_{I_{\i/2^{2}}}$};
		}
		\draw (0, 0) rectangle (\a, \a);
	\end{tikzpicture}}
	\caption{Stationary Canonical Completion. The red region represents $\Omega_{j}$ for $j = 2$, while the coloured region represents $\Omega$.}	
	\label{fig:canonical-completion-stationary}
\end{figure}

For a stationary canonical completion $\tilde{K}$ of $K$ we can define an extension $\tilde{F}: \bbR \to \bbR$ of $F$ as $\tilde{F}(x) = \tilde{K}(x, 0)$. We shall refer to $\tilde{F}$ as a \emph{canonical extension} of $F$. Conversely, $\tilde{F}$ is a canonical extension of $f$ if $\tilde{K}(x, y) = \tilde{F}(x - y)$ a canonical completion of $K_{\Omega}(x, y) = F(x-y)$. 

\begin{corollary}\label{thm:canonical-extn-exist}
	Every continuous positive-definite function $F:[-a, a] \to \bbR$ admits a canonical extension $\tilde{F}$ which satisfies
	\begin{equation*}
		\tilde{F}(x - y) = \langle\tilde{f}_{x}, \rho\tilde{f}_{y}\rangle
	\end{equation*}
	where $\tilde{f}_{x}:(0, a) \to \bbR$ is given by $\tilde{f}_{x}(y) = \tilde{f}(x + y)$ and $\rho:\sH \to \sH$ with $\rho g(y) = g(a - y)$.
\end{corollary}

\subsection{The Canonical Semigroup}

We shall now describe how the canonical semigroup can be constructed from the canonical semigroupoid corresponding to a canonical completion of $K_{\Omega}$. 

\subsubsection{Construction}

The canonical semigroup naturally arises when one attempts to describe the semigroupoid picture in terms of a single Hilbert space.
Notice that the kernels $K_{I_{t}} = K_{\Omega}|_{I_{t} \times I_{t}}$ are essentially identical up to a translation in that $K_{I_{s}}(s+u, s+v) = K_{I_{t}}(t+u, t+v)$ for every $s,t \in \bbR$ and every $u,v \in I_{0}$. Let $\sH_{t} = \sH(K_{I_{t}})$ for $t \in T$. 

By Theorem \ref{thm:contract-semigroupoid}, there exists a semigroupoid $\{\Phi_{ts}: t,s \in T \mbox{ and } t \geq s\}$ of contractions maps $\Phi: \sH_{s} \to \sH_{t}$ satisfying $\Phi_{ts}k_{u, I_{s}} = k_{u, I_{t}}$ for $t \geq s$. Moreover, define $\tilde{k}_{u, I_{v}}: I_{v} \to \bbR$ $\tilde{k}_{u, I_{v}}(w) = \tilde{K}(u, w)$ for $w \in I_{v}$ and note that $\Phi_{st}\tilde{k}_{u, I_{s}} = \tilde{k}_{u, I_{t}}$. Define $T_{s}: \sH \to \sH_{s}$ as $T_{s}g(u) = g(u - s)$. It follows that the adjoint of $T_{s}$ is given by $T_{s}^{\ast} = T_{-s}^{}$. Notice that $T_{s}k_{u, I_{0}} = k_{u+s, I_{s}}$ for $s, t \in \bbR$. We shall now reduce the canonical semigroupoid to a one-parameter semigroup.

\begin{theorem}[Canonical Semigroup]
	The operators $T_{s}^{\ast}\Phi_{st}^{}T_{t}^{}: \sH \to \sH$ for $s \geq t$ depend only on the difference $(s - t)$, that is
	\begin{equation*}
		T_{s}^{\ast}\Phi_{st}^{}T_{t}^{} = T_{s+h}^{\ast}\Phi_{s+h,t+h}^{}T_{t+h}^{} \quad \mbox{ for } h \in \bbR.
	\end{equation*}
	Define $\Phi_{s}^{} = T_{s}^{\ast}\Phi_{s0}^{}$ for $s \geq 0$. Then $\{\Phi_{t}\}_{t \geq 0}$ forms a strongly continuous semigroup on $\sH$:
	\begin{enumerate}
		\item $\Phi_{0}^{} = \bI$,
		\item $\Phi_{s}\Phi_{t} = \Phi_{s+t}$ for $s, t > 0$, and
		\item $\lim_{h \to 0^+} \|\Phi_{h}g - g\| = 0$ for every $g \in \sH$.
	\end{enumerate}
\end{theorem}
\begin{proof}
	Let $s, t \in \bbR$. Notice that for $u \in I_{0}$, we can write 
	\begin{align*}
		T_{s+h}^{\ast}\Phi_{s+h,t+h}^{}T_{t+h}^{}k_{u, I_{0}}
		&= T_{s+h}^{\ast}\Phi_{s+h,t+h}^{}\tilde{k}_{u + t + h, I_{t+h}}\\
		&= T_{s+h}^{\ast}k_{u + t + h, I_{s+h}}\\
		&= \tilde{k}_{u + t + h - (s + h) , I_{0}},
	\end{align*}
	while $T_{s}^{\ast}\Phi_{s,t}^{}T_{t}^{}k_{u, I_{0}} = T_{s}^{\ast}\Phi_{s,t}^{}k_{u + t, I_{t}} = T_{s}^{\ast}\tilde{k}_{u + t, I_{s}} = \tilde{k}_{u + t - s , I_{0}}$ implying $T_{s+h}^{\ast}\Phi_{s+h,t+h}^{}T_{t+h}^{}k_{u, I_{0}} = T_{s}^{\ast}\Phi_{s,t}^{}T_{t}^{}k_{u, I_{0}} $ for $u \in I_{0}$ and therefore, $T_{s+h}^{\ast}\Phi_{s+h,t+h}^{}T_{t+h}^{} = T_{s}^{\ast}\Phi_{s,t}^{}T_{t}^{}$. It is obvious that $\Phi_{0}^{} = \bI$ and we can argue as before that $\Phi_{s}\Phi_{t} = \Phi_{s+t}$ because
	\begin{align*}
		\Phi_{s}\Phi_{t}k_{u, I_{0}} 
		&= T_{s}^{\ast}\Phi_{s0}^{}T_{t}^{\ast}\Phi_{t0}^{}k_{u, I_{0}} \\
		&= T_{s}^{\ast}\Phi_{s0}^{}T_{t}^{\ast}\tilde{k}_{u, I_{t}} \\
		&= T_{s}^{\ast}\Phi_{s0}^{}\tilde{k}_{u - t, I_{0}} \\
		&= \tilde{k}_{u - t - s, I_{0}} \\
		&= \Phi_{s+t}k_{u, I_{0}}
	\end{align*}
	for $u \in I_{0}$. We need only verify that $\lim_{h \to 0^{+}} \Phi_{h}g = g$ for $g \in \sH$. Let $g = \sum_{i=1}^{n} \alpha_{i}k_{u_{i}, I_{0}}$ for some $n \geq 1$, $\{\alpha_{i}\}_{i=1}^{n} \subset \bbR$ and $\{u_{i}\}_{i=1}^{n} \subset I_{0}$. Then
	\begin{equation*}
		\| \Phi_{h} g - g \|^{2} =  \sum_{i, j=1}^{n} \alpha_{i}\alpha_{j} \left[ \tilde{K}(u_{i} + h, u_{j} + h) - \tilde{K}(u_{i}, u_{j} + h) - \tilde{K}(u_{i} + h, u_{j}) + \tilde{K}(u_{i}, u_{j}) \right] \to 0
	\end{equation*}
	as $h \to 0^{+}$. Because $\Span\{k_{u, I_{0}}: u \in I_{0}\}$ is dense in $\sH$ and $\|\Phi_{h} \| \leq 1$, it follows that $\Phi_{h}$ converges strongly to $\bI$ as $h \to 0^{+}$ (see \textcite[Lemma 9.4.7]{Eidelman2004}) and we are done. Alternatively, we could have used the equivalence of strong and weak continuity (see \textcite[Theorem 5.8]{Engel2000}).
\end{proof}

The following corollary is now immediate from Theorem \ref{thm:semigroupoid-completion}.
\begin{theorem}
	Let $a > 0$. There is a bijective correspondence between the canonical extensions $\tilde{F}$ of a continuous positive-definite function $F$ on $[-a, a]$ and strongly continuous semigroups $\{\Phi_{t}\}_{t \geq 0}$ of contractions on $\sH = \sH(K)$ where $K(x, y) = F(x - y)$ for $x, y \in [0, a]$ satisfying 
	\begin{equation*}
		\Phi_{t}k_{u} = k_{u-t} \quad \mbox{ for } 0 \leq t \leq u \leq a
	\end{equation*}
	given by
	\begin{equation}\label{eqn:nagy-repn}
		\tilde{F}(t) = \begin{cases}
			\langle k_{0}, \Phi_{t}k_{0} \rangle &\mbox{ for } t \geq 0 \\
			\langle k_{0}, \Phi_{-t}k_{0} \rangle &\mbox{ for } t < 0
		\end{cases}
	\end{equation}
	where $k_{0}(u) = F(u)$ for $0 \leq u \leq a$.
\end{theorem}
Although, all positive-definite functions admit a unitary representation resembling (\ref{eqn:nagy-repn}), canonical extensions $F_{\star}$ admit a very concrete representation of that kind. As a consequence of Theorem \ref{thm:canonical-extn-exist}, we have
\begin{corollary}
	Let $a > 0$. Every continuous positive-definite function $F$ on $[-a, a]$ admits a canonical extension $\tilde{F}$ with the representation \emph{(\ref{eqn:nagy-repn})} for some strongly continuous semigroup of contractions $\{\Phi_{t}\}_{t \geq 0}$ on $\sH(K)$ where $K(x, y) = F(x - y)$ for $x, y \in [0, a]$.
\end{corollary} 

\subsubsection{Generators of Canonical Semigroups}
The canonical semigroup $\{\Phi_{t}\}_{t \geq 0}$, like all strongly continuous one-parameter semigroups, admits a generator which is defined as the linear operator $\partial_{\star}: \sD(\partial_{\star}) \to \sH$ given by
\begin{equation*}
	\partial_{\star}f = \lim_{h\to0^{+}} \frac{1}{h} \left[ \Phi_{h}f - f \right]
\end{equation*}
for $f \in \sD(\partial_{\star}) = \{ f \in \sH: \lim_{h \to 0^{+}} \frac{1}{h} \left[ \Phi_{h}f - f \right] \mbox{ exists}\}.$ In general, the operator $\partial_{\star}$ is not bounded and its domain $\sD(\partial_{\star})$ is not equal to $\sH$. However, according to the Hille-Yosida Theorem for contraction semigroups (see Theorem 3.5 of \textcite{Engel2000}), the operator $\partial_{\star}$ is closed, its domain $\sD(\partial_{\star})$ is dense in $\sH$, and the operator $\lambda \bI - \partial_{\star}$ has a bounded inverse satisfying $\|(\lambda \bI - \partial_{\star})^{-1}\| \leq 1/\lambda$ for $\lambda > 0$ . Most importantly, the generator $\partial_{\star}$ uniquely determines the semigroup $\{\Phi_{t}\}_{t \geq 0}$.

The connection with semigroups furnishes many interesting representations for a canonical extension $F_{\star}$ in terms of the generator $\partial_{\star}$.
\begin{theorem}\label{thm:extn-generator-repn}
Let $F_{\star}$ be a canonical extension of a continuous positive-definite function $F$ on $[-a, a] \subset \bbR$ for some $a > 0$ and let $\partial_{\star}$ be the generator of the corresponding canonical semigroup. For $x > 0$, we have
\begin{align}
	F_{\star}(t) &= \lim_{\mu \to 0} \left\langle e^{t\partial_{\star}[\bI - \mu\partial_{\star}]^{-1}}k_{0}, k_{0}\right\rangle 
	\tag{Yosida Approximation Formula}\\
	F_{\star}(t) &= \lim_{n\to\infty}  \left\langle \left[\bI - \tfrac{t}{n} \partial_{\star}\right]^{-n} k_{0}, k_{0}\right\rangle 
	\tag{Post-Widder Inversion Formula}\\
	F_{\star}(t) &= \frac{1}{2 \pi \mathrm{i}} \lim_{n \to \infty} \int_{\epsilon-\mathrm{i}n}^{\epsilon + \mathrm{i}n} \langle e^{zt}[z\bI - \partial_{\star}]^{-1}k_{0}, k_{0}\rangle ~dz \tag{Cauchy Integral Formula}
\end{align}
where $\mathrm{i} = \sqrt{-1}$ and the convergence is uniform over compact intervals of $\bbR_{+}$. The integral is the last equation is to be understood as the usual contour integral from complex analysis.
\end{theorem}
\begin{proof}
	These can be readily seen as straightforward consequences of Theorem 3.5, Corollary 5.5 and Theorem 5.14 from \textcite{Engel2000} in our setting.
\end{proof}

Of course, the beautiful expressions in Theorem \ref{thm:extn-generator-repn} don't mean much to us if we can't calculate $\partial_{\star}$ independently of $F_{\star}$. In the following section, we consider a plausible situation in which the canonical extension can be recovered as the closure $\bar{\partial}$ of an explicitly defined operator $\partial$. In addition to bringing the expressions in Theorem \ref{thm:extn-generator-repn} to life, this proves that under the considered scenario the canonical extension is unique.

\subsection{Differential Equations in Hilbert Space}
The semigroup connection also allows us to think of the canonical extension as the solution of an abstract Cauchy problem in $\sH = \sH(K)$. Consider a function $f: \bbR_{+} \to \sH$. We shall denote the value of $f$ at $t \in \bbR_{+}$ by $f_{t}$. We say that $f$ is \emph{Fr\'{e}chet differentiable} at $t\in \bbR_{+}$ if there exists $\partial_{t}f_{t} \in \sH$ such that $\lim_{h \to 0} \| \tfrac{1}{h}(u_{t+h} - u_{t}) - \partial_{t}u_{t}\| = 0$. According to Proposition 6.2 of \textcite{Engel2000}, if $k_{0} \in \sD(\partial_{\star})$, then $f_{t} = \Phi_{t}k_{0} (= \tilde{k}_{0, I_{t}})$ is the unique solution of the abstract Cauchy problem: for $t \geq 0$
\begin{align}\label{eqn:cauchy-problem}
	\left\{ \begin{aligned}
		\partial_{t}f_{t} &= \partial_{\star}f_{t},\\
		f_{0} &= k_{0}.
	\end{aligned} \right.
\end{align}
When $k_{0} \notin \sD(\partial_{\star})$, then Proposition 6.4 of \textcite{Engel2000} tells us that the function $f: t \mapsto \Phi_{t}k_{0}$ can still be understood as the unique \emph{mild solution} to (\ref{eqn:cauchy-problem}) in the sense that for $t \geq 0$, $\int_{0}^{t} f_{u} ~du \in \sD(\partial_{\star})$ and
\begin{equation}\label{eqn:cauchy-problem-mild}
	f_{t} = k_{0} + \partial_{\star} \int_{0}^{t} f_{u} ~du.
\end{equation}
In essence, the problem of canonical positive-definite extension is equivalent to solving an abstract differential equation in a certain Hilbert space.

\subsubsection{Recovery of the Generator}

Computing the operator $\partial_{\star}$ explicitly or even identifying its domain $\sD(\partial_{\star})$ precisely is usually very difficult even when $\{\Phi_{t}\}_{t \geq 0}$ is known. Ours is a more complicated situation since we only know certain images of $\{\Phi_{t}\}_{t \geq 0}$ as given by
\begin{equation*}
	\Phi_{s} k_{t} = k_{t-s} \mbox{ for } 0 \leq s \leq t \leq a.
\end{equation*}

Fortunately, it is often possible to evaluate $\partial_{\star}$ over a subset $\sD \subset \sD(\partial_{\star})$. Let $\sD$ denote the set of integrals $\int_{0}^{a} \alpha(u)k_{u} ~du \in \sH$ where $\alpha$ is an infinitely differentiable real-valued function on $(0, a)$ with compact support. Note that $\sD$ is a dense linear subspace of $\sH$. The elements of  $\sD$ serve essentially the same purpose as that of test functions in the theory of distributions. By definition, 
\begin{eqnarray*}
	\partial_{\star}\left[\int_{0}^{a} \alpha(u)k_{u} ~du\right] 
	&=& \lim_{h\to0^{+}} \frac{1}{h} \left[ \int_{0}^{a} \alpha(u)k_{u-h} ~du - \int_{0}^{a} \alpha(u)k_{u} ~du \right] \\
	&=& \lim_{h\to0^{+}} \frac{1}{h} \left[ \int_{0}^{a} \alpha(u+h)k_{u} ~du - \int_{0}^{a} \alpha(u)k_{u} ~du \right] \\
	&=&  \int_{0}^{a} \left[ \lim_{h\to0^{+}} \frac{\alpha(u+h) - \alpha(u)}{h} \right]k_{u} ~du \\
	&=&  \int_{0}^{a} \alpha'(u)k_{u} ~du.
\end{eqnarray*}
Thus $\partial_{\star}\left[ \int_{0}^{a} \alpha(u)k_{u} ~du \right] = \int_{0}^{a} \alpha'(u)k_{u} ~du \in \sD$ and $\sD$ is invariant under $\partial_{\star}$.  In the same way, one can also work out $\partial_{\star}\left[ \int_{0}^{a} \alpha(u)k_{u} ~du \right]$ for piecewise once-differentiable $\alpha$ with compact support as in $\partial_{\star} \int_{s}^{t} k_{u} ~du = k_{s}-k_{t}$ for $0 < s < t < a$, although expanding $\sD$ to include such elements is probably not of much consequence.

Define the operator $\partial: \sD \to \sH$ as 
\begin{equation*}
	\textstyle \partial\left[ \int_{0}^{a} \alpha(u)k_{u} ~du \right] = \int_{0}^{a} \alpha'(u)k_{u} ~du.
\end{equation*}
Thus $\partial = \partial_{\star}|_{\sD}$ but here we have defined it exclusively in terms of $K$ without referring to $\partial_{\star}$ or $F_{\star}$. We would like to recover $\partial_{\star}$ from its restriction $\partial$ to a dense subset $\sD \subset \sH$. If $\partial_{\star}$ is continuous on $\sH$, then this is possible using an ordinary extension by continuity argument. However, $F$ is analytic if $\partial_{\star}$ is continuous (Remark \ref{rmk:bounded-gen}). Because analyticity implies unique extension anyway, the special case of bounded generators is unintersting in that it does not offer us any insight into the problem of canonical extension. 

In general, it is not possible to recover an unbounded operator from its restriction to a dense subspace. Fortunately, $\partial_{\star}$ is a closed operator, which means that the graph $\sG_{\star} = \{(f, \partial_{\star}f): f \in \sD(\partial_{\star})\}$ is a closed subset of $\sH \times \sH$. This makes a different kind of extension by continuity possible. If the closure $\bar{\sG}$ in $\sH \times \sH$ of a graph $\sG = \{(f, \partial f): f \in \sD\}$ for some operator $\partial$, is equal to $\sG_{\star}$ then that would mean that we can recover $\partial_{\star}$ as the \emph{closure} $\bar{\partial}$ of $\partial$, given by:
\begin{equation*}
	\bar{\partial}f = \lim_{j \to \infty} \partial f_{j}
\end{equation*}
where $\{f_{j}\}_{j=1}^{\infty} \subset \sD$ such that $\lim_{j\to \infty} f_{j} = f \in \sD(\partial_{\star})$ and $\lim_{j\to \infty}  \partial f_{j}$ exists. 
An alternative way of stating this is to say that $\sD$ is a \emph{core} of $\partial_{\star}$, which is to say that $\sD$ is dense in $\sD(\partial_{\star})$ with respect to the norm $\|f\|_{\partial} = \|f\| + \|\partial f\|$ where $f \in \sD$. We now present certain criteria for $\sD$ to be a core of $\partial_{\star}$. 
\begin{theorem}
Suppose that $e^{-\lambda x} \notin \sH$ for some $\lambda > 0$. Then $(\lambda \bI - \partial)\sD$ is dense in $\sH$ and 
	\begin{enumerate}
		\item $F$ admits a unique canonical extension $F_{\star}$,
		\item $\partial_{\star} = \bar{\partial}$ and $\sD$ is a core of $\sD(\partial_{\star})$,
	\end{enumerate}	
\end{theorem}
\begin{proof}
	Let $f$ be in the orthogonal complement of $(\lambda \bI - \partial)\sD$. Then
	\begin{equation*}
		\int_{0}^{a} [\lambda\alpha(u) - \alpha'(u)] f(u) ~du \rangle = \langle f, \int_{0}^{a} [\lambda\alpha(u) - \alpha'(u)] k_{u} ~du = 0
	\end{equation*}
	for every infinitely differentiable $\alpha$ with a compact support in $(0, a)$. Using some elementary distribution theory, it can be shown that this can only be true if $\lambda f + f' = 0$ or $f(u) = ce^{-\lambda u}$ for $0 \leq x < a$ for some $c \in \bbR$. Since, $e^{-\lambda x} \notin \sH$, it follows that $c = 0$ and $f = 0$, thus implying that $(\lambda \bI - \partial)\sD$ is dense in $\sH$. The conclusion now follows from Theorem 5.2 of \textcite{Engel2000}.
\end{proof}

The condition $e^{-\lambda x} \notin \sH$ is equivalent to saying that $K_{\lambda}(x, y) = e^{\lambda(x+y)}F(x-y) - c$ is not a reproducing kernel for any $c > 0$. It is unclear how stringent this requirement is, but in light of the consequence, we need only worry about the case of positive-definite functions $F$ for which $e^{-\lambda x} \in \sH$ for every $\lambda > 0$, which remains unsolved.

\begin{theorem}
If for every infinitely differentiable $\alpha: (0, a) \to \bbR$ with a compact support there exists $r > 0$ such that
\begin{equation}\label{eqn:davies-criterion}
	\textstyle \sum_{j=0}^{\infty} \frac{r^{j}}{j!}\sqrt{\int_{0}^{a}\int_{0}^{a}D^{j}\alpha(u)D^{j}\alpha(v)F(u-v) ~dudv} < \infty.	
\end{equation}
then, $F$ admits a unique canonical extension $F_{\star}$, $\partial_{\star} = \bar{\partial}$ and $\sD$ is a core of $\sD(\partial_{\star})$.
\end{theorem}
\begin{proof}
	This follows from Theorem 1.51 of \textcite{Davies1980} by noticing that the expression (\ref{eqn:davies-criterion}) is equivalent to 
	$\sum_{j=0}^{\infty} \frac{r^{j}}{j!}\|\partial^{j}f\| < \infty$ for $f = \int_{0}^{a} \alpha(u)k_{u} ~du$.
\end{proof}

Both conditions $e^{-\lambda x} \notin \sH$ and (\ref{eqn:davies-criterion}) are very difficult to verify in general and as a result, we are unable to construct examples of positive-definite function $F$ for which they apply. We conclude this section by pointing out the stringency of assuming that the generator $\partial_{\star}$ is bounded.
\begin{remark}\label{rmk:bounded-gen}
	The generator $\partial_{\star}$ is bounded precisely when the semigroup is uniformly continuous, that is $\lim_{h \to 0^{+}} \| \Phi_{t} - \bI \| = 0$. In this case, $\Phi_{t} = \exp (t\partial_{\star})$ and we can write
	\begin{equation*}
		\tilde{F}(t) = \langle k_{0}, \Phi_{t}k_{0} \rangle 
		= \langle k_{0}, \exp (t\partial_{\star}) k_{0} \rangle
		= \sum_{j = 0}^{\infty} \langle k_{0}, \partial_{\star}^{j}k_{0} \rangle \frac{t^{j}}{j!} 
	\end{equation*}
	which implies that $F_{\star}$ and hence, $F$ is analytic! The series has an infinite radius of convergence, which makes $\tilde{F}$ an entire function. 
\end{remark}

\section{Appendix}

\begin{proof}[Proof of Lemma \ref{thm:pm-adjoint}]
	Let $H : X \times X \to \bbR$ be such that $K + H, K - H \geq O$. Thus, $K + H, K - H \geq O$ which implies that $H(x,y) = H(y,x)$ for $x,y \in X$ and
	\begin{equation*}
		- \sum_{i,j = 1}^{n} \alpha_{i}\alpha_{j} K(x_{i}, x_{j}) \leq \sum_{i,j = 1}^{n} \alpha_{i}\alpha_{j} H(x_{i}, x_{j}) \leq \sum_{i,j = 1}^{n} \alpha_{i}\alpha_{j} K(x_{i}, x_{j}).
	\end{equation*} 
	for $\{\alpha_{i}\}_{i=1}^{n} \subset \bbR$ and $\{x_{i}\}_{i=1}^{n} \subset X$. Let $\sH$ be the RKHS of $K$ and $\sH_{0} = \Span \{ k_{x} : x \in X \} \sH$ where $k_{x}: X \to \bbR$ is defined by $k_{x}(y) = K(x, y)$ for $x, y \in X$. Let $B: \sH_{0} \times \sH_{0} \to \bbR$ be the symmetric bilinear linear functional given by $B(k_{x}, k_{y}) = H(x,y)$. $B$ is well-defined because of the above equation. Moreover, $| B(f,f) | \leq \| f \|^{2}$ for every $f \in \sH_{0}$.
	So, \begin{equation*}
		\|f \|^{2} + B(f,f), \|f \|^{2} - B(f,f) \geq 0
	\end{equation*}
	Notice that $\|f - g\|^{2} + B(f-g,f-g) \geq 0$ implies that
	\begin{equation*}
		\begin{split}
			B(f,g) &\leq \tfrac{1}{2}\left[ \|f - g\|^{2} + B(f,f) + B(g,g) \right] \\
			&\leq \|f \|^{2} + \|g \|^{2} - \langle f, g \rangle \\
		\end{split}
	\end{equation*}
	Replacing $f$ by $\sqrt{c} f$ and $g$ by $g / \sqrt{c}$ for some $c > 0$ gives
	\begin{equation*}
		\begin{split}
			B(f,g) &\leq c\|f \|^{2} + \|g \|^{2}/c - \langle f, g \rangle \\
			&\leq 2 \|f \| \| g \| +  \|f \| \| g \| = 3 \|f \| \| g \| 
		\end{split}
	\end{equation*}
	by choosing $c = \|g \| / \|f \|$ and applying the Cauchy-Schwarz inequality. By replacing $g$ by $-g$, we can derive $B(f,g) \geq -3 \|f \| \| g \| $. It follows that $ |B(f,g)| \leq 3 \| f \| \|g \|$ and therefore $B$ is continuous. It uniquely extends by continuity to $\sH \times \sH$ and admits a Riesz representation (\textcite[Theorem 3.8-4]{Kreyszig1978}) of the form $B(f,g) = \langle \Phi_{H}f, g\rangle$, where $\Phi_{H} \in \cL(\sH)$. Moreover, $\Phi_{H}$ is self-adjoint since $$\langle \Phi_{H}k_{x}, k_{y} \rangle = H(x,y) = H(y,x) = \langle \Phi_{H}k_{y}, k_{x} \rangle$$ for $x,y \in X$. By Proposition 2.13 of \textcite{Conway2019}, it follows that $$\| \Phi_{H} \|^{2} = \sup_{f \in \sH_{0} \setminus \{ 0\}} \frac{|\langle \Phi_{H}f,f \rangle|}{\|f\|^{2}} \leq 1$$ because $|\langle \Phi_{H}f,f \rangle| = |B(f,f)| \leq \|f\|^{2}$ for $f \in \sH_{0}$. Thus, $\Phi_{H}$ is a self-adjoint contraction. 
\end{proof}
\printbibliography[]
\end{document}